\documentclass[11pt]{amsart}
\usepackage{amssymb,amsmath,color}
\usepackage{tikz-dependency}
\usepackage{tikz}
\usepackage{enumerate}

\newtheorem{theorem}{Theorem}[section]
\newtheorem{proposition}[theorem]{Proposition}
\newtheorem{lemma}[theorem]{Lemma}
\newtheorem{definition}[theorem]{Definition}

\newtheorem{corollary}[theorem]{Corollary}

\newtheorem{remark}[theorem]{Remark}

\numberwithin{equation}{section}

\newcommand{\Rrr}{{\mathbb R}}

\newcommand{\Ddd}{{\widetilde{D}}}

\newcommand{\Simplex}{\Delta}

\newcommand{\rest}[2] {\left.#1\right| {_{#2}}} 

\DeclareMathOperator{\conv}{conv}

\newcommand{\vanish}[1]{}

\begin{document}

\title[Classification of uniform flag triangulations of the root polytope]
{Classification of uniform flag triangulations of the boundary
of the full root polytope of type $A$}

\author{Richard EHRENBORG, G\'abor HETYEI and Margaret READDY}

\address{Department of Mathematics, University of Kentucky, Lexington,
  KY 40506-0027.\hfill\break \tt http://www.math.uky.edu/\~{}jrge/,
  richard.ehrenborg@uky.edu.}

\address{Department of Mathematics and Statistics,
  UNC-Charlotte, Charlotte NC 28223-0001.\hfill\break
\tt http://www.math.uncc.edu/\~{}ghetyei/,
ghetyei@uncc.edu.}

\address{Department of Mathematics, University of Kentucky, Lexington,
  KY 40506-0027.\hfill\break
\tt http://www.math.uky.edu/\~{}readdy/,
margaret.readdy@uky.edu.}



\date{\today}  

\begin{abstract}
The full root polytope of type~$A$ is the convex hull of all pairwise differences of
the standard basis vectors which we represent by forward and backward
arrows. 
We completely classify all flag triangulations of this polytope
that are uniform in the sense that the edges may be described as a
function of the relative order of the indices of the four basis vectors
involved. These fifteen triangulations fall naturally into three classes:
three in the lex class, three in the revlex class and nine in the Simion class.
We also consider a refined face count where we distinguish between
forward and backward arrows. We prove the refined face
counts only depend on the class of the triangulations.
The refined face generating functions are expressed in
terms of
the Catalan and Delannoy generating functions
and the modified Bessel function of the first kind.
\end{abstract}

\maketitle

\section{Introduction}

Triangulations of root polytopes and of products of simplices have been a
subject of intense study in recent
years~\cite{Ardila_Beck_Hosten_Pfeifle_Seashore,Ceballos_Padrol_Sarmiento_old,
  Ceballos_Padrol_Sarmiento, Cellini_Marietti_abelian,
  Gelfand_Graev_Postnikov}. 
Motivated by an observation made
in~\cite{Cori_Hetyei}, we recently~\cite{Ehrenborg_Hetyei_Readdy}
established that the Simion type~$B$ associahedron~\cite{Simion} may be
realized as a pulling triangulation of the full root polytope of type~$A$, defined as
the convex hull of all differences of pairs of standard basis vectors in
Euclidean space. These vertices can be thought of as {\em arrows} between
numbered nodes. We also showed that all pulling triangulations are flag.
The full root polytope is the centrally symmetric variant of the
type~$A$ positive root polytope whose lexicographic and revlex
triangulations were studied by Gelfand, Graev and
Postnikov~\cite{Gelfand_Graev_Postnikov}.
A question naturally arises:
Are there other reasonably uniform triangulations of
the root polytope?

In this paper we fully answer this question. We classify all flag
triangulations that are uniform in the sense that the flag condition
depends only on the relative order on the numbering of the basis vectors
involved.
The key tool we use is a
characterization of triangulations of a product of simplices given by
Oh and Yoo~\cite{Oh_Yoo1,Oh_Yoo_2} in terms of {\em matching
  ensembles}. We determine that there are three classes of triangulations:
variants of the lexicographic pulling triangulation, variants of the
revlex pulling triangulation, and variants of the triangulation
representing the Simion type~$B$ associahedron. 

It is known that
all triangulations of the boundary of the root polytope
have the same face numbers. For pulling triangulations, this
was shown in~\cite{Hetyei_Legendre}.
For general triangulations it may be shown by
generalizing the argument in~\cite{Hetyei_Legendre} and using the fact that all
triangulations of a product of simplices have the same face
numbers~\cite[Lemma 2.3]{Oh_Yoo_2}.

To distinguish between the three
major classes of uniform flag triangulations,
we introduce a refined face count which keeps track of
the number of forward and backward arrows in each face.
Surprisingly we
find that the refined face count of a 
triangulation only depends on which class it belongs to,
regardless of how we fix the number of
forward and backward arrows.
For the Simion class
the generating function for Catalan numbers and
weighted generalizations of the
Delannoy numbers play a crucial role in the refined face count.
For the revlex class
different weighted generalizations of the
Delannoy numbers 
and 
modified Bessel functions of the first kind play the important role.
Finally for the lex class, the Catalan numbers play the essential role.

Our paper is structured as follows. In
Section~\ref{section_classification} we state our main classification
theorem and prove the sufficiency part. The necessity part is in
Section~\ref{section_sufficiency}.
In Section~\ref{section_fifteen_distinct_triangulations}
we outline the argument that all triangulations that we consider are
actually pairwise non-isomorphic.
Results facilitating the refined face count
in all cases are presented in Section~\ref{section_tools}. The actual
refined face count appears in
Sections~\ref{section_Simion}, \ref{section_revlex}
and~\ref{section_enumeration_nonest}, respectively.
We end the paper with concluding remarks.

\section{Preliminaries}
\label{section_preliminaries}

\subsection{The root polytope}

The {\em full root polytope}~$P_{n}$
is the convex hull of the $n(n+1)$ vertices $e_{j}-e_{i}$ where $i \neq j$
and
$\{e_{1}, e_{2}, \ldots, e_{n+1}\}$
is the standard orthonormal basis of the Euclidean space $\mathbb{R}^{n+1}$.
This polytope was first studied by Cho~\cite{Cho}, and it is called the
``full'' type~$A$ root polytope in the work of
Ardila, Beck, Ho\c{s}ten, Pfeifle and
Seashore~\cite{Ardila_Beck_Hosten_Pfeifle_Seashore}
in order to distinguish it from
the {\em positive root polytope}~$P^{+}_{n}$
which is the convex hull of the origin $0$ and
the positive roots $e_{j} - e_{i}$ where
$1 \leq i \leq j \leq n$.
The full root polytope is also known as
the Legendre polytope~\cite{Hetyei_Legendre}.
This name is motivated by the fact that the polynomial
$\sum_{j=0}^{n} f_{j-1} \cdot ((x-1)/2)^{j}$
is the $n$th Legendre polynomial, where $f_{i}$ is the number of
$i$-dimensional faces in any pulling triangulation of the boundary of~$P_{n}$.
Another way to view the $n$-dimensional root polytope is to
intersect the hyperplane
$x_{1} + x_{2} + \cdots + x_{n+1} = 0$
with the $(n+1)$-dimensional cross-polytope
formed by the convex hull of the vertices
$\pm 2e_{1}$, $\pm 2e_{2}$, $\ldots$, $\pm 2 e_{n+1}$.

The full root polytope~$P_{n}$ contains the positive root polytope~$P^{+}_{n}$.
The polytope~$P^{+}_{n}$ was first studied by Gelfand, Graev and
Postnikov~\cite{Gelfand_Graev_Postnikov} and later by
Postnikov~\cite{Postnikov_P}. Many properties of the full root polytope~$P_{n}$
are straightforward generalizations of those of the positive root
polytope~$P^{+}_{n}$.   
In this paper we refer to $P_{n}$ as the root polytope.

We use the shorthand notation $(i,j)$
for the vertex $e_{j}-e_{i}$ of the root polytope~$P_{n}$, and we
introduce the notation 
$$
V_{n}
=
\{(i,j) \:\: : \:\: 1 \leq i,j \leq n+1, i \neq j\} .
$$
We consider the vertices $(i,j)$ as the set of all
directed nonloop edges on the vertex set $\{1,2, \ldots, n+1\}$. To avoid
confusion between edges and vertices of the root polytope, we will
refer to the vertices of~$P_{n}$ as {\em arrows}.
The positive root polytope~$P^{+}_{n}$ is then the convex hull of the origin
and of the forward arrows.
Suppressing the orientation, we arrive at the edges
representing vertices in the work of Gelfand, Graev and
Postnikov~\cite{Gelfand_Graev_Postnikov}. 

A subset of arrows is contained in some proper face of $P_{n}$
exactly when there is no $i \in \{1,2, \ldots, n+1\}$ that is both the
head and the tail of an arrow;
see~\cite[Lemmas~4.2 and~4.4]{Hetyei_Legendre}.
Equivalently, the faces are products of two
simplices~\cite[Lemma 2.2]{Ehrenborg_Hetyei_Readdy}: 
they are of the form
$(-\Simplex_{I}) \times \Simplex_{J}$ where $I,J\neq \emptyset$,
$I \cap J=\emptyset$ and the symbol $\Simplex_{K}$
denotes the convex hull of the set $\{e_{i} \: : \: i \in K\}$ for
$K\subseteq \{1,2, \ldots,  n+1\}$.
Keeping consistent with our shorthand notation for the vertices,
in the remainder of the paper we simply write
each proper face of $P_{n}$ as $\Simplex_{I} \times \Simplex_{J}$.
The facets of $P_{n}$ are exactly the faces
$\Simplex_{I} \times \Simplex_{J}$ where the disjoint union of $I$ and $J$
is $\{1,2, \ldots, n+1\}$.  The edges of $P_{n}$ are of the form
$\Simplex_{I} \times \Simplex_{J}$ where $\{|I|,|J|\}=\{1,2\}$. 
The two-dimensional faces $\Simplex_{I} \times \Simplex_{J}$
of $P_{n}$ are either squares when $|I| = |J| = 2$ or
triangles when $\{|I|,|J|\}=\{1,3\}$. Extending the
terminology introduced in~\cite{Gelfand_Graev_Postnikov} for 
the positive root polytope~$P^{+}_{n}$, we make the following definition.
\begin{definition}
  A set of arrows $\sigma\subset V_n$ is {\em admissible} if it is
  contained in $I\times J$ for some pair $(I,J)$ of disjoint subsets of
  $\{1,2, \ldots, n+1\}$. 
\end{definition}  

Affine independent subsets of vertices of faces of the root polytope
are described as follows.
A set $S=\{(i_{1},j_{1}),(i_{2},j_{2}), \ldots, (i_{k},j_{k})\}$
is a $(k-1)$-dimensional simplex if and only if, 
disregarding the orientation of the directed
edges, the set $S$ contains no cycle, that is, it is a
forest~\cite[Lemma~2.4]{Hetyei_Legendre}. The analogous observations were
made for the positive root polytope~$P^{+}_{n}$
in~\cite{Gelfand_Graev_Postnikov} and for products of simplices
in~\cite[Lemma~6.2.8]{deLoera_Rambau_Santos}
(see also~\cite[Lemma~2.1]{Ceballos_Padrol_Sarmiento_old}).     

Clearly every triangulation of the boundary~$\partial P_{n}$ of the root
polytope that does not introduce any new vertices must consist of faces
whose vertices are represented by admissible sets of arrows which,
disregarding the orientations of the arrows, are forests. To these known
necessary conditions we would like to add one more necessary condition:
the $2$-skeleton of every such triangulation must be a triangulation of
the $2$-skeleton of the boundary~$\partial P_{n}$ of the root
polytope. As noted above, the $2$-skeleton of the boundary~$\partial
P_{n}$ has only triangular and square faces. Triangulating such a
complex without introducing any new vertices amounts to adding exactly
one diagonal to each square, cutting each square into two triangles. 
Thus our necessary condition may be equivalently restated as
follows. Recall that a simplicial complex is {\em flag}
if the minimal non-faces have cardinality~$2$.

\begin{lemma}
\label{lemma_easy}
Let $\triangle_{n}$ be a triangulation of the boundary
of the root polytope~$P_{n}$ which does not introduce any new vertices.
Then its $2$-skeleton satisfies the following conditions:
\begin{enumerate}[(a)]
\item
For each three-element subset $\{i,j,k\}$ of the set $\{1,2, \ldots, n+1\}$,
the two sets $\{(i,j),(i,k)\}$ and $\{(j,i),(k,i)\}$ are edges
in the complex~$\triangle_{n}$.
\item
For two disjoint pairs $\{i_{1},i_{2}\}$ and $\{j_{1},j_{2}\}$ of
the set $\{1,2, \ldots, n+1\}$, exactly one of the two sets
$\{(i_{1},j_{1}),(i_{2},j_{2})\}$
and $\{(i_{1},j_{2}),(i_{2},j_{1})\}$ is an edge in the
complex~$\triangle_{n}$.
\item The $2$-skeleton of $\triangle_n$ is flag.
\end{enumerate}  
\end{lemma}
\begin{proof}
The set of line segments of the form $[e_{j} - e_{i},  e_{k} - e_{i}]$
are exactly the edges of the root polytope~$P_{n}$, hence condition~(a)
is equivalent to stating that $\triangle_{n}$ must contain all edges of
$P_{n}$. Condition~(b) is a way to restate that 
exactly one of the two diagonals of
any square face $\Delta_{\{i_{1},i_{2}\}} \times \Delta_{\{j_{1},j_{2}\}}$
belongs to the triangulation. Finally, condition~(c) is an obvious
consequence of the fact that a triangulation of a
$2$-dimensional polygonal complex that does not introduce new vertices
is uniquely determined by the set of diagonals added to create the
triangulation. 
\end{proof}

\begin{definition}
\label{definition_permissible}  
We say the simplicial complex $\triangle_n$ on the vertex set
$V_{n}$ is {\em permissible} if it satisfies the following criteria:
\begin{enumerate}
\item Each face of $\triangle_n$ is an admissible set of arrows.
\item Disregarding the orientation on the arrows, the set of arrows
  contained in any face of $\triangle_n$ is circuit-free (a forest).
\item The complex $\triangle_n$ satisfies the conditions stated in
  Lemma~\ref{lemma_easy}.   
\end{enumerate}
\end{definition}  
Equivalently, the faces of $\triangle_n$ are simplices
contained in the boundary of the root polytope~$P_{n}$, and the
$2$-skeleton of $\triangle_n$ is a triangulation of the boundary of
$P_{n}$. 

Most, but not all triangulations we will consider in this paper are {\em pulling
  triangulations}. We will use Hudson's~\cite{Hudson} definition. 
\begin{definition}[Hudson]
\label{definition_pull}
Given a polytopal complex $\Gamma$ on a vertex set
$V$ and a linear ordering $<$ of the vertex set, 
the {\em pulling triangulation}
$\triangle(\Gamma)$ with respect to $<$ is defined
recursively as follows. We set 
$\triangle(\Gamma)=\Gamma$ if $\Gamma$ consists
of a single vertex. Otherwise let $v$ be the least element of $V$ with
respect to $<$ and set 
$$
\triangle(\Gamma)
=
\triangle(\Gamma- v)\cup
\bigcup_F \left\{ \conv(\{v\}\cup G)\::\: G\in
\triangle(\Gamma(F))\right\} ,
$$
where the union runs over the facets $F$ not containing $v$ of the
maximal faces of $\Gamma$ which contain~$v$.
Here $\Gamma- v$ is the complex consisting
of all faces of~$\Gamma$ not
containing the vertex $v$, and 
$\Gamma(F)$ is the complex of all
faces of $\Gamma$ contained in~$F$. The triangulations
$\triangle(\Gamma- v)$ and $\triangle(\Gamma(F))$ are
with respect to the order~$<$
restricted to their respective vertex sets.
\end{definition}

\begin{theorem}[Hudson]
The pulling triangulation $\triangle(\Gamma)$
is a triangulation of
the polytopal complex~$\Gamma$
without introducing any new vertices.
\label{theorem_pull}
\end{theorem}

A pulling triangulation of the boundary complex of a polytope~$P$ 
may be geometrically described as follows.
\begin{theorem}
Let~$P$ be an $n$-dimensional polytope in~$\Rrr^{n}$
with the linearly ordered vertex set
$V = \{v_{1} < v_{2} < \cdots < v_{m}\}$.
For each vertex~$v_{i}$ pick a generic vector~$w_{i}$ such that
the point $v_{i} + w_{i}$ can see every facet containing the vertex~$v_{i}$.
Then the combinatorial type of polytope
$Q = \conv(\{
v_{1} + \epsilon \cdot w_{1},
v_{2} + \epsilon^{2} \cdot w_{2},
\ldots,
v_{m} + \epsilon^{m} \cdot w_{m}
\})$
for sufficiently small enough $\epsilon > 0$,
does not depend on the choice of $\epsilon, w_{1}, \ldots, w_{m}$.
Furthermore, the polytope~$Q$ is a simplicial polytope.
The boundary of~$Q$, that is, $\partial Q$ is the
pulling triangulation of the boundary~$\partial P$
with respect to the linear order $\{v_{1}, v_{2}, \ldots, v_{m}\}$.
\end{theorem}

In a recent paper~\cite{Ehrenborg_Hetyei_Readdy}, the authors
have shown that the Simion type~$B$ associahedron~\cite{Simion},
also known as the cyclohedron,
is combinatorially equivalent to a pulling triangulation of the boundary of
the root polytope.
Here we only point out the following key
observation~\cite[Theorem 3.1]{Ehrenborg_Hetyei_Readdy}.
\begin{theorem}
Every pulling triangulation of the boundary of the root polytope~$P_{n}$
is flag.
\end{theorem}
Thus the triangulation giving rise
to a combinatorial equivalent of the Simion type~$B$ associahedron is
completely determined by the rules given in the associated column of
Table~\ref{table_six_way}, which summarizes all the pulling
triangulations considered in~\cite{Ehrenborg_Hetyei_Readdy}.

\newcommand{\lift}[1]{\raisebox{2.5mm}{#1}}

\newcommand{\base}[4]
{
\node[circle, inner sep = 0.9pt] (a1) at (0, 0) {$#1$};
\node[circle, inner sep = 0.9pt] (a2) at (0.7, 0) {$#2$};
\node[circle, inner sep = 0.9pt] (a3) at (1.4, 0) {$#3$};
\node[circle, inner sep = 0.9pt] (a4) at (2.1, 0) {$#4$};
}

\newcommand{\TheNewNewTable}{
\begin{table}[t]
\begin{center}
\begin{tabular}{|c|c||c|c|c|}  
\hline
  Type & Order of nodes  & Lexicographic & Revlex & Type~$B$\\
  & & pulling & pulling & associahedron \\
\hline
\hline  
\lift{$THTH$} & 
\lift{$i_{1}<j_{1}<i_{2}<j_{2}$}
&
\begin{tikzpicture}
\base{i_{1}}{j_{1}}{i_{2}}{j_{2}}
\draw[->, thick] (a1) to[out=55, in=125] (a2);
\draw[->, thick] (a3) to[out=55, in=125] (a4);
\end{tikzpicture}
&
\begin{tikzpicture}
\base{i_{1}}{j_{1}}{i_{2}}{j_{2}}
\draw[->, thick] (a1) to[out=55, in=125] (a4);
\draw[<-, thick] (a2) to[out=55, in=125] (a3);
\end{tikzpicture}
& 
\begin{tikzpicture}
\base{i_{1}}{j_{1}}{i_{2}}{j_{2}}
\draw[->, thick] (a1) to[out=55, in=125] (a4);
\draw[<-, thick] (a2) to[out=55, in=125] (a3);
\end{tikzpicture}
\\
\hline
\lift{$HTHT$} & 
\lift{$j_{1}<i_{1}<j_{2}<i_{2}$}
&
\begin{tikzpicture}
\base{j_{1}}{i_{1}}{j_{2}}{i_{2}}
\draw[<-, thick] (a1) to[out=55, in=125] (a2);
\draw[<-, thick] (a3) to[out=55, in=125] (a4);
\end{tikzpicture}
&
\begin{tikzpicture}
\base{j_{1}}{i_{1}}{j_{2}}{i_{2}}
\draw[<-, thick] (a1) to[out=55, in=125] (a4);
\draw[->, thick] (a2) to[out=55, in=125] (a3);
\end{tikzpicture}
& 
\begin{tikzpicture}
\base{j_{1}}{i_{1}}{j_{2}}{i_{2}}
\draw[<-, thick] (a1) to[out=55, in=125] (a2);
\draw[<-, thick] (a3) to[out=55, in=125] (a4);
\end{tikzpicture}
\\
\hline
\lift{$THHT$} & 
\lift{$i_{1}<j_{1}<j_{2}<i_{2}$}
&
\begin{tikzpicture}
\base{i_{1}}{j_{1}}{j_{2}}{i_{2}}
\draw[->, thick] (a1) to[out=55, in=125] (a2);
\draw[<-, thick] (a3) to[out=55, in=125] (a4);
\end{tikzpicture}
&
\begin{tikzpicture}
\base{i_{1}}{j_{1}}{j_{2}}{i_{2}}
\draw[->, thick] (a1) to[out=55, in=125] (a3);
\draw[<-, thick] (a2) to[out=55, in=125] (a4);
\end{tikzpicture}
& 
\begin{tikzpicture}
\base{i_{1}}{j_{1}}{j_{2}}{i_{2}}
\draw[->, thick] (a1) to[out=55, in=125] (a2);
\draw[<-, thick] (a3) to[out=55, in=125] (a4);
\end{tikzpicture}
\\
\hline
\lift{$HTTH$} & 
\lift{$j_{1}<i_{1}<i_{2}<j_{2}$}
&
\begin{tikzpicture}
\base{j_{1}}{i_{1}}{i_{2}}{j_{2}}
\draw[<-, thick] (a1) to[out=55, in=125] (a2);
\draw[->, thick] (a3) to[out=55, in=125] (a4);
\end{tikzpicture}
&
\begin{tikzpicture}
\base{j_{1}}{i_{1}}{i_{2}}{j_{2}}
\draw[<-, thick] (a1) to[out=55, in=125] (a3);
\draw[->, thick] (a2) to[out=55, in=125] (a4);
\end{tikzpicture}
& 
\begin{tikzpicture}
\base{j_{1}}{i_{1}}{i_{2}}{j_{2}}
\draw[<-, thick] (a1) to[out=55, in=125] (a2);
\draw[->, thick] (a3) to[out=55, in=125] (a4);
\end{tikzpicture}
\\
\hline
\lift{$TTHH$} &  
\lift{$i_{1}<i_{2}<j_{1}<j_{2}$}
&
\begin{tikzpicture}
\base{i_{1}}{i_{2}}{j_{1}}{j_{2}}
\draw[->, thick] (a1) to[out=55, in=125] (a3);
\draw[->, thick] (a2) to[out=55, in=125] (a4);
\end{tikzpicture}
&
\begin{tikzpicture}
\base{i_{1}}{i_{2}}{j_{1}}{j_{2}}
\draw[->, thick] (a1) to[out=55, in=125] (a4);
\draw[->, thick] (a2) to[out=55, in=125] (a3);
\end{tikzpicture}
&  
\begin{tikzpicture}
\base{i_{1}}{i_{2}}{j_{1}}{j_{2}}
\draw[->, thick] (a1) to[out=55, in=125] (a4);
\draw[->, thick] (a2) to[out=55, in=125] (a3);
\end{tikzpicture}
\\
\hline
\lift{$HHTT$} & 
\lift{$j_{1}<j_{2}<i_{1}<i_{2}$}
&
\begin{tikzpicture}
\base{j_{1}}{j_{2}}{i_{1}}{i_{2}}
\draw[<-, thick] (a1) to[out=55, in=125] (a3);
\draw[<-, thick] (a2) to[out=55, in=125] (a4);
\end{tikzpicture}
&
\begin{tikzpicture}
\base{j_{1}}{j_{2}}{i_{1}}{i_{2}}
\draw[<-, thick] (a1) to[out=55, in=125] (a4);
\draw[<-, thick] (a2) to[out=55, in=125] (a3);
\end{tikzpicture}
& 
\begin{tikzpicture}
\base{j_{1}}{j_{2}}{i_{1}}{i_{2}}
\draw[<-, thick] (a1) to[out=55, in=125] (a4);
\draw[<-, thick] (a2) to[out=55, in=125] (a3);
\end{tikzpicture}
\\
\hline
\end{tabular}
\end{center}
\vspace*{3mm}
\caption{Pairs of arrows that are edges in three triangulations of
the boundary~$\partial P_{n}$ of the root polytope.}
\label{table_six_way}  
\end{table}
}

\TheNewNewTable

The first two columns in Table~\ref{table_six_way} are the 
rules for two other pulling triangulations of the boundary of
the root polytope~$P_{n}$, also discussed
in~\cite{Ehrenborg_Hetyei_Readdy}. These are the lexicographic (lex) and
revlex pulling orders.
Their restriction to the positive root polytope~$P^{+}_n$
are called the antistandard, respectively, standard
triangulations in~\cite{Gelfand_Graev_Postnikov}. The terminology we use
in~\cite{Ehrenborg_Hetyei_Readdy} was introduced
in~\cite{Hetyei_Legendre}, where it was observed 
that these are pulling triangulations. (These pulling triangulations
are not to be confused with the terms ``lexicographic triangulation'' and
``reverse lexicographic triangulation'' used in~\cite{Sturmfels} where
the first is a {\em placing}  triangulation, and only the second is a true pulling
triangulation.) The lexicographic pulling order was also studied
in~\cite{Ardila_Beck_Hosten_Pfeifle_Seashore}.
In this paper we characterize all triangulations which are defined by
similar local rules. Some of these are not necessarily pulling
triangulations. In this processs we obtain an independent verification
of the fact that the local rules shown in Table~\ref{table_six_way}
yield triangulations of the boundary~$\partial P_{n}$ of the root polytope.

\subsection{Characterizing triangulations of
$\Simplex_{a-1} \times \Simplex_{b-1}$ via matching ensembles}

Our key tool to verify that the flag complexes we define triangulate
the boundary of the root polytope~$P_{n}$ is the characterization of
the triangulations of the Cartesian product
$\Simplex_{a-1} \times \Simplex_{b-1}$
given by Oh and Yoo~\cite{Oh_Yoo1,Oh_Yoo_2}.
See also~\cite{Ceballos_Padrol_Sarmiento_old}.
We identify the vertices of $\Simplex_{a-1} \times \Simplex_{b-1}$
with edges in the complete bipartite graph~$K_{a,b}$,
whose vertex set is
$\{1,2, \ldots, a\} \uplus \{\overline{1}, \overline{2}, \ldots, \overline{b}\}$,
and call this the
{\em bipartite graph representation} of
$\Simplex_{a-1} \times \Simplex_{b-1}$.
By~\cite[Lemma~6.2.8]{deLoera_Rambau_Santos}
a set of affine independent vertices
of $\Simplex_{a-1} \times \Simplex_{b-1}$ corresponds to a forest in
the bipartite graph representation, and maximal affine independent
sets correspond to trees. Facets of a triangulation of
$\Simplex_{a-1} \times \Simplex_{b-1}$
thus correspond to spanning trees.
The results of Oh and Yoo
characterize 
those sets of spanning trees that correspond to a triangulation
of $\Simplex_{a-1} \times \Simplex_{b-1}$.

\begin{definition}
A family $\mathcal{M}$ of matchings of $K_{a,b}$
is a {\em matching ensemble} if it satisfies the following 
three axioms:
\begin{description}
\item[Support axiom]
For $I\subseteq \{1,2, \ldots, a\}$ and
$\overline{J}\subseteq \{\overline{1},\overline{2}, \ldots, \overline{b}\}$
with $|I|=|\overline{J}|$ there is a unique matching 
in $\mathcal{M}$ that matches the elements of $I$ with the elements of
$\overline{J}$ in the subgraph induced by $I \uplus \overline{J}$ of $K_{a,b}$.
\item[Closure axiom]
Any subset of a matching in $\mathcal{M}$
is also a matching in $\mathcal{M}$.  
\item[Linkage axiom]
If $m$ is a non-empty matching in $\mathcal{M}$ and $v$ is
any vertex of $K_{a,b}$ not incident to any edge of $m$ then there is
an edge $e \in m$ and there is an edge $e'\not\in m$ incident to $v$
such that the resulting matching $m'=(m-e)\cup e'$ also belongs to 
$\mathcal{M}$.    
\end{description}
\label{definition_SA_CA_LA}
\end{definition}  

For~$T$ a spanning tree of $K_{a,b}$ define $\phi(T)$ to be the set of all
matchings contained in the edges of the tree~$T$.
Extend this notion to families of spanning trees by defining
\begin{align*}
\Phi(\mathcal{T})
& =
\bigcup_{T \in \mathcal{T}} \phi(T) .
\end{align*}
S.\ Oh and H.\ Yoo proved the following
result~\cite[Theorem~5.4]{Oh_Yoo_2}.
\begin{theorem}[Oh--Yoo]
\label{theorem_me}
The function $\Phi$ is a bijection between families of spanning trees
representing triangulations of $\Simplex_{a-1} \times \Simplex_{b-1}$
and matching ensembles of the bipartite graph $K_{a,b}$.
\end{theorem}

Ceballos, Padrol and Sarmiento~\cite[Lemma~2.5]{Ceballos_Padrol_Sarmiento}
explicitly describe the inverse $\Phi^{-1}$.
\begin{lemma}[Ceballos--Padrol--Sarmiento]
Given a matching ensemble $\mathcal{M}$ on $K_{a,b}$, the spanning tree~$T$
of $K_{a,b}$ belongs to $\Phi^{-1}(\mathcal{M})$
if and only if for each matching 
$m \in \mathcal{M}$, there is no cycle in $T \cup m$ that alternates
between $T$ and~$m$.   
\label{lemma_mefacets}
\end{lemma}

Closely related to this result is the following lemma,
essentially due to
Postnikov; see Lemma~12.6 in~\cite{Postnikov_P}.
Although Postnikov
originally made the
statement 
in the case of
spanning trees, the proof carries over with very
little modification to the case of forests.
Recall that for a forest~$F$ in~$K_{a,b}$ we denote
$\Simplex_{F}$ to be simplex in $\Simplex_{a-1} \times \Simplex_{b-1}$
whose vertices correspond to the edges of the forest.

\begin{lemma}[Postnikov]
Let $F$ and $F'$ be two forests in the bipartite graph $K_{a,b}$.
The intersection of the two simplices $\Simplex_{F} \cap \Simplex_{F'}$
is either empty or a simplex represented by
a set of edges of $K_{a,b}$ if and only if the graph $F \cup F'$ does not
contain a cycle of length greater than or equal to $4$
in which the edges alternate between~$F$ and~$F'$.
\label{lemma_Postnikov}
\end{lemma}
\begin{proof}
The proof of the necessity is exactly the same
as for spanning trees.
If there is a cycle
$(i_{1},\overline{j_{1}},i_{2}, \overline{j_{2}}, \ldots,  i_{k},\overline{j_{k}})$
such that
$$
\{\{i_{1},\overline{j_{1}}\}, \{i_{2},\overline{j_{2}}\}, \ldots,
\{i_{k},\overline{j_{k}}\}\} \subseteq F
\:\:\:\: \text{ and } \:\:\:\:
\{\{i_{1},\overline{j_{k}}\}, \{i_{2},\overline{j_{1}}\}, \ldots,
\{i_{k},\overline{j_{k-1}}\}\} \subseteq F'
$$
then the point
$1/k \cdot \sum_{s=1}^{k} (e_{i_{s}} - e_{\overline{j_{s}}})$
belongs
to the intersection $\Simplex_{F} \cap \Simplex_{F'}$, but all vertices
of the smallest dimensional faces of the two simplices containing the
point do not belong to the intersection. To prove the converse, we only
need to add one 
sentence to Postnikov's proof. 
Direct all edges
$\{i,\overline{j}\} \in F\setminus F'$ from $i$ to $\overline{j}$ and
direct all edges $\{i,\overline{j}\} \in F'\setminus F$ from
$\overline{j}$ to $i$. The resulting set $U(F,F')$ of directed edges is
acyclic. We select a height function that is constant on the connected
components of $F\cap F'$ and increases along the directed edges in
$U(F,F')$ joining two connected components of $F\cap F'$. Since we
started with forests instead of spanning trees, the resulting height
function is still undefined on those nodes that are not incident to any edge
in the union $F \cup F'$. For these, we select the height to be constant
less than any of the already defined values.  
\end{proof}

\section{Classifying uniform flag triangulations of the root polytope}
\label{section_classification}

\subsection{Uniform triangulations and their classification}

A common property of all three flag complexes described in
Table~\ref{table_six_way} is that the edges are defined in
a {\em uniform fashion}.
We make this clear in the following definition.

\begin{definition}
\label{definition_uflag}
A flag simplicial complex $\triangle_{n}$ on the vertex set~$V_{n}$
is a {\em uniform flag complex}
if determining whether or not
a pair of vertices $\{(i_{1},j_{1}),(i_{2},j_{2})\}$ forms an edge
depends only on the equalities and inequalities
between the values of $i_{1}$, $i_{2}$, $j_{1}$ and~$j_{2}$.
\end{definition}

We begin with the necessary conditions for describing
uniform flag triangulations.
To facilitate making statements,
we introduce some new terminology and notation. We use the letter
$T$ to mark the tail of each arrow and the letter $H$ to mark the head.
For each pair of arrows on four nodes, we will indicate the
relative order of the two heads and two tails by writing down the
appropriate letters left to right in the order as they occur. We will
refer to the resulting word as the {\em type} of the pair of arrows.
After that we will simply state in words the condition
that a pair of arrows
of a given type must satisfy to be an edge of the triangulation.
Examples of this convention are given in Table~\ref{table_six_way}.

The main classification result in this paper is the following.
\begin{theorem}
Let  $\triangle_{n}$ be a permissible uniform flag complex on the vertex
set~$V_{n}$ for some $n \geq 5$.
Then the complex~$\triangle_{n}$ represents a triangulation of
the boundary~$\partial P_{n}$ of the root polytope if and only if it
satisfies exactly one of the following conditions:
\begin{enumerate}
\item
Both $THTH$ and $HTHT$ types of pairs of arrows do not nest,
and both $HTTH$ and $THHT$ types of arrows do not cross.
\item
Both $THTH$ and $HTHT$ types of pairs of arrows nest,
and both $HTTH$ and $THHT$ types of arrows cross.
\item
Exactly one of the $THTH$ and $HTHT$ types of pairs of arrows nest.
Furthermore, if
both $THHT$ and $HTTH$ types of pairs cross
then both $TTHH$ and $HHTT$ types of pairs nest.
\end{enumerate}
\label{theorem_main}
\end{theorem}  

The three classes in listed in Theorem~\ref{theorem_main} are
pairwise mutually exclusive. We name them as follows, and give
brief motivations.
\begin{enumerate}
\item 
This class contains the triangulation obtained by
the lexicographic pulling order of the root polytope
and hence is named {\em the lex class.}
\item 
This class contains the triangulation obtained by
the revlex pulling order of the root polytope
and hence is named {\em the revlex class.}
\item 
This class contains the Simion type~$B$ associahedron
and hence is called {\em the Simion class}.
Furthermore, we subdivide this class
into the three subclasses,
the Simion subclass of types~$a$ through~$c$, according to:
\begin{enumerate}
\item
Both $THHT$ and $HTTH$ types of pairs do not cross.
\item
Exactly one of the $THHT$ and $HTTH$ types of pairs cross.
\item
Both $THHT$ and $HTTH$ types of pairs cross,
and both $TTHH$ and $HHTT$ types of pairs nest.
\end{enumerate}
\end{enumerate}

In Subsection~\ref{subsection_necessary_conditions}
we prove the necessity part of Theorem~\ref{theorem_main}
in
Propositions~\ref{proposition_nonest}
through~\ref{proposition_nest_follow}.
The sufficiency part of Theorem~\ref{theorem_main} is
proved Section~\ref{section_sufficiency}.
The main tool for proving these results is
Theorem~\ref{theorem_legendreme}
which gives necessary and sufficient conditions for
a simplicial complex on the vertex set~$V_{n}$
to be a triangulation of the boundary~$\partial P_{n}$
of the root polytope. These conditions,
called the support and the linkage axioms,
are based upon
Definition~\ref{definition_SA_CA_LA}.

\newcommand{\smallbase}[4]
{
\node[circle, inner sep = 0.9pt] (a1) at (0, 0) {$\scriptstyle #1$};
\node[circle, inner sep = 0.9pt] (a2) at (0.6, 0) {$\scriptstyle #2$};
\node[circle, inner sep = 0.9pt] (a3) at (1.2, 0) {$\scriptstyle #3$};
\node[circle, inner sep = 0.9pt] (a4) at (1.8, 0) {$\scriptstyle #4$};
}

\begin{table}[ht]
\begin{center}
\begin{tabular}{|c|c|c||c|c|c|}
\hline
\vphantom{$\displaystyle \sum$}
$\triangle$ & $\triangle^{*}$ & $\overline{\triangle}$ &
$\triangle$ & $\triangle^{*}$ & $\overline{\triangle}$ \\
\hline\hline
\begin{tikzpicture}
\smallbase{T}{H}{T}{H}
\draw[->, thick] (a1) to[out=55, in=125] (a4);
\draw[<-, thick] (a2) to[out=55, in=125] (a3);
\end{tikzpicture}
&
\begin{tikzpicture}
\smallbase{H}{T}{H}{T}
\draw[<-, thick] (a1) to[out=55, in=125] (a4);
\draw[->, thick] (a2) to[out=55, in=125] (a3);
\end{tikzpicture}
&
\begin{tikzpicture}
\smallbase{T}{H}{T}{H}
\draw[->, thick] (a1) to[out=55, in=125] (a4);
\draw[<-, thick] (a2) to[out=55, in=125] (a3);
\end{tikzpicture}
&
\begin{tikzpicture}
\smallbase{T}{H}{T}{H}
\draw[->, thick] (a1) to[out=55, in=125] (a2);
\draw[->, thick] (a3) to[out=55, in=125] (a4);
\end{tikzpicture}
&
\begin{tikzpicture}
\smallbase{H}{T}{H}{T}
\draw[<-, thick] (a1) to[out=55, in=125] (a2);
\draw[<-, thick] (a3) to[out=55, in=125] (a4);
\end{tikzpicture}
&
\begin{tikzpicture}
\smallbase{T}{H}{T}{H}
\draw[->, thick] (a1) to[out=55, in=125] (a2);
\draw[->, thick] (a3) to[out=55, in=125] (a4);
\end{tikzpicture}
\\
\hline\hline
\begin{tikzpicture}
\smallbase{H}{T}{H}{T}
\draw[<-, thick] (a1) to[out=55, in=125] (a4);
\draw[->, thick] (a2) to[out=55, in=125] (a3);
\end{tikzpicture}
&
\begin{tikzpicture}
\smallbase{T}{H}{T}{H}
\draw[->, thick] (a1) to[out=55, in=125] (a4);
\draw[<-, thick] (a2) to[out=55, in=125] (a3);
\end{tikzpicture}
&
\begin{tikzpicture}
\smallbase{H}{T}{H}{T}
\draw[<-, thick] (a1) to[out=55, in=125] (a4);
\draw[->, thick] (a2) to[out=55, in=125] (a3);
\end{tikzpicture}
&
\begin{tikzpicture}
\smallbase{H}{T}{H}{T}
\draw[<-, thick] (a1) to[out=55, in=125] (a2);
\draw[<-, thick] (a3) to[out=55, in=125] (a4);
\end{tikzpicture}
&
\begin{tikzpicture}
\smallbase{T}{H}{T}{H}
\draw[->, thick] (a1) to[out=55, in=125] (a2);
\draw[->, thick] (a3) to[out=55, in=125] (a4);
\end{tikzpicture}
&
\begin{tikzpicture}
\smallbase{H}{T}{H}{T}
\draw[<-, thick] (a1) to[out=55, in=125] (a2);
\draw[<-, thick] (a3) to[out=55, in=125] (a4);
\end{tikzpicture}
\\
\hline\hline
\begin{tikzpicture}
\smallbase{T}{H}{H}{T}
\draw[->, thick] (a1) to[out=55, in=125] (a2);
\draw[<-, thick] (a3) to[out=55, in=125] (a4);
\end{tikzpicture}
&
\begin{tikzpicture}
\smallbase{H}{T}{T}{H}
\draw[<-, thick] (a1) to[out=55, in=125] (a2);
\draw[->, thick] (a3) to[out=55, in=125] (a4);
\end{tikzpicture}
&
\begin{tikzpicture}
\smallbase{H}{T}{T}{H}
\draw[<-, thick] (a1) to[out=55, in=125] (a2);
\draw[->, thick] (a3) to[out=55, in=125] (a4);
\end{tikzpicture}
&
\begin{tikzpicture}
\smallbase{T}{H}{H}{T}
\draw[->, thick] (a1) to[out=55, in=125] (a3);
\draw[<-, thick] (a2) to[out=55, in=125] (a4);
\end{tikzpicture}
&
\begin{tikzpicture}
\smallbase{H}{T}{T}{H}
\draw[<-, thick] (a1) to[out=55, in=125] (a3);
\draw[->, thick] (a2) to[out=55, in=125] (a4);
\end{tikzpicture}
&
\begin{tikzpicture}
\smallbase{H}{T}{T}{H}
\draw[<-, thick] (a1) to[out=55, in=125] (a3);
\draw[->, thick] (a2) to[out=55, in=125] (a4);
\end{tikzpicture}
\\
\hline\hline
\begin{tikzpicture}
\smallbase{H}{T}{T}{H}
\draw[<-, thick] (a1) to[out=55, in=125] (a2);
\draw[->, thick] (a3) to[out=55, in=125] (a4);
\end{tikzpicture}
&
\begin{tikzpicture}
\smallbase{T}{H}{H}{T}
\draw[->, thick] (a1) to[out=55, in=125] (a2);
\draw[<-, thick] (a3) to[out=55, in=125] (a4);
\end{tikzpicture}
&
\begin{tikzpicture}
\smallbase{T}{H}{H}{T}
\draw[->, thick] (a1) to[out=55, in=125] (a2);
\draw[<-, thick] (a3) to[out=55, in=125] (a4);
\end{tikzpicture}
&
\begin{tikzpicture}
\smallbase{H}{T}{T}{H}
\draw[<-, thick] (a1) to[out=55, in=125] (a3);
\draw[->, thick] (a2) to[out=55, in=125] (a4);
\end{tikzpicture}
&
\begin{tikzpicture}
\smallbase{T}{H}{H}{T}
\draw[->, thick] (a1) to[out=55, in=125] (a3);
\draw[<-, thick] (a2) to[out=55, in=125] (a4);
\end{tikzpicture}
&
\begin{tikzpicture}
\smallbase{T}{H}{H}{T}
\draw[->, thick] (a1) to[out=55, in=125] (a3);
\draw[<-, thick] (a2) to[out=55, in=125] (a4);
\end{tikzpicture}
\\
\hline\hline
\begin{tikzpicture}
\smallbase{H}{H}{T}{T}
\draw[<-, thick] (a1) to[out=55, in=125] (a4);
\draw[<-, thick] (a2) to[out=55, in=125] (a3);
\end{tikzpicture}
&
\begin{tikzpicture}
\smallbase{T}{T}{H}{H}
\draw[->, thick] (a1) to[out=55, in=125] (a4);
\draw[->, thick] (a2) to[out=55, in=125] (a3);
\end{tikzpicture}
&
\begin{tikzpicture}
\smallbase{H}{H}{T}{T}
\draw[<-, thick] (a1) to[out=55, in=125] (a4);
\draw[<-, thick] (a2) to[out=55, in=125] (a3);
\end{tikzpicture}
&
\begin{tikzpicture}
\smallbase{H}{H}{T}{T}
\draw[<-, thick] (a1) to[out=55, in=125] (a3);
\draw[<-, thick] (a2) to[out=55, in=125] (a4);
\end{tikzpicture}
&
\begin{tikzpicture}
\smallbase{T}{T}{H}{H}
\draw[->, thick] (a1) to[out=55, in=125] (a3);
\draw[->, thick] (a2) to[out=55, in=125] (a4);
\end{tikzpicture}
&
\begin{tikzpicture}
\smallbase{H}{H}{T}{T}
\draw[<-, thick] (a1) to[out=55, in=125] (a3);
\draw[<-, thick] (a2) to[out=55, in=125] (a4);
\end{tikzpicture}
\\
\hline\hline
\begin{tikzpicture}
\smallbase{T}{T}{H}{H}
\draw[->, thick] (a1) to[out=55, in=125] (a4);
\draw[->, thick] (a2) to[out=55, in=125] (a3);
\end{tikzpicture}
&
\begin{tikzpicture}
\smallbase{H}{H}{T}{T}
\draw[<-, thick] (a1) to[out=55, in=125] (a4);
\draw[<-, thick] (a2) to[out=55, in=125] (a3);
\end{tikzpicture}
&
\begin{tikzpicture}
\smallbase{T}{T}{H}{H}
\draw[->, thick] (a1) to[out=55, in=125] (a4);
\draw[->, thick] (a2) to[out=55, in=125] (a3);
\end{tikzpicture}
&
\begin{tikzpicture}
\smallbase{T}{T}{H}{H}
\draw[->, thick] (a1) to[out=55, in=125] (a3);
\draw[->, thick] (a2) to[out=55, in=125] (a4);
\end{tikzpicture}
&
\begin{tikzpicture}
\smallbase{H}{H}{T}{T}
\draw[<-, thick] (a1) to[out=55, in=125] (a3);
\draw[<-, thick] (a2) to[out=55, in=125] (a4);
\end{tikzpicture}
&
\begin{tikzpicture}
\smallbase{T}{T}{H}{H}
\draw[->, thick] (a1) to[out=55, in=125] (a3);
\draw[->, thick] (a2) to[out=55, in=125] (a4);
\end{tikzpicture}
\\
\hline
\end{tabular}
\end{center}
\caption{The action of the two involutions
$\triangle \longmapsto \triangle^{*}$
and
$\triangle \longmapsto \overline{\triangle}$. Observe the columns of
$\triangle^{*}$ and $\overline{\triangle}$ are mirror images of each other.}
\label{table_star_and_overline}
\end{table}

We end this subsection by introducing two commuting operations
on triangulations.
Let  $\triangle_{n}$ be a uniform flag complex on the vertex set~$V_{n}$.
Let the {\em dual triangulation}
$\triangle_{n}^{*}$
be the triangulation obtained by
reversing all the arrows.
Let the {\em reflected dual triangulation}
$\overline{\triangle_{n}}$
be the triangulation obtained by
reversing all the arrows and then 
replacing node $i$ with $n+2-i$.

\begin{lemma}
Let $\triangle_{n}$ be a uniform flag complex on the vertex set~$V_{n}$.
Then the uniform flag complexes
$\triangle_{n}$ and $\triangle_{n}^{*}$
belong to the same (sub)class.
Furthermore,
the conditions on the types
$THTH$, $HTHT$, $TTHH$ and $HHTT$
stay invariant under the
involution $\triangle_{n} \longmapsto \overline{\triangle_{n}}$,
whereas the condition on
the types $THHT$ and $HTTH$ are exchanged.
\label{lemma_involutions}
\end{lemma}
\begin{proof}
See Table~\ref{table_star_and_overline}.
\end{proof}

\subsection{Necessary conditions for uniform flag complex}
\label{subsection_necessary_conditions}

The next result is essential for proving the
necessary and sufficient conditions in
Theorem~\ref{theorem_main}. 

\begin{theorem}
Let $\triangle_{n}$ be a permissible uniform flag complex on the vertex
set~$V_{n}$. Let $\mathcal{M}$ be the family of all
faces that are matchings, that is, let
$$
\mathcal{M}
=
\{\{(i_{1},j_{1}),(i_{2},j_{2}), \ldots, (i_{k},j_{k})\} \in \triangle_{n}
\: : \: 
|\{i_{1},j_{1},i_{2},j_{2}, \ldots, i_{k},j_{k}\}|=2k\}.
$$
Identify each vertex $(i,j)$ with the vertex $e_{j} - e_{i}$ of the root polytope.
Then the complex $\triangle_{n}$ represents a triangulation of
the boundary~$\partial P_{n}$ of the root polytope if and only if 
the family of matchings $\mathcal{M}$ satisfies the following two properties:   
\begin{itemize}
\item[(SA)]
For two disjoint subsets $I,J\subset \{1,2, \ldots,  n+1\}$ satisfying $|I|=|J|$
there is a unique $\sigma\in {\mathcal{M}}$ such that 
$\sigma\subseteq I\times J$ and $|\sigma|=|I|$;
\item[(LA)]
Assume $I$ and $J$ are disjoint subsets of $\{1,2, \ldots,  n+1\}$.
Let $\sigma$ be a non-empty matching in $\mathcal{M}$ such that
$\sigma\subseteq I\times J$.
Then for each $k \not \in I \cup J$ there is
an edge $(i,j) \in \sigma$ such that
$(\sigma-\{(i,j)\}) \cup \{(k,j)\} \in {\mathcal{M}}$.
Also, for each $k \not \in I \cup J$ there is
an edge $(i,j) \in \sigma$ such that
$(\sigma-\{(i,j)\}) \cup \{(i,k)\}\in {\mathcal{M}}$. 
\end{itemize}
\label{theorem_legendreme}  
\end{theorem}  
\begin{proof}
Condition~(1) stated in Definition~\ref{definition_permissible} implies
that each face of $\triangle_{n}$ represents a subset of a proper face
$\Simplex_{I} \times \Simplex_{J}$ of the root polytope~$P_{n}$,
where $I$ and $J$ are disjoint
subsets of $\{1,2, \ldots, n+1\}$. Under these conditions $\triangle_{n}$
represents a triangulation of $\partial P_{n}$ if and only if for each
pair $(I,J)$ of disjoint nonempty subsets of $\{1,2, \ldots, n+1\}$ the
set of faces whose vertices are contained in $I\times J$ represent a
triangulation of the face $\Simplex_{I} \times \Simplex_{J}$ of~$P_{n}$.  
Property (SA) is equivalent to the support axiom
in the definition of a matching ensemble,
while property (LA) is equivalent to the linkage axiom.
The closure axiom is an immediate consequence of the fact
that a subset of a face is a face in a simplicial complex. By
Theorem~\ref{theorem_me} the family of matchings $\mathcal{M}$ must satisfy
the stated axioms.

To prove the converse, assume that $\mathcal{M}$ satisfies the stated
axioms and let $I$ and $J$ be an arbitrary pair of nonempty disjoint
subsets of $\{1,2, \ldots, n+1\}$.
Then the {\em restriction} $\rest{\mathcal{M}}{I \times J}$ of $\mathcal{M}$
to $I\times J$, i.e., the set of matchings contained in $\mathcal{M}$ whose
elements belong to $I\times J$, is a matching ensemble. By
Theorem~\ref{theorem_me} there is a triangulation $\triangle(I,J)$
of $\Simplex_{I} \times \Simplex_{J}$ corresponding to
this matching ensemble for which the elements of
$\rest{\mathcal{M}}{I \times J}$ are the matchings of the complete bipartite graph
$K_{I,J}$ contained in the spanning trees representing the facets of
$\triangle(I,J)$. It suffices to show that $\triangle(I,J)$ is
the family $\rest{\triangle_{n}}{I \times J}$ of faces of $\triangle_{n}$ whose
vertices are contained in $I \times J$. 

Assume, by way of contradiction,
that the two simplicial complexes
$\triangle(I,J)$ and $\rest{\triangle_{n}}{I\times J}$ differ.
If there is a face $\sigma$ of $\triangle(I,J)$ that does not belong to
$\rest{\triangle_{n}}{I \times J}$
then (the vertex sets being equal) this
face $\sigma$ also contains an edge $\{(i_{1},j_{1}),(i_{2},j_{2})\}$ that does
not belong to $\rest{\triangle_{n}}{I\times J}$, since
$\rest{\triangle_{n}}{I\times J}$ is a flag complex which would contain
$\sigma$ if it contained all of its edges. By part (a) of
Lemma~\ref{lemma_easy} the set $\{i_{1},j_{1},i_{2},j_{2}\}$ must have
four distinct elements, and by part (b) of the same lemma we must
have $\{(i_{1},j_{2}),(i_{2},j_{1})\} \in \rest{\triangle_{n}}{I\times J}$. By the
definition of $\mathcal{M}$ we have
$\{(i_{1},j_{2}),(i_{2},j_{1})\} \in \rest{{\mathcal M}}{I\times J}$
and by our assumption we also have
$\{(i_{1},j_{1}),(i_{2},j_{2})\} \in \rest{{\mathcal{M}}}{I\times J}$. This
violates the uniqueness part of the support axiom (SA) for the pair of
subsets $(\{i_{1},i_{2}\},\{j_{1},j_{2}\})$.
Hence $\triangle(I,J)$ is contained in $\rest{\triangle_{n}}{I\times J}$,
that is,
$\triangle(I,J) \subseteq \rest{\triangle_{n}}{I\times J}$.

We are left with the
possibility of having a face $\sigma$ in $\rest{\triangle_{n}}{I\times J}$
that does not belong to $\triangle(I,J)$. After adding a few more
vertices, if necessary, we may assume that this face $\sigma$ is a
facet of $\rest{\triangle_{n}}{I\times J}$. The vertices of the
facet $\sigma$ are arrows from $I$ to $J$ which, disregarding their
orientation, must form a forest.
If some arrows of $\sigma$ form a cycle
$\{\{i_{1},j_{1}\}, \{i_{2},j_{1}\}, \{i_{2},j_{2}\}, \ldots, \{i_{k},j_{k}\}, \{i_{1},j_{k}\}\}$
then the uniqueness part of the support axiom (SA) is
violated for the pair of sets
$(\{i_{1}, i_{2}, \ldots, i_{k}\}, \{j_{1}, j_{2}, \ldots, j_{k}\})$,
since both $\{(i_{1},j_{1}), (i_{2},j_{2}), \ldots, (i_{k},j_{k})\}$
and $\{(i_{1},j_{k}), (i_{2},j_{1}), \ldots, (i_{k},j_{k-1})\}$
belong to $\mathcal{M}$.
Consider the centroid $1/|\sigma| \cdot \sum_{(i,j) \in \sigma} (e_{j} - e_{i})$
of the face of~$P_{n}$ represented by~$\sigma$.
This point belongs to some facet~$\Simplex_{T}$
of the triangulation $\triangle(I,J)$. Here $T$ is the
spanning tree of $K_{I,J}$ representing the facet. All vertices of
$\sigma$ cannot be represented by edges belonging to $T$,
for otherwise $\sigma$ belongs to $\triangle(I,J)$.
Hence, by Lemma~\ref{lemma_Postnikov} there is a cycle
$(i_{1},j_{1},i_{2},j_{2}, \ldots, i_{k},j_{k})$ in $K_{I,J}$ such that the matching 
$\{(i_{1},j_{1}),(i_{2},j_{2}), \ldots, (i_{k},j_{k})\}$ belongs to $T$ and
the matching $\{(i_{1},j_{k}),(i_{2},j_{1}), \ldots, (i_{k},j_{k-1})\}$ belongs to
$\sigma$. Both matchings belong to~$\mathcal{M}$, and we have reached a
contradiction with Lemma~\ref{lemma_mefacets}.
\end{proof}

\begin{figure}
\begin{center}
\begin{tikzpicture}
\node[circle, inner sep = 0.9pt] (a1) at (0, 0) {$\scriptstyle{T}$};
\node[circle, inner sep = 0.9pt] (a2) at (0.7, 0) {$\scriptstyle{H}$};
\node[circle, inner sep = 0.9pt] (a3) at (1.4, 0) {$\scriptstyle{T}$};
\node[circle, inner sep = 0.9pt] (a4) at (2.1, 0) {$\scriptstyle{H}$};
\draw[->, thick] (a1) to[out=55, in=125] (a2);
\draw[->, thick] (a3) to[out=55, in=125] (a4);

\node[circle, inner sep = 0.9pt] () at (2.8, 0) {\&};

\node[circle, inner sep = 0.9pt] (b1) at (3.5, 0) {$\scriptstyle{H}$};
\node[circle, inner sep = 0.9pt] (b2) at (4.2, 0) {$\scriptstyle{T}$};
\node[circle, inner sep = 0.9pt] (b3) at (4.9, 0) {$\scriptstyle{H}$};
\node[circle, inner sep = 0.9pt] (b4) at (5.6, 0) {$\scriptstyle{T}$};
\draw[<-, thick] (b1) to[out=55, in=125] (b2);
\draw[<-, thick] (b3) to[out=55, in=125] (b4);

\node[circle, inner sep = 0.9pt] () at (6.3, 0) {$\Longrightarrow$};

\node[circle, inner sep = 0.9pt] (c1) at (7.0, 0) {$\scriptstyle{T}$};
\node[circle, inner sep = 0.9pt] (c2) at (7.7, 0) {$\scriptstyle{H}$};
\node[circle, inner sep = 0.9pt] (c3) at (8.4, 0) {$\scriptstyle{H}$};
\node[circle, inner sep = 0.9pt] (c4) at (9.1, 0) {$\scriptstyle{T}$};
\draw[->, thick] (c1) to[out=55, in=125] (c2);
\draw[<-, thick] (c3) to[out=55, in=125] (c4);

\node[circle, inner sep = 0.9pt] () at (9.8, 0) {\&};

\node[circle, inner sep = 0.9pt] (d1) at (10.5, 0) {$\scriptstyle{H}$};
\node[circle, inner sep = 0.9pt] (d2) at (11.2, 0) {$\scriptstyle{T}$};
\node[circle, inner sep = 0.9pt] (d3) at (11.9, 0) {$\scriptstyle{T}$};
\node[circle, inner sep = 0.9pt] (d4) at (12.6, 0) {$\scriptstyle{H}$};
\draw[<-, thick] (d1) to[out=55, in=125] (d2);
\draw[->, thick] (d3) to[out=55, in=125] (d4);
\end{tikzpicture}
\end{center}
\caption{A graphical representation of
Proposition~\ref{proposition_nonest}.}
\label{figure_proposition_nonest}
\end{figure}

We now begin to obtain the necessary conditions.

\begin{proposition}
Let $\triangle_{n}$ be a uniform flag complex on the vertex set~$V_{n}$
representing a triangulation of the boundary~$\partial P_{n}$ 
of the root polytope for $n \geq 4$.
Assume that pairs of arrows of types
$THTH$ and $HTHT$ do not nest in the triangulation~$\triangle_{n}$.
Then pairs of arrows of types $HTTH$ and $THHT$ 
do not cross in the triangulation~$\triangle_{n}$.
\label{proposition_nonest}
\end{proposition}
\begin{proof}
Assume by way of contradiction, that $HTTH$ type pairs of arrows
cross in $\triangle_{n}$. 
Hence the edge
$\sigma = \{(2,5), (4,1)\}$ is in the triangulation~$\triangle_{n}$.
Now let $k=3$ and
apply 
the linkage axiom~(LA) in Theorem~\ref{theorem_legendreme}.
We obtain either the edge
$\{(2,5), (4,3)\}$
or
$\{(2,3), (4,1)\}$,
contradicting the assumed condition on $THTH$ or~$HTHT$.
The second conclusion follows by reversing all the arrows in the proof.
\end{proof}

The next necessary condition is completely analogous to that of
Proposition~\ref{proposition_nonest}.

\begin{proposition}
\label{proposition_nest}
Let $\triangle_{n}$ be a uniform flag complex on the vertex set~$V_{n}$
representing a triangulation of the boundary~$\partial P_{n}$ 
of the root polytope for $n \geq 4$.
Assume that pairs of arrows of the types
$THTH$ and $HTHT$ nest in the triangulation~$\triangle_{n}$.
Then pairs of arrows of types $HTTH$ and $THHT$ 
cross in the triangulation~$\triangle_{n}$.
\end{proposition}  
\begin{proof}
The proof is completely analogous to
the proof of Proposition~\ref{proposition_nonest},
but this time use the edge
$\sigma = \{(2,1), (4,5)\}$ in $\triangle_{n}$.
\end{proof}

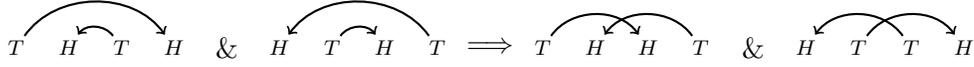
\begin{figure}
\begin{center}
\begin{tikzpicture}
\node[circle, inner sep = 0.9pt] (a1) at (0, 0) {$\scriptstyle{T}$};
\node[circle, inner sep = 0.9pt] (a2) at (0.7, 0) {$\scriptstyle{H}$};
\node[circle, inner sep = 0.9pt] (a3) at (1.4, 0) {$\scriptstyle{T}$};
\node[circle, inner sep = 0.9pt] (a4) at (2.1, 0) {$\scriptstyle{H}$};
\draw[->, thick] (a1) to[out=55, in=125] (a4);
\draw[<-, thick] (a2) to[out=55, in=125] (a3);

\node[circle, inner sep = 0.9pt] () at (2.8, 0) {\&};

\node[circle, inner sep = 0.9pt] (b1) at (3.5, 0) {$\scriptstyle{H}$};
\node[circle, inner sep = 0.9pt] (b2) at (4.2, 0) {$\scriptstyle{T}$};
\node[circle, inner sep = 0.9pt] (b3) at (4.9, 0) {$\scriptstyle{H}$};
\node[circle, inner sep = 0.9pt] (b4) at (5.6, 0) {$\scriptstyle{T}$};
\draw[<-, thick] (b1) to[out=55, in=125] (b4);
\draw[->, thick] (b2) to[out=55, in=125] (b3);

\node[circle, inner sep = 0.9pt] () at (6.3, 0) {$\Longrightarrow$};

\node[circle, inner sep = 0.9pt] (c1) at (7.0, 0) {$\scriptstyle{T}$};
\node[circle, inner sep = 0.9pt] (c2) at (7.7, 0) {$\scriptstyle{H}$};
\node[circle, inner sep = 0.9pt] (c3) at (8.4, 0) {$\scriptstyle{H}$};
\node[circle, inner sep = 0.9pt] (c4) at (9.1, 0) {$\scriptstyle{T}$};
\draw[->, thick] (c1) to[out=55, in=125] (c3);
\draw[<-, thick] (c2) to[out=55, in=125] (c4);

\node[circle, inner sep = 0.9pt] () at (9.8, 0) {\&};

\node[circle, inner sep = 0.9pt] (d1) at (10.5, 0) {$\scriptstyle{H}$};
\node[circle, inner sep = 0.9pt] (d2) at (11.2, 0) {$\scriptstyle{T}$};
\node[circle, inner sep = 0.9pt] (d3) at (11.9, 0) {$\scriptstyle{T}$};
\node[circle, inner sep = 0.9pt] (d4) at (12.6, 0) {$\scriptstyle{H}$};
\draw[<-, thick] (d1) to[out=55, in=125] (d3);
\draw[->, thick] (d2) to[out=55, in=125] (d4);
\end{tikzpicture}
\end{center}
\caption{A graphical representation of
Proposition~\ref{proposition_nest}.}
\label{figure_proposition_nest}
\end{figure}

\begin{proposition}
Let $\triangle_{n}$ be a uniform flag complex on the vertex set~$V_{n}$
representing a triangulation of the boundary~$\partial P_{n}$ 
of the root polytope for $n\geq 5$.
Assume exactly one of the $THTH$ and $HTHT$ type of pairs of arrows nest,
and the other type does not nest.
Also assume that both $THHT$ and $HTTH$ type of
pairs of arrows cross.
Then both $TTHH$ and $HHTT$
type of pairs nest.
\label{proposition_nest_follow}
\end{proposition}
\begin{proof}
Without loss of generality we may assume that $THTH$ type of arrows nest
and $HTHT$ type of pairs do not nest; the opposite case may be dealt
with by reversing all arrows. 
Assume that $HHTT$ type of arrows cross
and observe that 
both $\{(2,1),(4,3),(6,5)\}$ and $\{(2,5),(4,1),(6,3)\}$
form faces in the triangulation $\triangle_{n}$.
Note that this contradicts
the support axiom~(SA) in Theorem~\ref{theorem_legendreme}
and thus $HHTT$ type of arrows must nest.
A similar contradiction may be reached when $TTHH$ type of
arrows cross, by considering the faces
$\{(1,6),(3,2),(5,4)\}$ and $\{(1,4),(3,6),(5,2)\}$.
\end{proof}  

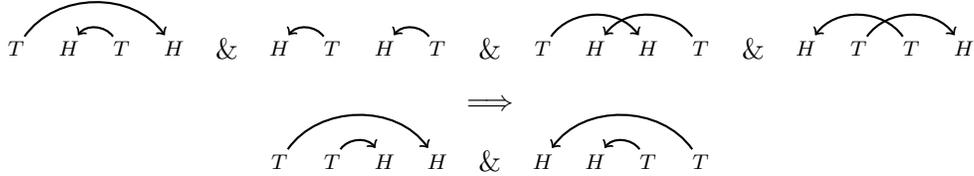
\begin{figure}
\begin{center}
\begin{tikzpicture}
\node[circle, inner sep = 0.9pt] (a1) at (0, 1.5) {$\scriptstyle{T}$};
\node[circle, inner sep = 0.9pt] (a2) at (0.7, 1.5) {$\scriptstyle{H}$};
\node[circle, inner sep = 0.9pt] (a3) at (1.4, 1.5) {$\scriptstyle{T}$};
\node[circle, inner sep = 0.9pt] (a4) at (2.1, 1.5) {$\scriptstyle{H}$};
\draw[->, thick] (a1) to[out=55, in=125] (a4);
\draw[<-, thick] (a2) to[out=55, in=125] (a3);

\node[circle, inner sep = 0.9pt] () at (2.8, 1.5) {\&};

\node[circle, inner sep = 0.9pt] (b1) at (3.5, 1.5) {$\scriptstyle{H}$};
\node[circle, inner sep = 0.9pt] (b2) at (4.2, 1.5) {$\scriptstyle{T}$};
\node[circle, inner sep = 0.9pt] (b3) at (4.9, 1.5) {$\scriptstyle{H}$};
\node[circle, inner sep = 0.9pt] (b4) at (5.6, 1.5) {$\scriptstyle{T}$};
\draw[<-, thick] (b1) to[out=55, in=125] (b2);
\draw[<-, thick] (b3) to[out=55, in=125] (b4);

\node[circle, inner sep = 0.9pt] () at (6.3, 1.5) {\&};

\node[circle, inner sep = 0.9pt] (c1) at (7.0, 1.5) {$\scriptstyle{T}$};
\node[circle, inner sep = 0.9pt] (c2) at (7.7, 1.5) {$\scriptstyle{H}$};
\node[circle, inner sep = 0.9pt] (c3) at (8.4, 1.5) {$\scriptstyle{H}$};
\node[circle, inner sep = 0.9pt] (c4) at (9.1, 1.5) {$\scriptstyle{T}$};
\draw[->, thick] (c1) to[out=55, in=125] (c3);
\draw[<-, thick] (c2) to[out=55, in=125] (c4);

\node[circle, inner sep = 0.9pt] () at (9.8, 1.5) {\&};

\node[circle, inner sep = 0.9pt] (d1) at (10.5, 1.5) {$\scriptstyle{H}$};
\node[circle, inner sep = 0.9pt] (d2) at (11.2, 1.5) {$\scriptstyle{T}$};
\node[circle, inner sep = 0.9pt] (d3) at (11.9, 1.5) {$\scriptstyle{T}$};
\node[circle, inner sep = 0.9pt] (d4) at (12.6, 1.5) {$\scriptstyle{H}$};
\draw[<-, thick] (d1) to[out=55, in=125] (d3);
\draw[->, thick] (d2) to[out=55, in=125] (d4);

\node[circle, inner sep = 0.9pt] () at (6.3, 0.75) {$\Longrightarrow$};

\node[circle, inner sep = 0.9pt] (e1) at (3.5, 0) {$\scriptstyle{T}$};
\node[circle, inner sep = 0.9pt] (e2) at (4.2, 0) {$\scriptstyle{T}$};
\node[circle, inner sep = 0.9pt] (e3) at (4.9, 0) {$\scriptstyle{H}$};
\node[circle, inner sep = 0.9pt] (e4) at (5.6, 0) {$\scriptstyle{H}$};
\draw[->, thick] (e1) to[out=55, in=125] (e4);
\draw[->, thick] (e2) to[out=55, in=125] (e3);

\node[circle, inner sep = 0.9pt] () at (6.3, 0) {\&};

\node[circle, inner sep = 0.9pt] (f1) at (7.0, 0) {$\scriptstyle{H}$};
\node[circle, inner sep = 0.9pt] (f2) at (7.7, 0) {$\scriptstyle{H}$};
\node[circle, inner sep = 0.9pt] (f3) at (8.4, 0) {$\scriptstyle{T}$};
\node[circle, inner sep = 0.9pt] (f4) at (9.1, 0) {$\scriptstyle{T}$};
\draw[<-, thick] (f1) to[out=55, in=125] (f4);
\draw[<-, thick] (f2) to[out=55, in=125] (f3);

\end{tikzpicture}
\end{center}
\caption{A graphical representation of
the first case of Proposition~\ref{proposition_nest_follow}.}
\label{figure_proposition_nest_follow}
\end{figure}

We conclude this section by an observation that we will frequently use
in our proofs in the situations when we need to consider only forward
arrows or only backward arrows.

\begin{theorem}
\label{theorem_forward}  
Let  $\triangle_{n}$ be a permissible uniform flag complex on the vertex
set~$V_{n}$, satisfying the condition that $THTH$ type of
pairs of arrows do not nest. Then the restriction $\triangle_{n}^{+}$ of
$\triangle_{n}$ to the set of forward arrows represents a triangulation of
the union of those boundary facets of the positive polytope~$P^{+}_{n}$
that do not contain the origin. 
\end{theorem}  
\begin{proof}
Since all the arrows are forward arrows and any
non-incident pair of such arrows is of the type $TTHH$ or $THTH$,
the face structure of $\triangle_{n}^{+}$ depends
on the rules associated with $TTHH$ and $THTH$.
If such arrow pairs cross
then the faces of any $\rest{\triangle_{n}^{+}}{I \times J}$ are all sets of
forward arrows  
$\{(i_{1},j_{1}),(i_{2},j_{2}), \ldots, (i_{k},j_{k})\}$
where
$i_{1} < i_{2} < \cdots < i_{k}$ and
$j_{1} < j_{2} < \cdots < j_{k}$ hold.
If such pairs of arrows nest 
then the faces of $\rest{\triangle_{n}}{I \times J}$ are all sets of forward arrows
$\{(i_{1},j_{1}),(i_{2},j_{2}), \ldots, (i_{k},j_{k})\}$ such that
$i_{1} < i_{2} < \cdots < i_{k}$ and
$j_{1} > j_{2} > \cdots > j_{k}$ hold.
In other words, we obtain the well-known lexicographic and
revlex pulling triangulations of~$P^{+}_{n}$. 
\end{proof}

\section{Verifying the sufficiency part of the classification}
\label{section_sufficiency}

In this section we show that any uniform flag complex $\triangle_{n}$ on $V_{n}$ that
satisfies one of the sets of criteria listed in
Theorem~\ref{theorem_main} represents a triangulation of the boundary
$\partial P_{n}$ of the root polytope. We do this by verifying in each
case that the support and linkage axioms of
Theorem~\ref{theorem_legendreme} are satisfied. 

To verify the support axiom (SA), we will list the elements of
the disjoint union $I \cup J$
in order, marking each element of~$I$ with a $T$ (tail) and each
element of~$J$ with a $H$ (head). Thus we obtain a word containing the
same number of letters $T$ and $H$. We also associate 
to each tail an up $(1,1)$ step, and to
each head a down $(1,-1)$ step. These steps yield a lattice path
starting at the origin and ending on the $x$-axis. We will describe the
unique matching contained in $\triangle_{n}$ whose tail set is $I$ and head
set is $J$ in terms of this {\em associated $TH$-word} and {\em
  associated lattice path}.

During the verification of the support axiom (SA) we will often treat
forward and backward arrows separately. This separation leads to
considering special associated $TH$-words and lattice paths. 
\begin{definition}
We call a lattice path consisting of up $(1,1)$ steps and down $(1,-1)$
steps a {\em lower (upper) Dyck path} if it starts and ends on the $x$-axis and
never goes above (below) the $x$-axis. We call a $TH$-word a lower
(upper) Dyck word if replacing each $T$ with an up step and each $H$
with a down step results in a lower (upper) Dyck path.  
\end{definition}  
The following lemma is straightforward.
\begin{lemma}
\label{lemma_backward_only}
For any uniform flag complex on $V_{n}$ and any face $\sigma\subset V_{n}$
that is a matching and consists only of backward arrows the lattice path
associated to the tail set $I$ and head set $J$ is a lower Dyck~path. 
\end{lemma}  
Indeed, as we read the list of heads and tails from left to right, the number
of tails cannot exceed the 
number of heads. The next lemma is a partial converse of
Lemma~\ref{lemma_backward_only}.

\begin{lemma}
Let $\triangle_{n}$ be a uniform flag complex on $V_{n}$ that satisfies the
necessary conditions stated in Theorem~\ref{theorem_main} and has the
property that $HTHT$ type of pairs of arrows do not nest.
Let $I$ and $J$ be two sets satisfying
$I, J \subset \{1,2, \ldots, n+1\}$,  $I\cap J=\emptyset$
and $|I| = |J| \neq 0$.
Assume that the $TH$-word $w$ associated with the two sets $I$ and $J$ is a
lower Dyck word. Then there is a unique matching contained in
$\triangle_{n}$ that matches $I$ to~$J$. Furthermore, this matching
consists of backward arrows only. 
\label{lemma_backward_only_converse}
\end{lemma}  
\begin{proof}
First we show that no matching from $I$ to $J$ contains a forward
arrow. Assume by way of contradiction that there is a smallest
counterexample to this statement, and let $(i,j)$ be the forward arrow
with the smallest tail $i$ in such an example. The associated
lattice path must start with a down step, hence the least element of
$I\cup J$ is a head $j_{1}$, the head of a backward arrow $(i_{1},j_{1})$. If
$i_{1}<i$ holds then the removal of the arrow $(i_{1},j_{1})$ yields a smaller
counterexample, contradicting the assumption of minimality. We
cannot have $i_{1}>j$ either as $HTHT$ type of pairs of arrows do not
nest. Hence we have $j_{1}<i<i_{1}<j$ and $HTTH$ type of pairs of arrows
cross. The same reasoning may be repeated for any backward arrow whose
head is to the left of $i$: the tails of these arrows are all between
$i$ and~$j$, in particular no tail of a backward arrow is to the left of
$i$. By the choice of $i$, there is no tail of a forward arrow to the
left of $i$ either: the first up step in the associated lattice path is
contributed by the tail $i$. Hence $i$ is preceded by $k\geq 1$ heads:
$j_{1}, j_{2}, \ldots,  j_{k}$, and the tails $i_{1}, i_{2}, \ldots, i_{k}$ of these
backward arrows all occur before $j$. The associated lattice path goes
above the $x$-axis before the down step associated to $j$ unless there
is a backward arrow $(i',j')$ whose head $j'$ occurs before~$j$, while
$i'$ occurs only after it. The pair $\{(i,j),(i',j')\}$ is a crossing
$THHT$ type of pair of arrows. Since both $THHT$ and $HTTH$ type of pairs
cross, the complex $\triangle_{n}$ cannot belong to the lex class. It cannot
belong to the revlex class either because $HTHT$ type of pairs of
arrows do not nest by our assumption. We are left with the possibility
that $HTHT$ type of pairs do not nest and $THTH$ type of pairs nest. By
Proposition~\ref{proposition_nest_follow} $HHTT$ type of pairs nest. As a
consequence no backward arrow $(i',j')$ satisfying $j' < j < i'$ can cross
any arrow $(i_{s},j_{s})$ satisfying $j_{s} < i < i_{s} < j$. But then the associated
lattice path goes above the $x$-axis at the step associated to
$\max(i_{1},i_{2}, \ldots, i_{k})$ and we obtain a contradiction.

Having established that no matching can contain a forward arrow, we may
show the existence of a unique matching by induction. Regardless on the
condition on the $HHTT$ type of pairs of arrows, the first step of the
associated lattice path is a down step, corresponding to a head~$j_{1}$. We
only need to show that there is a unique way to identify the tail
$i_{1}$ of this arrow, and that the removal of the steps associated
to $j_{1}$ and $i_{1}$ results in a lattice path that does not go above the
horizontal axis. In the case when $HHTT$ type of pairs nest, $i_{1}$ must be
the tail marking the first return to the horizontal axis, because all
arrows whose head is between $j_{1}$ and $i_{1}$ must also have their tail
in the same interval as backward arrows cannot cross. The removal of
the first down step and the first return to the $x$-axis yields a
lattice path that does not go above the $x$-axis. Finally, in the case
when $HHTT$ type of pairs cross, $i_{1}$ must be the least tail, marking the
first up step, as any backward arrow whose tail precedes $i_{1}$ would form a
$HHTT$ type of nesting pair with $(i_{1},j_{1})$. The removal of the first up
step and the first down step yields once again a lattice path that does
not go above the $x$-axis.   
\end{proof}  

\begin{remark}
{\em
The  proof of Lemma~\ref{lemma_backward_only_converse} defines the
matchings induced by the associated lattice paths in a recursive
fashion. It is not difficult to prove the following explicit
rules by induction:
\begin{enumerate}  
\item
If $HHTT$ type of arrows nest then each head~$j$, representing a down
step, is matched to the tail~$i$ representing the first return to the
same level.
\item
If $HHTT$ type of arrows cross then the $k$th head in the left to
right order is matched to the $k$th tail in the left to right order.  
\item
In either case, each head $j$ is matched to a tail $i$ in such a
way that there is no return to the $x$-axis strictly between the down
step associated to $j$ and the up step associated to $i$. 
\end{enumerate}  
}
\label{remark_backward_only}
\end{remark}

\begin{lemma}
Assume the same conditions hold as in
Lemma~\ref{lemma_backward_only_converse},
but with the extra condition that the associated word $w$
factors as a product of two lower Dyck words $w_{1}$ and $w_{2}$.
Then the matching as a graph is a disjoint union
of the two matchings for $w_{1}$, respectively $w_{2}$.
\label{lemma_lower_Dyck_word}
\end{lemma}  
\begin{proof}
The matching obtained by taking the union
of the two matchings for $w_{1}$, respectively $w_{2}$,
satisfies all the conditions. Hence by uniqueness the result follows.
\end{proof}

\begin{remark}
{\em
We will use Lemmas~\ref{lemma_backward_only} and
\ref{lemma_backward_only_converse} for backward arrows most of the time.
The reader should note that the dual statements, for forward arrows
and associated lattice paths being upper Dyck paths, also hold.
}
\end{remark}
These lemmas were about the case when $THTH$ type of arrows do not nest. We
will consider the case when they nest in the dual form below.

\begin{lemma}
\label{lemma_forward_only}
Let $\triangle_{n}$ be a uniform flag complex on $V_{n}$ that satisfies the
necessary conditions stated in Theorem~\ref{theorem_main} and has the
property that $HTHT$ type of pairs of arrows nest.
Let $I$ and $J$ be two sets satisfying
$I,J \subset \{1,2, \ldots, n+1\}$,  $I\cap J=\emptyset$
and $|I| = |J| \neq 0$.
There is a matching in $\triangle_{n}$ consisting of forward arrows only
that matches $I$ to $J$ if and only if every tail precedes every head,
that is, $i<j$ holds for all $i \in I$ and $j \in J$. Furthermore, if
every tail precedes every head in $I\cup J$ then the matching contained
in $\triangle_{n}$ that matches $I$ to $J$ is unique.
\end{lemma}  
\begin{proof}
It is a direct consequence of the condition on $THTH$ type of pairs of
arrows that all tails must precede all heads in every face that consists
of forward arrows only. Conversely assume that $I=\{i_{1},i_{2}, \ldots, i_{k}\}$
and $J=\{j_{1},j_{2}, \ldots, j_{k}\}$ satisfy
$i_{1} < i_{2} < \cdots < i_{k}$ and $i_{k} < j$ for
all $j\in J$. Clearly, any arrow whose tail is in $I$ and whose head is
in $J$ is a forward arrow. Any pair of such arrows is a $TTHH$ type of pair.
Hence the only matching contained in $\triangle_{n}$ is
$\{(i_{1},j_{1}),(i_{2},j_{2}), \ldots, (i_{k},j_{k})\}$ where $j_{1}<j_{2}<\cdots <j_{k}$
holds if $TTHH$ type of pairs cross and $j_{1}>j_{2}>\cdots >j_{k}$ holds when
$TTHH$ type of pairs nest.  
\end{proof}

The verification of the linkage axiom (LA) is
facilitated by the following observations.

\begin{lemma}
Let $\triangle_{n}$ be any permissible uniform flag complex on $V_{n}$. Then the
complex $\triangle_{n}$ satisfies the relevant part of the linkage axiom~(LA)
when $k$ is inserted as a tail and there exists an $i \in I$ 
such that there is no element of $I \cup J$ strictly between $k$ and $i$.
Similarly,
when $k$ is inserted as a head and there exists a $j \in J$
such that there is no element of $I \cup J$ strictly between $k$ and $j$,
the linkage axiom is satisfied.
\label{lemma_LA}
\end{lemma}
In other words, if we insert a new tail in position $k$ next to a tail,
then we may extend the arrow containing the old tail to contain the new
tail instead. A similar observation can be made about inserting
a new head. We only need to verify the linkage axiom (LA) in the cases
when a new head is inserted between two tails (or as the least or
largest node, next to a tail), and when a new tail is inserted
between two heads (or as the least or largest node, next to a head).

\begin{lemma}
Let $\triangle_{n}$ be any uniform flag complex on $V_{n}$ satisfying the
necessary conditions stated in Theorem~\ref{theorem_main}. Assume that
$HTHT$ type of arrows do not nest.
Let $I$ and $J$ be two sets
satisfying $I, J \subset \{1,2, \ldots, n+1\}$,  $I\cap J=\emptyset$,
and $|I|=|J|\neq 0$.
Assume that $TH$-word associated with the two sets $I$ and $J$ is a
lower Dyck word. 
If $k > \min(I \cup J)$ is inserted as a tail or $k < \max(I \cup J)$ and $k$ is
inserted as a head, then the relevant part of the linkage axiom is
verified in such a way that the new arrow is also a backward arrow.
\label{lemma_LA_backward}
\end{lemma}
\begin{proof}
By symmetry, it is enough to consider the case when $k$ is inserted as a tail.
By Lemma~\ref{lemma_LA} 
we can assume that $k$ is inserted
after head~$j$ of an arrow $(i,j)$.
If $HHTT$ type of pairs nest then we match $j$ to $k$ and we remove the tail $i$.
Note that the new arrow does not introduce any crossings.
If $HHTT$ type of pairs cross then consider the least tail $i'>k$,
which is the tail of an arrow $(i',j')$.
Note that $j' \leq j$ since otherwise there would be two nesting
backward arrows.
Match $k$ to $j'$ and remove the tail~$i'$.
By Remark~\ref{remark_backward_only}  part (2),
it is straightforward to see that we obtain a matching belonging to $\triangle_{n}$.
Note that in both cases, all the new arrows are backward arrows.
\end{proof}

\subsection{The lex class}
\label{subsection_nonest}

In this subsection we show that all four triangulations in the lex class
are possible.

\begin{theorem}
\label{theorem_nonest}  
Let  $\triangle_{n}$ be a permissible uniform flag complex on the vertex
set~$V_{n}$ that has the property
that both $THTH$ and $HTHT$ types of arrows do not nest.
Then the complex~$\triangle_{n}$ represents a
triangulation of the boundary~$\partial P_{n}$ of the root polytope. 
\end{theorem}  
\begin{proof}
We begin by observing that Proposition~\ref{proposition_nonest}
implies that
$THHT$ and $HTTH$ types of arrows do not cross.
Next we verify the support and linkage axioms
of Theorem~\ref{theorem_legendreme}.

To verify the support axiom (SA), 
mark maximal runs of the associated lattice path
that are above the $x$-axis with forward arrows.
Similarly,
mark the maximal runs of the lattice path
that are below the $x$-axis with backward arrows. The fact that no end
of a forward arrow can occur between the head and tail of a backward
arrow follows from the fact that arrows of opposite direction neither
cross nor nest. By the same reason no tail of a backward arrow
can occur between the head and tail of a forward arrow. 
Within each maximal run, match the heads $H$ and the
tails $T$ according to the $TTHH$ and $HHTT$ rules.
By Lemmas~\ref{lemma_backward_only}, \ref{lemma_backward_only_converse}
and their duals, this can be done
in a unique way.
Furthermore, note that
arrows of opposite direction do not
cross and do not nest.
This shows that there is a unique way to obtain a matching between
the set of tails~$I$ and the set of heads~$J$.

To verify the linkage axiom (LA),
we may by symmetry assume that
we are  inserting a new tail at position $k$.
Denote this node by $T'$.
If the new tail is adjacent to an old tail,
we are done by
Lemma~\ref{lemma_LA}.
If a new tail $T'$ is inserted in between two heads
we have three possible cases.
If the two heads are part of a $THHT$ pattern
then we are between two maximal runs as described above.
Take any of the two heads adjacent
to the inserted element, unlink it from its pair and link it to the
inserted tail;
see the first line of Figure~\ref{figure_proof_of_nonest}.
If the two heads are part of a $HHTT$ pattern, there are two
subcases depending upon whether the two arrows nest or cross,
and these two cases are explained in the second and third line of
Figure~\ref{figure_proof_of_nonest}.
In each of these two subcases, one can verify that the new matching
is in fact a face.
Finally, 
if the two heads are part of a $TTHH$ pattern,
it is mirror symmetric to the previous case.
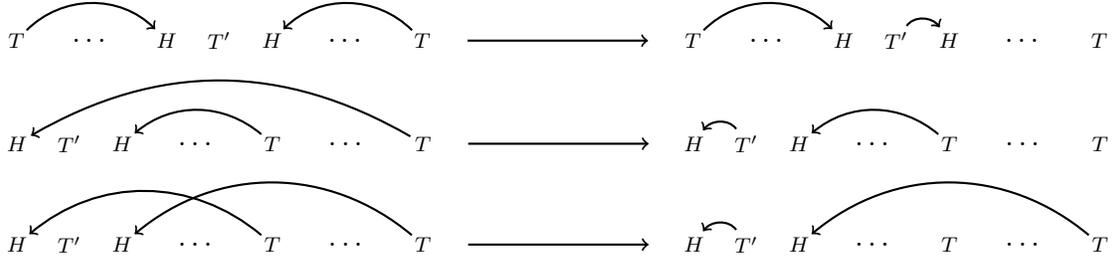
\begin{figure}
\begin{center}
\begin{tikzpicture}
\node[circle, inner sep = 0.9pt] (a1) at (0, 0) {$\scriptstyle{T}$};
\node[circle, inner sep = 0.9pt] (a2) at (2, 0) {$\scriptstyle{H}$};
\node[circle, inner sep = 0.9pt] (a3) at (2.7, 0) {$\scriptstyle{T'}$};
\node[circle, inner sep = 0.9pt] (a4) at (3.4, 0) {$\scriptstyle{H}$};
\node[circle, inner sep = 0.9pt] (a5) at (5.4, 0) {$\scriptstyle{T}$};
\draw[->, thick] (a1) to[out=45, in=135] (a2);
\draw[<-, thick] (a4) to[out=45, in=135] (a5);
\node at (1.0, 0) {\ldots};
\node at (4.4, 0) {\ldots};

\draw[->, thick] (6.0,0) -- (8.4,0);

\node[circle, inner sep = 0.9pt] (b1) at (9, 0) {$\scriptstyle{T}$};
\node[circle, inner sep = 0.9pt] (b2) at (11, 0) {$\scriptstyle{H}$};
\node[circle, inner sep = 0.9pt] (b3) at (11.7, 0) {$\scriptstyle{T'}$};
\node[circle, inner sep = 0.9pt] (b4) at (12.4, 0) {$\scriptstyle{H}$};
\node[circle, inner sep = 0.9pt] (b5) at (14.4, 0) {$\scriptstyle{T}$};
\draw[->, thick] (b1) to[out=45, in=135] (b2);
\draw[->, thick] (b3) to[out=55, in=125] (b4);
\node at (10.0, 0) {\ldots};
\node at (13.4, 0) {\ldots};
\end{tikzpicture}
\\
\begin{tikzpicture}
\node[circle, inner sep = 0.9pt] (a1) at (0, 0) {$\scriptstyle{H}$};
\node[circle, inner sep = 0.9pt] (a2) at (0.7, 0) {$\scriptstyle{T'}$};
\node[circle, inner sep = 0.9pt] (a3) at (1.4, 0) {$\scriptstyle{H}$};
\node[circle, inner sep = 0.9pt] (a4) at (3.4, 0) {$\scriptstyle{T}$};
\node[circle, inner sep = 0.9pt] (a5) at (5.4, 0) {$\scriptstyle{T}$};
\draw[<-, thick] (a1) to[out=30, in=150] (a5);
\draw[<-, thick] (a3) to[out=40, in=140] (a4);
\node at (2.4, 0) {\ldots};
\node at (4.4, 0) {\ldots};

\draw[->, thick] (6.0,0) -- (8.4,0);

\node[circle, inner sep = 0.9pt] (b1) at (9, 0) {$\scriptstyle{H}$};
\node[circle, inner sep = 0.9pt] (b2) at (9.7, 0) {$\scriptstyle{T'}$};
\node[circle, inner sep = 0.9pt] (b3) at (10.4, 0) {$\scriptstyle{H}$};
\node[circle, inner sep = 0.9pt] (b4) at (12.4, 0) {$\scriptstyle{T}$};
\node[circle, inner sep = 0.9pt] (b5) at (14.4, 0) {$\scriptstyle{T}$};
\draw[<-, thick] (b1) to[out=55, in=125] (b2);
\draw[<-, thick] (b3) to[out=40, in=140] (b4);
\node at (11.4, 0) {\ldots};
\node at (13.4, 0) {\ldots};
\end{tikzpicture}
\\
\begin{tikzpicture}
\node[circle, inner sep = 0.9pt] (a1) at (0, 0) {$\scriptstyle{H}$};
\node[circle, inner sep = 0.9pt] (a2) at (0.7, 0) {$\scriptstyle{T'}$};
\node[circle, inner sep = 0.9pt] (a3) at (1.4, 0) {$\scriptstyle{H}$};
\node[circle, inner sep = 0.9pt] (a4) at (3.4, 0) {$\scriptstyle{T}$};
\node[circle, inner sep = 0.9pt] (a5) at (5.4, 0) {$\scriptstyle{T}$};
\draw[<-, thick] (a1) to[out=40, in=140] (a4);
\draw[<-, thick] (a3) to[out=40, in=140] (a5);
\node at (2.4, 0) {\ldots};
\node at (4.4, 0) {\ldots};

\draw[->, thick] (6.0,0) -- (8.4,0);

\node[circle, inner sep = 0.9pt] (b1) at (9, 0) {$\scriptstyle{H}$};
\node[circle, inner sep = 0.9pt] (b2) at (9.7, 0) {$\scriptstyle{T'}$};
\node[circle, inner sep = 0.9pt] (b3) at (10.4, 0) {$\scriptstyle{H}$};
\node[circle, inner sep = 0.9pt] (b4) at (12.4, 0) {$\scriptstyle{T}$};
\node[circle, inner sep = 0.9pt] (b5) at (14.4, 0) {$\scriptstyle{T}$};
\draw[<-, thick] (b1) to[out=55, in=125] (b2);
\draw[<-, thick] (b3) to[out=40, in=140] (b5);
\node at (11.4, 0) {\ldots};
\node at (13.4, 0) {\ldots};
\end{tikzpicture}
\end{center}
\caption{Verifying the linkage axiom
in the lex class of triangulations.}
\label{figure_proof_of_nonest}
\end{figure}
\end{proof}

\subsection{The revlex class}
\label{subsection_nest}

We now turn our attention to the
revlex class and show that all four triangulations are
possible.

\begin{theorem}
\label{theorem_nest}  
Let  $\triangle_{n}$ be a permissible uniform flag complex on the vertex
set~$V_{n}$ that has the properties
that both $THTH$ and $HTHT$ types of arrows nest.
Then the complex~$\triangle_{n}$ represents a
triangulation of the boundary~$\partial P_{n}$ of the root polytope. 
\end{theorem}  
\begin{proof}
By Proposition~\ref{proposition_nest} 
we conclude that $THHT$ and $HTTH$ types of arrows cross.
Hence
if we disregard the direction of the arrows,
we have that
the pattern
\begin{tikzpicture}
\node[fill, circle, inner sep = 0.01pt] (a1) at (0, 0) {.};
\node[fill, circle, inner sep = 0.01pt] (a2) at (0.5, 0) {.};
\node[fill, circle, inner sep = 0.01pt] (a3) at (1.0, 0) {.};
\node[fill, circle, inner sep = 0.01pt] (a4) at (1.5, 0) {.};
\draw[-, thick] (a1) to[out=30, in=150] (a2);
\draw[-, thick] (a3) to[out=30, in=150] (a4);
\end{tikzpicture}
cannot occur. Thus all pairs of arrows must either nest or cross.

The fact that every pair of arrows nests or crosses
implies that all arrows must arch over the
{\em midpoint} of the set $I \cup J$, that is, the point that has the same number of
elements of $I \cup J$ to the left of it as the number of such elements
to the right. For example, for the word $TTHT|HHTH$, the midpoint is
marked with a vertical bar. Hence the number of tails to the left of the
midpoint must equal the number of heads to the right of it. The tails to
the left of the midpoint are tails of the forward arrows. They must be
matched with the heads to the right of the midpoint. The analogous
statement is true for the heads and tails of the backward arrows.
Now match the left of midpoint tails with the right of the midpoint heads
according to the $TTHH$ rule.
Similarly, match the right of midpoint tails with the left of the midpoint heads
according to the $HHTT$ rule.
Both of these matchings are unique, proving the support axiom.

To verify the linkage axiom~(LA),
by symmetry
it is enough to verify
the linkage axiom after inserting a new tail at position~$k$, denoted by~$T'$.
Assume that there another tail $T$ of an arrow~$(i,j)$
in the set of tails $I$ adjacent to $T'$
and that the position $i$ is on the same side of the midpoint
as position~$k$.
Then we can replace the arrow $(i,j)$ with~$(k,j)$.
If there is no such arrow~$(i,j)$,
there is no tail on the same side as~$T'$,
that is, the situation is
$T \cdots T | H H \cdots H T' H \cdots H$
or
$H \cdots H T' H \cdots H H | T \cdots T$.
These two possibilities are symmetric, so it enough to consider
the first one.
Note that all the arrows are forward arrows.
Let $(i,j)$ be the arrow with the smallest value of $j$,
that is, the arrow attached to the first $H$.
Replace the arrow $(i,j)$ with the new arrow $(k,j)$.
Note that this yields the new word
$T \cdots T H | H \cdots H T' H \cdots H$,
where there is now an $H$ on the left of the new midpoint.
This completes the verification of the linkage axiom.
\end{proof}

\subsection{The Simion class of triangulations}
\label{subsection_other}

In this section we conclude the proof of Theorem~\ref{theorem_main}. The
remaining case is the Simion class of triangulations.
\begin{theorem}
Let  $\triangle_{n}$ be a permissible uniform flag complex on the vertex
set~$V_{n}$ 
that has the properties
that exactly one of the types $THTH$ and $HTHT$ do nest.
Then the complex~$\triangle_{n}$ represents a
triangulation of the boundary~$\partial P_{n}$ of the root polytope. 
\label{theorem_existence_simion_class}
\end{theorem}

Using the involution
$\triangle_{n} \longmapsto \triangle_{n}^{*}$
it is enough to consider the cases
where
$THTH$ type of arrows nest
and
$HTHT$ type of arrows do not nest
in the following propositions.
We begin by considering all four flag complexes 
in the Simion subclass of type~$a$.

\begin{proposition}
Let  $\triangle_{n}$ be a permissible uniform flag complex on the vertex
set~$V_{n}$ that has the properties
that the $THTH$ type of arrows nest, the $HTHT$ type of arrows do not 
nest, and both $THHT$ and $HTTH$ types of arrows do not cross.
Then the complex~$\triangle_{n}$ represents a
triangulation of the boundary~$\partial P_{n}$ of the root polytope. 
\label{proposition_existence_simion_subclass_a}
\end{proposition}
\begin{proof}
Let $w$ be the associated $TH$-word to the two sets $I$ and $J$.
Factor the word $w$ as follows 
$$ w = w_{1} T \cdots T w_{h} T w_{h+1} H w_{h+2} H \cdots H w_{2h+1} $$
where the factors $w_{1}, w_{2}, \ldots, w_{2h+1}$
are lower Dyck words.
Note that such a factorization exists and is unique
since the $h$ $T$s in the expression correspond
to left-to-right maxima of the lattice path. Similarly,
the $h$ $H$s correspond to right-to-left maxima.
Create a matching from $I$ to $J$ by
making $h$ forward arrows from the $h$ $T$s
to the $h$ $H$s, according to the $TTHH$ rule.
Finally, apply
Lemma~\ref{lemma_backward_only_converse}
to each factor $w_{r}$ to create a matching consisting
of only backward arrows.
It is straightforward to see that there is no crossing
between a forward and a backward arrow, that is,
the $THHT$ and $HTTH$ conditions hold.
Furthermore, no backward arrow nests a forward arrow,
the $HTHT$ condition follows
and
finally no two forward arrows follow each other,
so the $THTH$ condition is also true.

Next we show the uniqueness part of the support axiom. Assume a
matching from $I$ to~$J$ contains $h$ forward arrows. Since $THTH$ type of
pairs nest, the tails of all forward arrows precede all heads. 
The heads and tails of the forward  arrows partition the number line
into $2h+1$ segments.
Since $THHT$ and $HTTH$ types of pairs do not cross
and $HTHT$ do not nest,
backward arrows that have one end in one of these line
segments have their other end in the same line segment.
By Lemma~\ref{lemma_backward_only} the part of the lattice path associated
to all backward arrows in one of these line segments represent a
lattice path starting and ending at the same level and never going above
the level where it started. Hence the tails of the forward arrows mark
the first ascents to level $1$, $2$, $\ldots$, $h$ and the heads of the
forward arrows mark the last descents to levels $h-1$, $h-2$, $\ldots$,
$0$. This observation shows the unique determination of the endpoints of
the forward arrows. By Lemma~\ref{lemma_forward_only} there is a unique
way to match the heads and tails of the forward arrows, and by
Lemma~\ref{lemma_backward_only_converse} there is a unique way to match
the heads and tails of the backward arrows within each segment created
by the endpoints of the forward arrows.

To verify the linkage axiom~(LA), note that
Lemma~\ref{lemma_LA_backward} is applicable unless $k$ is inserted as a
tail at the beginning of a run of backward arrows or as a head at the
end of such a run. Assume $k$ is inserted as a tail; the case when $k$
is inserted as a head is completely analogous. If $k$ is inserted right
after a tail, then we are done by Lemma~\ref{lemma_LA}. If $k$ is inserted right
after a head $j$, then this head is necessarily the head of a forward
arrow $(i,j)$. In this case insert the backward arrow $(k,j)$ and remove the
arrow $(i,j)$. Note that this results in a matching in $\triangle_{n}$ in
which there is one less forward arrow and the arrow $(k,j)$ becomes part
of the run of backward arrows immediately succeeding it. We are left
with the case when the inserted tail $k$ satisfies $k<\min(I\cup J)$. If
the current matching contains at least one forward arrow, then we select
the forward arrow $(i,j)$ with the smallest $i$. We remove the arrow
$(i,j)$ and add the arrow $(k,j)$. This move does not change the
crossing or nesting properties of arrows of the same direction, nor does
it create crossing arrows of opposite directions. Finally, if the
current matching on $I\cup J$ consists of backward arrows only, then we
associate an initial $NE$ step to $k$, and continue with the lattice
path associated to $I\cup J$ which now starts and ends at level $1$ and
never goes above that level. Let $j\in J$ be the head associated to the
last descent from level $1$ to level $0$. This is the head of a backward
arrow $(i,j)$. Removing $(i,j)$ and adding $(k,j)$ results in a valid
matching because of part (3) of Remark~\ref{remark_backward_only}. 
\end{proof}

We next turn our attention to the Simion subclass of type~$c$.
\begin{proposition}
Let  $\triangle_{n}$ be a permissible uniform flag complex on the vertex
set~$V_{n}$ that has the properties
that the $THTH$ type of arrows nest, the $HTHT$ type of arrows do not
nest, both $THHT$ and $HTTH$ types of arrows cross
and
both $TTHH$ and $HHTT$ type of arrows nest.
Then the complex~$\triangle_{n}$ represents a
triangulation of the boundary~$\partial P_{n}$ of the root polytope. 
\label{proposition_existence_simion_subclass_c}
\end{proposition}
\begin{proof}
Let $w$ be the associated $TH$-word to the two sets $I$ and $J$
and assume that the associated lattice path reaches height~$h$.
That is, there are at least $h$~ascents ($T$) before the path reaches
height~$h$, and at least $h$~descents ($H$) since it leaves height $h$ the
last time.
Hence we factor the word $w$ uniquely as
\begin{align*}
w
& =
H^{m_{1}} T H^{m_{2}} T \cdots T H^{m_{h}} T
\cdot u \cdot
H T^{n_{1}} H T^{n_{2}} H \cdots H T^{n_{h}}
\end{align*}
Using the rule for $TTHH$ (nest),
match the first $h$ tails in this expression with the
last $h$ heads to create $h$ pairwise nesting forward arrows.
The remaining nodes contribute the subword
$w^{\prime}
=
H^{m_{1}+m_{2}+\cdots+m_{h}}
\cdot u \cdot
T^{n_{1}+n_{2}+\cdots+n_{h}}$.
Note that this is a lower Dyck word.
Thus apply 
Lemma~\ref{lemma_backward_only_converse}
to this word to create the remaining arrows,
which are all backward arrows.
Directly by the construction
rules hold for types $THTH$, $TTHH$ and $HHTT$.
Next notice that $HTTH$ must cross, since all tails
of backward arrows occur after the tails of forward arrows.
By the symmetric argument $THHT$ must also cross.
The final condition for $HTHT$
is that no backward arrow can nest a forward arrow.
Recall that the word $w$
reaches its maximum height somewhere in the factor $u$.
Hence the word $w^{\prime}$ reaches the maximum of $0$
also in the factor $u$.
Thus $w^{\prime}$ can be factored into two lower Dyck words
$w^{\prime}
=
(H^{m_{1}+m_{2}+\cdots+m_{h}} \cdot u_{1})
\cdot
(u_{2} \cdot T^{n_{1}+n_{2}+\cdots+n_{h}})$.
By Lemma~\ref{lemma_lower_Dyck_word}
each backward arrow has either its tail or head (or both)
in the word $u_{1} \cdot u_{2}$. Hence it lies between
the tail and and head of every forward arrow
and last condition is proved.

Next we prove uniqueness.
Assume that the associated lattice path reaches height $h$.
Observe that the conditions imply that all the forward arrows must nest.
Let $(i^{\prime},j^{\prime})$
be the innermost forward arrow.
Since no backward arrow nests, whether in front or behind this shortest forward
arrow, each backward arrow must have either its tail or head (or both)
in the interval from $i^{\prime}$ to $j^{\prime}$.
Let $h'$ be the number of forward arrows.
Hence the first $h'$ $T$s in the associated $TH$-word
correspond to the tails of the forward arrows.
Similarly, the last $h'$ $H$s correspond to the heads of
these forward arrows. 
Since the backward arrows nest. Consider the backward arrow $(i'',j'')$
with the smallest head $j''$. If $j'' < i'$ then there is no backward
arrow above the tail $i''$, only $h'$ forward arrows.
Similarly, if there is no such backward arrow then there
is no backward arrow above $i'$, and only $h'$ forward arrows.
In both cases we obtain that the maximal height $h$ is $h'$.
This agrees with the construction in the previous paragraph.
The uniqueness in Lemma~\ref{lemma_backward_only_converse}
implies that the matching from $I$ to $J$ is unique.

Before proving the linkage axiom recall that the $TH$-word $w$
has the following factorization
\begin{align*}
w
& =
\underbrace{H^{m_{1}} T H^{m_{2}} T \cdots T H^{m_{h}} T}_{1\text{st factor}}
\cdot
z
\cdot
\underbrace{H T^{n_{1}} H T^{n_{2}} H \cdots H T^{n_{h}}}_{5\text{th factor}}
, \\
\intertext{where}
z
& =
\underbrace{u_{1} T u_{2} T \cdots u_{m} T}_{2\text{nd factor}}
\cdot
\underbrace{x}_{3\text{rd factor}}
\cdot
\underbrace{H v_{1} H v_{2} \cdots H v_{n}}_{4\text{th factor}} ,
\end{align*}
where
$m = \sum_{i=1}^{h} m_{i}$,
$n = \sum_{i=1}^{h} n_{i}$
and
$u_{1}, u_{2}, \ldots, u_{m}, x, v_{1}, v_{2},  \ldots, v_{n}$
are all lower Dyck words.
Consider the case when we insert a new tail~$T'$.
By Lemma~\ref{lemma_LA_backward}
it can be inserted in the one of the lower Dyck words
$u_{1}, \ldots, u_{m}, x, v_{1},  \ldots, v_{n}$
or immediately after one of these words.
If inserted in the $1$st factor,
or immediately afterwards,
we can move one of the tails
of the forward arrows.
The case of insertion in the $2$nd or $3$rd factor
is already taken care of.
In the $4$th factor, insertion in front of a $v_{i}$ must occur
immediately after 
a head, which can be used for the switch.
Finally, in the $5$th factor 
when $n > 0$ we can move one of tails of the closest backward edge
from the $5$th factor to the $4$th factor.
When $n=0$, remove the innermost forward arrow
and use its head. Observe that this is only case that changes
the number of forward arrows. This completes the proof of the linkage axiom (LA).
\end{proof}

Finally, we consider half of the cases in
the Simion subclass of type~$b$.
The proof is a mixture of the two previous proofs.

\begin{proposition}
Let  $\triangle_{n}$ be a permissible uniform flag complex on the vertex
set~$V_{n}$ that has the properties
that the $THTH$ type of arrows nest, the $HTHT$ type of arrows do not
nest, the $HTTH$ type of arrows cross and the $THHT$ type of arrows do not cross.
Then the complex~$\triangle_{n}$ represents a
triangulation of the boundary~$\partial P_{n}$ of the root polytope. 
\label{proposition_existence_simion_subclass_b}
\end{proposition}
\begin{proof}
Let $w$ be the associated $TH$-word to the two sets $I$ and $J$
and assume that the associated lattice path reaches height $h$.
Factor the word $w$ as follows 
$z_{1} \cdot z_{2} = z_{1} \cdot H w_{h+2} H \cdots H w_{2h+1}$
where the factors $w_{h+2}, w_{h+3}, \ldots, w_{2h+1}$
are lower Dyck words.
Note that such a factorization exists and is unique
since the $h$ $H$s are right-to-left maxima.
Note that the word $z_{1}$ contains at least $h$ $T$s.
Factor $z_{1}$ as
$H^{m_{1}} T H^{m_{2}} T \cdots T H^{m_{h}} T w_{h+1}$.
That is, we have
\begin{align*}
w = z_{1} \cdot z_{2}
& =
H^{m_{1}} T H^{m_{2}} T \cdots T H^{m_{h}} T w_{h+1}
\cdot
H w_{h+2} H \cdots H w_{2h+1} .
\end{align*}
Create a matching from $I$ to $J$ by
making $h$ forward arrows from the first $h$ $T$s in
$z_{1}$ to the $h$ $H$s in $z_{2}$, according to the $TTHH$ rule.
Apply
Lemma~\ref{lemma_backward_only_converse}
to each factor $w_{h+2}$ through $w_{2h+1}$ to create matchings consisting
of only backward arrows.
Finally, on the remaining letters
$H^{m_{1}}$, $H^{m_{2}}$,  $\ldots$, $H^{m_{h}}$, $w_{h+1}$
apply
Lemma~\ref{lemma_backward_only_converse}
again
to create the remaining backward arrows.
It is straightforward to see that this matching
satisfies all the conditions and hence the existence
part of the support axiom (SA) holds.

We now prove the uniqueness part of the support axiom (SA).
We denote the number of forward arrows in the matching
between $I$ and $J$ by $h$. The endpoints of these arrows partition the
number line into $2h+1$ segments. Since $THHT$ pairs do not cross, just
as in the proof of
Proposition~\ref{proposition_existence_simion_subclass_a},
backward
arrows having one end in one of the last $h$ segments have both ends in
the same segment. The endpoints of these arrows within the same segment
are associated to a lattice path starting and ending at the same
level. Before the leftmost head of a forward arrow the numbers of tails
exceeds the number of heads by $h$ and this is the largest height
reached by the associated lattice path. The heads of the forward arrows
mark the last descents to level $h-1$, $h-2$, $\ldots$, $0$,
respectively. Since $HTTH$ type of pairs cross, no tail of a backward arrow
may occur before the tail of the forward arrow: the leftmost $h$ tails
are tails of forward arrows, and they mark the first $h$ ascents in the
associated lattice path. The rest of the proof of the uniqueness and
existence parts of the support axiom is very similar to the one in the proof
of Proposition~\ref{proposition_existence_simion_subclass_a},
thus we omit the details.
We only underscore the key difference:
we treat all backward arrows whose tail is
to the right of the heads of the forward arrows as a single set:
after removing the tails of the  forward arrows this yields a lattice path
from level $0$ to level $0$ that never goes above the $x$-axis, and the
tails of the forward arrows are correctly reinserted if and only if they
are to the left of the tails of all backward arrows. Hence we 
consider $h+1$ runs of backward arrows: the last $h$ runs are just like
in the proof of Proposition~\ref{proposition_existence_simion_subclass_a}
and the first run is different. 

To verify the linkage axiom (LA) we observe that
Lemma~\ref{lemma_LA_backward} is applicable in all cases except when a
new head is inserted at the end or a new tail is inserted at the
beginning of any of the $h+1$ runs of backward arrows. If a new tail is
inserted at the beginning of one of the last $h$ runs of backward arrows then
we may proceed exactly as in the proof of
Proposition~\ref{proposition_existence_simion_subclass_a}
when $k$ is inserted right after
the head of a forward arrow. If $k < \min(I\cup J)$ is inserted as a tail,
again we may proceed as in the proof of
Proposition~\ref{proposition_existence_simion_subclass_a}
(note that there is nothing to the
left of this inserted new tail).  If $k$ is inserted as a head at the
end of one of the first $h$ runs of backward arrows then $k$ is
immediately followed by the head $j$ of a forward arrow $(i,j)$ and
Lemma~\ref{lemma_LA} is applicable. We are left with the case when
$k > \max(I \cup J)$ and $k$ is inserted as a head. If the matching on
$I \cup J$ contains at least one forward arrow, we proceed as in the
proof of Proposition~\ref{proposition_existence_simion_subclass_a}
(note that the dual case
when $k$ is inserted as a tail at the beginning was explained). Finally,
when the matching on $I\cup J$ consists of backward arrows only, then
let $(i,j)$ be the backward arrow whose tail $i$ is $\min(I)$. Replacing
$(i,j)$ with $(i,k)$ yields a matching in $\triangle_{n}$. 
\end{proof}
  
\begin{proposition}
Let  $\triangle_{n}$ be a permissible uniform flag complex on the vertex
set~$V_{n}$ that has the properties
that the $THTH$ type of arrows nest, the $HTHT$ type of arrows do not
nest, the $HTTH$ type of arrows do not cross and the $THHT$ type of arrows cross.
Then the complex~$\triangle_{n}$ represents a
triangulation of the boundary~$\partial P_{n}$ of the root polytope. 
\label{proposition_existence_simion_subclass_b_part_two}
\end{proposition}
\begin{proof}
Follows from 
Proposition~\ref{proposition_existence_simion_subclass_b}
and
Lemma~\ref{lemma_involutions}
by applying the involution
$\triangle_{n} \longmapsto \overline{\triangle_{n}}$.
\end{proof}

\begin{proof}[Proof of Theorem~\ref{theorem_existence_simion_class}.]
The result 
in the case $THTH$ nests and $HTHT$ does not nest
follows by combining
Propositions~\ref{proposition_existence_simion_subclass_a},
\ref{proposition_existence_simion_subclass_c},
\ref{proposition_existence_simion_subclass_b}
and~\ref{proposition_existence_simion_subclass_b_part_two}.
The case when
$THTH$ does not nest and $HTHT$ nests
follows by the involution
$\triangle_{n} \longmapsto \triangle_{n}^{*}$.
\end{proof}

\section{Fifteen distinct triangulations}
\label{section_fifteen_distinct_triangulations}

The classification given in Theorem~\ref{theorem_main} was accompanied
by the observation that some triangulations are (reflected) duals of each
other. Taking the dual of a uniform flag triangulation amounts to taking
a centrally symmetric copy. Taking the reflected dual amounts to
composing this reflection about the origin with a reflection
about the subspace defined as the intersection of the
$\lceil n/2 \rceil$ hyperplanes 
$x_i=x_{n+2-i}$ for $1 \leq i \leq \lceil n/2 \rceil$.
In this section we outline how to prove that there are no other 
isomorphisms between the uniform flag triangulations of the root polytope.

\begin{theorem}
For $n \geq 4$ there are
$15$ non-isomorphic uniform flag triangulations of the boundary of
the root polytope~$P_{n}$.
They are distributed as follows:
the lex class and the revlex class each contain $3$~triangulations;
the two Simion subclasses $a$ and $b$ each contain $4$~triangulations;
and finally
the Simion subclass $c$ only contains $1$~triangulation.
\label{theorem_15_different_triangulations} 
\end{theorem}

The upper bound of $15$ is a direct consequence of
Lemma~\ref{lemma_involutions}. 
In the rest of this section we show
how the triangulations differ already when $n=4$.

In every flag triangulation of the boundary $\partial P_{n}$ of
the root polytope each arrow
$(i,j)$ has at least $2n-2$ neighbors: the set
$\{(i,k) \: : \: k \not\in \{i,j\}\} \cup \{(k,j) \: : \: k\not\in \{i,j\}\}$
is a subset of cardinality $2n-2$ of the set of neighbors of $(i,j)$. 
\begin{definition}
The {\em excess degree} $\varepsilon(i,j)$ of the arrow $(i,j)\in V_{n}$
in a uniform flag complex $\triangle_{n}$ on the vertex set $V_{n}$ is the
number of arrows $(i',j')$ such that $|\{i,i',j,j'\}|=4$ and
$\{(i,j),(i',j')\} \in \triangle_{n}$.
\end{definition}  
In other words,
the excess degree $\varepsilon(i,j)$ is the amount by which the degree
of $(i,j)$ exceeds $2n-2$. 
Table~\ref{table_excess_degrees_p_q_r}
shows how to compute the excess degrees.
For example, whenever $THTH$ pairs of arrows nest and $i<j$,
then for each arrow $(i,j)$
there are $\binom{q}{2}$ ways to select $(i',j')$ satisfying $i<j'<i'<j$,
where $p = i-1$, $q = j-i-1$ and $r = n+1-j$.

In Table~\ref{table_excess_degrees_n_4} we show
the sorted lists of the excess degrees for
the $15$ triangulations when $n=4$. It is straightforward to observe
that these lists are distinct, showing that the triangulations
are non-isomorphic in general, and hence
Theorem~\ref{theorem_15_different_triangulations} holds.

\newcommand{\lifta}[1]{\raisebox{1.5mm}{#1}}
\newcommand{\liftb}[1]{\raisebox{2.0mm}{#1}}
\newcommand{\liftc}[1]{\raisebox{2.5mm}{#1}}

\renewcommand{\arraystretch}{1.5}
\begin{table}[t]
\begin{center}
\begin{tabular}{|c|c|c||c|c|c|}
\hline
Condition & Contribution & Contribution &
Condition & Contribution & Contribution \\
                & to $\varepsilon(i,j)$ & to $\varepsilon(j,i)$ &
                & to $\varepsilon(i,j)$ & to $\varepsilon(j,i)$ \\
\hline
\hline
\begin{tikzpicture}
\smallbase{T}{H}{T}{H}
\draw[->, thick] (a1) to[out=55, in=125] (a4);
\draw[<-, thick] (a2) to[out=55, in=125] (a3);
\end{tikzpicture}
& \liftc{$\binom{q}{2}$}  & \liftc{$p r$} &
\begin{tikzpicture}
\smallbase{T}{H}{T}{H}
\draw[->, thick] (a1) to[out=55, in=125] (a2);
\draw[->, thick] (a3) to[out=55, in=125] (a4);
\end{tikzpicture}
& \lifta{$\binom{p}{2}+\binom{r}{2}$}  & \lifta{$0$} \\
\hline
\begin{tikzpicture}
\smallbase{H}{T}{H}{T}
\draw[<-, thick] (a1) to[out=55, in=125] (a4);
\draw[->, thick] (a2) to[out=55, in=125] (a3);
\end{tikzpicture}
&  \liftc{$p r$} & \liftc{$\binom{q}{2}$} &
\begin{tikzpicture}
\smallbase{H}{T}{H}{T}
\draw[<-, thick] (a1) to[out=55, in=125] (a2);
\draw[<-, thick] (a3) to[out=55, in=125] (a4);
\end{tikzpicture}
& \lifta{$0$} & \lifta{$\binom{p}{2}+\binom{r}{2}$} \\
\hline
\begin{tikzpicture}
\smallbase{T}{H}{H}{T}
\draw[->, thick] (a1) to[out=55, in=125] (a2);
\draw[<-, thick] (a3) to[out=55, in=125] (a4);
\end{tikzpicture}
& \lifta{$\binom{r}{2}$}  & \lifta{$\binom{p}{2}$} &
\begin{tikzpicture}
\smallbase{T}{H}{H}{T}
\draw[->, thick] (a1) to[out=55, in=125] (a3);
\draw[<-, thick] (a2) to[out=55, in=125] (a4);
\end{tikzpicture}
& \liftb{$q r$}  & \liftb{$p q$} \\
\hline
\begin{tikzpicture}
\smallbase{H}{T}{T}{H}
\draw[<-, thick] (a1) to[out=55, in=125] (a2);
\draw[->, thick] (a3) to[out=55, in=125] (a4);
\end{tikzpicture}
& \lifta{$\binom{p}{2}$} & \lifta{$\binom{r}{2}$} &
\begin{tikzpicture}
\smallbase{H}{T}{T}{H}
\draw[<-, thick] (a1) to[out=55, in=125] (a3);
\draw[->, thick] (a2) to[out=55, in=125] (a4);
\end{tikzpicture}
& \liftb{$p q$} & \liftb{$q r$} \\
\hline
\begin{tikzpicture}
\smallbase{H}{H}{T}{T}
\draw[<-, thick] (a1) to[out=55, in=125] (a4);
\draw[<-, thick] (a2) to[out=55, in=125] (a3);
\end{tikzpicture}
& \liftc{$0$}  & \liftc{$p r + \binom{q}{2}$} &
\begin{tikzpicture}
\smallbase{H}{H}{T}{T}
\draw[<-, thick] (a1) to[out=55, in=125] (a3);
\draw[<-, thick] (a2) to[out=55, in=125] (a4);
\end{tikzpicture}
& \liftb{$0$}  & \liftb{$p q + q r$} \\
\hline
\begin{tikzpicture}
\smallbase{T}{T}{H}{H}
\draw[->, thick] (a1) to[out=55, in=125] (a4);
\draw[->, thick] (a2) to[out=55, in=125] (a3);
\end{tikzpicture}
& \liftc{$p r + \binom{q}{2}$} & \liftc{$0$}  &
\begin{tikzpicture}
\smallbase{T}{T}{H}{H}
\draw[->, thick] (a1) to[out=55, in=125] (a3);
\draw[->, thick] (a2) to[out=55, in=125] (a4);
\end{tikzpicture}
& \liftb{$p q + q r$} & \liftb{$0$}  \\
\hline
\end{tabular}
\end{center}
\caption{Table to compute the excess degrees
$\varepsilon(i,j)$ and $\varepsilon(j,i)$ for $1 \leq i < j \leq n+1$,
where $p=i-1$, $q=j-i-1$ and $r=n+1-j$.}
\label{table_excess_degrees_p_q_r}
\end{table}

\newcommand{\smallsmallbase}[4]
{
\node[circle, inner sep = 0.9pt] (a1) at (0, 0) {$\scriptscriptstyle #1$};
\node[circle, inner sep = 0.9pt] (a2) at (0.4, 0) {$\scriptscriptstyle #2$};
\node[circle, inner sep = 0.9pt] (a3) at (0.9, 0) {$\scriptscriptstyle #3$};
\node[circle, inner sep = 0.9pt] (a4) at (1.3, 0) {$\scriptscriptstyle #4$};
}

\newcommand{\crossbackward}
{
\begin{tikzpicture}
\smallsmallbase{H}{H}{T}{T}
\draw[<-, thick] (a1) to[out=55, in=125] (a3);
\draw[<-, thick] (a2) to[out=55, in=125] (a4);
\end{tikzpicture}
}

\newcommand{\crossforward}
{
\begin{tikzpicture}
\smallsmallbase{T}{T}{H}{H}
\draw[->, thick] (a1) to[out=55, in=125] (a3);
\draw[->, thick] (a2) to[out=55, in=125] (a4);
\end{tikzpicture}
}

\newcommand{\nestbackward}
{
\begin{tikzpicture}
\smallsmallbase{H}{H}{T}{T}
\draw[<-, thick] (a1) to[out=55, in=125] (a4);
\draw[<-, thick] (a2) to[out=55, in=125] (a3);
\end{tikzpicture}
}

\newcommand{\nestforward}
{
\begin{tikzpicture}
\smallsmallbase{T}{T}{H}{H}
\draw[->, thick] (a1) to[out=55, in=125] (a4);
\draw[->, thick] (a2) to[out=55, in=125] (a3);
\end{tikzpicture}
}

\begin{table}[t]
\begin{center}
\begin{tabular}{|l|c|l|}
\hline    
Class/ & $HHTT$ \& $TTHH$ & Sorted list of excess degrees \\[-1mm]
subclass & conditions & \\
\hline
Lex
&
\nestbackward\nestforward
& 
$1^{6}, 2^{4}, 3^{2}, 4^{4}, 6^{4}$
\\[-2mm]
& 
\nestbackward\crossforward
&
$0^{1}, 1^{3}, 2^{7}, 3^{1}, 4^{4}, 6^{4}$
\\[-2mm]
& 
\crossbackward\crossforward
&
$0^{2}, 2^{10}, 4^{4}, 6^{4}$
\\[0mm]
\hline
Simion $a$         
&
\nestbackward\nestforward
&
$1^{5}, 2^{5}, 3^{5}, 6^{5}$
\\[-2mm]
&
\crossbackward\nestforward
&
$0^{1}, 1^{3}, 2^{4}, 3^{5}, 4^{4}, 6^{3}$
\\[-2mm]
&
\nestbackward\crossforward
&
$1^{4}, 2^{4}, 3^{8}, 6^{4}$
\\[-2mm]
&
\crossbackward\crossforward
&
$0^{1}, 1^{2}, 2^{3}, 3^{8}, 4^{4}, 6^{2}$
\\[-1mm]
\hline
Simion $b$         
&
\nestbackward\nestforward
&
$0^{1}, 1^{2}, 2^{5}, 3^{7}, 4^{1}, 5^{1}, 6^{3}$
\\[-2mm]
&
\crossbackward\nestforward
&
$0^{2}, 1^{1}, 2^{5}, 3^{4}, 4^{5}, 5^{1}, 6^{2}$
\\[-2mm]
&
\nestbackward\crossforward
&
$0^{2}, 1^{2}, 2^{1}, 3^{10}, 4^{1}, 5^{2}, 6^{2}$
\\[-2mm]
&
\crossbackward\crossforward
&
$0^{3}, 1^{1}, 2^{1}, 3^{7}, 4^{5}, 5^{2}, 6^{1}$
\\[-1mm]
\hline
Simion $c$ 
&
\nestbackward\nestforward
&
$0^{2}, 2^{4}, 3^{8}, 4^{3}, 5^{2}, 6^{1}$
\\[-1mm]
\hline
Revlex
&
\nestbackward\nestforward
&
$0^{4}, 2^{4}, 4^{10}, 6^{2}$
\\[-2mm]
&
\crossbackward\nestforward
&
$0^{4}, 2^{4}, 3^{1}, 4^{7}, 5^{3}, 6^{1}$
\\[-2mm]
&
\crossbackward\crossforward
&
$0^{4}, 2^{4}, 3^{2}, 4^{4}, 5^{6}$
\\[0mm]
\hline
\end{tabular}    
\end{center}  
\caption{The sorted lists of the excess degrees for $n=4$,
where superscripts denotes multiplicities.}
\label{table_excess_degrees_n_4}
\end{table}

\section{Tools for refined face enumeration}
\label{section_tools}

In this section we introduce some terminology and results that we will use to
prove theorems regarding the refined face counting in uniform flag
triangulations of the boundary $\partial P_{n}$ of the root polytope.
Our triangulations are defined by a set of rules,
independent of the dimension. After fixing such a set of rules, we
will simultaneously consider each triangulation $\triangle_{n}$
determined on the vertex set $V_{n}$ defined by these rules, for each $n \geq 0$.
Note that the set $V_{0}$ is the empty set, and the only face contained
in~$\triangle_{0}$ is the empty set. 

In order to compute the associated generating function, we introduce
the following more general notion.
Let $\mathcal{F}  = (\mathcal{F}_{0}, \mathcal{F}_{1}, \ldots)$
be a sequence of collections of arrows, where
$\mathcal{F}_{n}$ is a subset of the power set of $V_{n}$,
that is, $\mathcal{F}_{n} \subseteq 2^{V_{n}}$.
We define the associated 
{\em face generating function}
$$
F(\mathcal{F}, x,y,t)
=
\sum_{n,i,j \geq 0} f(\mathcal{F}_{n},i,j) \cdot x^{i}y^{j}t^{n}
$$
where $f(\mathcal{F}_{n},i,j)$ is the number of sets in $\mathcal{F}_{n}$
consisting of $i$ forward arrows and $j$ backward arrows.
Our interest is to compute  this generating function
when $\mathcal{F}$ is the collection
$\mathcal{F}  =  (\triangle_{0}, \triangle_{1}, \ldots)$.
In this case $f(\triangle_{n},i,j)$
counts the faces of the simplicial complex~$\triangle_{n}$ with $i$
forward arrows and $j$ backward arrows. We call
the polynomial $\sum_{i,j \geq 0} f(\triangle_{n},i,j) \cdot x^{i} y^{j}$
the {\em face polynomial}. 
We note that the power of $t$ in a term of $F(\triangle_{n},x,y,t)$
is the same as the number of vertices
in an associated facet of the triangulation.
The number of cases to be considered can be reduced by extending the
notion of the dual and reflected dual triangulations to all families
$\mathcal{F}  = (\mathcal{F}_{0}, \mathcal{F}_{1}, \ldots)$.

The following lemma is a straightforward
consequence of the definitions.
\begin{lemma}
Let $\mathcal{F}  = (\mathcal{F}_{0}, \mathcal{F}_{1}, \ldots)$
be a sequence of collections of arrows, where $\mathcal{F}_{n} \subseteq
2^{V_{n}}$. Let $\mathcal{F}^{*}  = (\mathcal{F}_{0}^{*},
\mathcal{F}_{1}^{*}, \ldots)$ be the dual family obtained by
reversing all arrows in all sets, and let $\overline{\mathcal{F}}  =
(\overline{\mathcal{F}_{0}}, 
\overline{\mathcal{F}_{1}}, \ldots)$ be the reflected dual family
obtained by reversing each arrow, and replacing each node $i$ in $V_{n}$
with $n+2-i$. Then the following two equalities hold:
$$
F(\mathcal{F}^{*}, x,y,t)
=
F(\mathcal{F}, y,x,t)
\quad\mbox{and}\quad
F(\overline{\mathcal{F}}, x,y,t)
=
F(\mathcal{F}, x,y,t).
$$
\label{lemma_dual_enumeration}  
\end{lemma}  
The families of uniform flag triangulations defined by a set of rules are {\em
coherent}  in the sense that they are closed under the insertion and
removal of isolated nodes. To make this informal observation precise,
consider the map $\pi_{k}: {\mathbb N}-\{k\}\rightarrow {\mathbb N}$ given
by 
\begin{align*}
  \pi_{k}(m)
& =
  \begin{cases}
    m & \mbox{if $m< k$,}\\
    m-1 & \mbox{if $m>k$.}\\
  \end{cases}  
\end{align*}
\begin{definition}
Let $\mathcal{F} = (\mathcal{F}_{0}, \mathcal{F}_{1},\ldots)$
be a collection of families of sets such that for each $n$
the family~$\mathcal{F}_{n}$ consists of subsets of $V_{n}$.
We call such a collection {\em coherent}
if for each subset $\sigma$ of $V_{n}$ and each
$k \in \{1,2, \ldots, n+1\}$ that is not incident to any arrow in $\sigma$, the
set $\pi_{k}(\sigma)=\{(\pi_{k}(i),\pi_{k}(j))\::\:   (i,j)\in \sigma\}$
belongs to $\mathcal{F}_{n-1}$ if and only if $\sigma \in \mathcal{F}_{n}$.  
In particular, for a coherent collection
$\mathcal{F}=(\mathcal{F}_{0}, \mathcal{F}_{1}, \ldots)$,
the empty set either belongs to all
$\mathcal{F}_{n}$ or it belongs to none of them. 
\end{definition}
By abuse of terminology we will say that $\mathcal F$ {\em contains the
  empty set} if all families $\mathcal{F}_{n}$ in $\mathcal F$ contain it. 
Sets of arrows in coherent collections may be enumerated by counting only
the {\em saturated}  sets in the families, which we now define. 
\begin{definition}
For $n\geq 1$, a subset $\sigma$ of $V_{n}$ is {\em saturated} if the set of
endpoints of its arrows is the set $\{1,2, \ldots, n+1\}$. We also
consider the empty set to be a saturated subset of $V_{0}=\emptyset$.
Given a coherent collection
$\mathcal{F} = (\mathcal{F}_{0}, \mathcal{F}_{1}, \ldots)$,
we denote the family of saturated sets in
$\mathcal{F}_{n}$ by $\widehat{\mathcal{F}}_{n}$. 
\end{definition}

Note that $\widehat{\mathcal{F}_{0}}=\{\emptyset\}$ exactly when $\mathcal
F$ contains the empty set.
One of our key enumeration tools is the following observation.

\begin{lemma}
\label{lemma_all_full}
Given a coherent collection of families $\mathcal{F}$ of arrows, the
face generating function satisfies
\begin{align}
F(\mathcal{F},x,y,t)
& =
\frac{-t}{(1-t)^{2}} \cdot \delta_{\mathcal{F}_{0}, \{\emptyset\}} 
+
\frac{1}{(1-t)^{2}} \cdot F\left(\widehat{\mathcal{F}},x,y,\frac{t}{1-t}\right) ,
\label{equation_all_full_1} \\
F(\widehat{\mathcal{F}},x,y,z)
& =
\frac{z}{1+z} \cdot \delta_{\widehat{\mathcal{F}}_{0}, \{\emptyset\}}
+
\frac{1}{(1+z)^{2}} \cdot F\left(\mathcal{F},x,y,\frac{z}{1+z}\right) ,
\label{equation_all_full_2}
\end{align}  
where $\delta$ denotes the Kronecker delta function.   
\end{lemma}  
\begin{proof}
Consider the second term of equation~\eqref{equation_all_full_1}.
The power $n$ in a term
$x^{i} y^{j} t^{n}$ in
$F(\widehat{\mathcal{F}},x,y,t)$
counts the number of spaces between nodes
in the digraph.
The substitution $t \longmapsto t/(1-t)$
corresponds to subdividing each space into
more spaces, that is, inserting isolated nodes
into these spaces.
Finally,
the factor of $1/(1-t)^{2}$ corresponds to inserting
isolated nodes before and after the full subset.
This completes the proof of the first
equation in the case when
$\widehat{\mathcal{F}}_{0} = \emptyset$.
In the case when
$\widehat{\mathcal{F}}_{0} = \{\emptyset\}$
we need to correct the right-hand side of the first
equation by subtracting $1/(1-t)^{2}$ contributed by $z^{0}$ in
$F(\widehat{\mathcal{F}},x,y,z)$ and we need to add $1/(1-t)$ to
account for the empty set belonging to all families $\mathcal{F}_{n}$. 
Equation~\eqref{equation_all_full_1} is equivalent to~\eqref{equation_all_full_2}
by noting that $z=t/(1-t)$ is equivalent to $t=z/(1+z)$.
\end{proof}

An interesting special case is counting all facets with a given number of forward
and backward arrows. The following statement is straightforward to check.
\begin{lemma}
\label{lemma_facets}
Let $\triangle_{n}$ be any uniform flag triangulation of the boundary $\partial
P_{n}$ of the root polytope. Then a face $\sigma\in \triangle_{n}$ is a
facet if and only if it is saturated and contains no isolated nodes,
that is, every $i\in \{1,2,\ldots,n+1\}$ is incident to some element of
$\sigma$. 
\end{lemma}
In other words, as a subset of $V_{n}$, a facet is the set of edges of a
forest with no isolated nodes. Such a forest has $n+1$ nodes and $n$
arrows. If the number of forward arrows is $i$ then the number of
backward arrows is $n-i$, contributing a term
$x^iy^{n-i}t^n$ to the generating function of all faces and the term
$x^iy^{n-i}z^n$ to the generating function of all saturated faces. 
\begin{corollary}
\label{corollary_facets}
Let $\mathcal{F}  =  (\triangle_{0}, \triangle_{1}, \ldots)$ be a
coherent family of uniform flag triangulations. Then the facet generating
function $\sum_{n\geq 0}\sum_{i=0}^n f(\triangle_{n},i,n-i)
x^iy^{n-i} z^n$ may be obtained by substituting $x/w$ into $x$, $y/w$ into $y$
and $wz$ into $t$ in the face generating function $F(\mathcal{F},x,y,t)$
and then evaluating the 
resulting expression at $w=0$. Alternatively it may also be obtained by   
substituting $x/w$ into $x$, $y/w$ into $y$
and $wz$ into $z$ in $F(\widehat{\mathcal{F}},x,y,z)$ and then evaluating the
resulting expression at $w=0$.  
\end{corollary}

For any uniform flag triangulation of the boundary $\partial P_{n}$ of
the root polytope, the vertex sets consisting only of forward (backward)
arrows form subcomplexes whose face numbers are easier to count. We
first review the rephrasing of a known result. 

A useful way to express our results is in terms of
the generating function for the Catalan numbers:
\begin{align*}
C(u)
& =
\sum_{n \geq 0} C_{n} \cdot u^{n} = \frac{1-\sqrt{1-4u}}{2u} .
\end{align*}

\begin{proposition}
Let $\mathcal{F}$ be a family of uniform flag triangulations defined by a
set of rules that contains the rule that $HTHT$ types of pairs of arrows
do not nest. Then the following two identities hold:
\begin{align}
F(\mathcal{F},0,y,t)
& =
\frac{1-t-\sqrt{1-(4y+2)t+t^{2}}}{2yt} ,
\label{equation_F_0_y_t} \\
F(\widehat{\mathcal{F}},0,y,z)
& =
\frac{1}{1+z}\cdot \left(C(yz(z+1))+z\right) .
\label{equation_follow_full}  
\end{align}
\label{proposition_follow}  
\end{proposition}
\begin{proof}
Setting $x=0$ implies that we are only interested
in digraphs with backward arrows. By the dual of
Theorem~\ref{theorem_forward}  
the subcomplex of faces formed by
backward arrows represents the lexicographic or revlex
pulling triangulation of the faces of~$P^{+}_{n}$ not
containing the origin. (The choice depends on the rule for the $HHTT$
pairs of arrows.) 
Both triangulations have the same face numbers.
The proof of Theorem~5.4 in~\cite[Theorem 5.4]{Hetyei_Legendre}
implies the quadratic equation
\begin{align}
F(\mathcal{F},0,y,t)
& =
1  +  t \cdot F(\mathcal{F},0,y,t)  +  y t \cdot F(\mathcal{F},0,y,t)^{2}.
\label{equation_quadraticF}  
\end{align}
Solving it yields~\eqref{equation_F_0_y_t}.
Identity~\eqref{equation_follow_full}  
follows by applying~\eqref{equation_all_full_2} in Lemma~\ref{lemma_all_full}
and~\eqref{equation_F_0_y_t}.
\end{proof}  

\begin{remark}
{\em
Another way to prove~\eqref{equation_F_0_y_t}
is to use Theorem~5.4 and Corollary~5.6
in~\cite{Hetyei_Legendre}, which states
\begin{align*}
\sum_{j=0}^{n} f(\mathcal{F}_{n},0,j) \cdot \left(\frac{u-1}{2}\right)^{j}
& =
\sum_{j=0}^{n}
\frac{1}{j+1} \cdot \binom{n+j}{j} \cdot \binom{n}{j}
\cdot \left(\frac{u-1}{2}\right)^{j}
 = 
\frac{P_{n}^{(-1,1)}(u)}{n+1} , 
\end{align*}
where $P_{n}^{(-1,1)}(u)$ is a Jacobi polynomial. The stated equation
follows by integrating the well-known generating
function~\cite[22.9.1]{Abramowitz-Stegun} of the Jacobi polynomials
$P_{n}^{(-1,1)}(u)$.  
}
\end{remark}

We will later use the following corollary of
Proposition~\ref{proposition_follow}.
It has a direct bijective proof;
see~\cite{Ehrenborg_Hetyei_Readdy_Catalan}.
\begin{corollary}
Let $\mathcal{F}$ be a family of uniform flag triangulations
such that $HTHT$ types of pairs of arrows do not nest.
Then the sum over all forests $F$ consisting of $k \geq 1$
backward arrows, no forward arrows and no isolated nodes is
\begin{align}
G_{k}(z) 
= 
\sum_{F} z^{\text{\#nodes of $F$}}
& = 
C_{k} \cdot z^{k+1} \cdot (z+1)^{k-1} .
\label{equation_G_1}
\end{align}
Similarly, if the uniform flag triangulations $\mathcal{F}$ 
satisfies the requirement that $THTH$ types of pairs of arrows do not nest,
then the sum over all forests $F$ consisting of $k \geq 1$
forward arrows, no backward arrows and no isolated nodes
also yields
the identity~\eqref{equation_G_1}.
\label{corollary_Catalan_times_1+z}
\end{corollary}
\begin{proof}
Equation~\eqref{equation_G_1}
follows by considering the coefficient of $y^{k}$
in equation~\eqref{equation_follow_full}.
Observe that there is an extra factor of $z$ since we are counting
the number of nodes.
The second statement follows by reversing the first statement.
\end{proof}
Note that the lower extreme cases of
Corollary~\ref{corollary_Catalan_times_1+z} enumerate the
anti-standard trees and the noncrossing alternating trees
of Gelfand, Graev and Postnikov~\cite{Gelfand_Graev_Postnikov}.
We will also use the following refined variant of
Proposition~\ref{proposition_follow}.
\begin{proposition}
Let $\mathcal{F}$ be a family of uniform flag triangulations defined by a
set of rules that contains the rule that $HTHT$ type of pairs of arrows
do not nest. For each $n\geq 0$, let $\mathcal{F}^{(i)}_{n}$ denote the
set of all faces
where the sequence of heads and tails, listed in increasing order,
satisfies the condition that
the $i$ smallest nodes are heads and the next node is a tail.
Then the two resulting collections
$\mathcal{F}^{(i)}
= 
\bigl(\mathcal{F}^{(i)}_{0}, \mathcal{F}^{(i)}_{1}, \ldots\bigr)$
and
$\widehat{\mathcal{F}^{(i)}} =
\bigl(\widehat{\mathcal{F}^{(i)}_{0}}, \widehat{\mathcal{F}^{(i)}_{1}}, \ldots\bigr)$
satisfy
\begin{align}
F(\mathcal{F}^{(i)},0,y,t)
& =
\frac{y^{i} t^{i} F(\mathcal{F},0,y,t)^{i}}{(1-t)^{i+1}},
\label{equation_follow_refined}
\\
F(\widehat{\mathcal{F}^{(i)}},0,y,z)
& =
\frac{1}{1+z} \cdot
\left(y z (1+z) C(yz(z+1))\right)^{i} .
\label{equation_follow_refined_full}  
\end{align}
\label{proposition_follow_refined}   
\end{proposition}
\begin{proof}
Without loss of generality we may assume that $HHTT$ type pairs of
arrows nest. Indeed, if we fix the head-tail pattern of the nodes
then we are counting the faces in a triangulation of the convex hull of
all vertices represented by all backward arrows whose heads and tails
are selected from the prescribed set of heads and tails. For example, if
we fix $n=3$ and the pattern $THTH$ (contributing to the collection
$\mathcal{F}^{(1)}$) then we have to count the faces in the
triangulation of the convex hull of $e_{2}-e_{1}$, $e_4-e_{1}$ and
$e_4-e_3$. We obtain the same face numbers for the lexicographic and
revlex triangulations. 
  
First we show that for $i>1$ the formal power series
$F(\mathcal{F}^{(i)},0,y,t)$ satisfies the relation
\begin{align}
F(\mathcal{F}^{(i)},0,y,t)
& =
t \cdot F(\mathcal{F}^{(i)},0,y,t) 
\label{equation_firecurrence} \\
& +
\sum_{d \geq 1}
y^{d} t^{d} \cdot F(\mathcal{F}^{(i-1)},0,y,t) \cdot 
F(\mathcal{F},0,y,t)^{d-1} \cdot (1 + t F(\mathcal{F},0,y,t)).  
\nonumber
\end{align}
The term $t\cdot F(\mathcal{F}^{(i)},0,y,t)$ corresponds to the
possibility that the node $1$ is not incident to any arrow in a face,
and the $d$th term in the next sum accounts for the possibility that the node $1$ is
the head of exactly 
$d$ arrows. Assume the set of tails of these arrows is
$\{i_{1},i_{2}, \ldots,  i_{d}\}$, where $1<i_{1}<i_{2}<\cdots< i_{d}$.
These arrows contribute a factor $y^{d} t^{d}$. Since no
pairs of arrows are allowed to cross, the $d$ arrows whose tail is $1$
partition the set of arrows into $d+1$ classes. Each of these classes
may be empty, except for the set of arrows whose head precedes $i_{1}$:
these arrows form a set where $i-1$ tails are followed by the first head
when we list their endpoints in increasing order. These arrows
contribute a factor of $F(\mathcal{F}^{(i-1)},0,y,t)$. For each $j$ in
$\{1,2, \ldots,  d-1\}$, the set of arrows whose head belongs to the set
$\{i_{j}+1, i_{j}+2, \ldots,  i_{j+1}\}$ contribute a factor $F({\mathcal
  F},0,y,t)^{d-1}$. (Note that the factor of $t$ contributed by~$i_{j}$ is
already counted.) Finally the set of arrows whose tail is $i_{d}+1$ or a
larger element contributes a factor of $1+t F(\mathcal{F},0,y,t)$.
Rearranging~\eqref{equation_firecurrence} yields the recurrence
\begin{align*}
F(\mathcal{F}^{(i)},0,y,t)
& =
F(\mathcal{F}^{(i-1)},0,y,t)
\cdot 
\frac{yt(1+t F(\mathcal{F},0,y,t))}{(1-t)(1-ytF(\mathcal{F},0,y,t))}.
\end{align*}
By~\eqref{equation_quadraticF} we may replace the factor
$1+t F(\mathcal{F},0,y,t)$ with
$F(\mathcal{F},0,y,t) \cdot (1- yt F(\mathcal{F},0,y,t))$.
Thus we obtain
\begin{align*}
F(\mathcal{F}^{(i)},0,y,t)
& =
F(\mathcal{F}^{(i-1)},0,y,t)
\cdot 
\frac{ytF(\mathcal{F},0,y,t)}{1-t}.
\end{align*}
Combining this recurrence with the expression
\begin{align*}
F(\mathcal{F}^{(1)},0,y,t)
& =
\frac{ytF(\mathcal{F},0,y,t)}{(1-t)^2},
\end{align*}
which may be shown in a completely analogous fashion,
equation~\eqref{equation_follow_refined}
follows by induction on $i$. 

Combining equations~\eqref{equation_all_full_2},
\eqref{equation_F_0_y_t}
and~\eqref{equation_follow_refined}
yields the identity
\eqref{equation_follow_refined_full}.
\end{proof}

When $HTHT$ pairs of arrows do not nest, we obtain a very different
expression for enumerating faces containing backward arrows only. We
employ its {\em dual form}, obtained after reversing all
arrows, using a generalization of the Delannoy numbers.
Recall a {\em Delannoy path} from $(0,0)$ to $(a,b)$ is a lattice path
consisting of North steps $(0,1)$, East steps $(1,0)$ and NE steps
$(1,1)$. The number of Delannoy paths from $(0,0)$ to $(a,b)$ is the
Delannoy number $D_{a,b}$; see~\cite{Banderier_Schwer}
for more on Delannoy numbers.

\begin{definition}
Given two non-negative integers $a$ and $b$,
the {\em Delannoy polynomial $D_{a,b}(x)$}
is the total weight of all Delannoy paths from $(0,0)$ to $(a,b)$, where
each step contributes a factor of $x$. Thus the coefficient of $x^{j}$ in
$D_{a,b}(x)$ is the number of Delannoy paths from $(0,0)$ to $(a,b)$
having $j$ steps. 
\end{definition}  
The bivariate ordinary generating function of the Delannoy polynomials
is given by 
\begin{align}
D(u,v,x)
& =
\sum_{a,b\geq 0} u^{a} v^{b} \cdot D_{a,b}(x)
=
\frac{1}{1-x(u+v+uv)} .
\label{equation_Delannoy}
\end{align}
\begin{proposition}
Let $\mathcal{F}$ be a family of uniform flag triangulations defined by a
set of rules requiring that $THTH$ type of pairs of arrows nest.
Then the collection $\widehat{\mathcal{F}}$ of families of
saturated faces satisfies
\begin{align}
F(\widehat{\mathcal{F}},x,0,z)
& =
1 + xz \cdot \sum_{a,b\geq 0} D_{a,b}(x) \cdot z^{a+b}. 
\label{equation_saturated_Delannoy}
\end{align}
In particular, the contribution to $F(\widehat{\mathcal{F}},x,0,z)$
of all saturated faces having $a+1$ tails and $b+1$ heads is
$D_{a,b}(x) \cdot xz^{a+b+1}$.
\label{proposition_nest_full}
\end{proposition}  
\begin{proof}
The constant $1$ on the right-hand side
of~\eqref{equation_saturated_Delannoy}
accounts for the empty face.   
The condition on the $THTH$ type of pairs of arrows implies that for any
face consisting of forward arrows only, all tails precede all heads.
Assume that the set of tails of forward arrows representing a saturated face has
cardinality $a+1$ and 
that the set of heads has cardinality $b+1$.
Let the set of tails be
$\{i_{1} < i_{2} < \cdots < i_{a+1}\}$
and the set of heads be
$\{j_{1}, j_{2}, \ldots, j_{b+1}\}$.
If $TTHH$ type of pairs of arrows nest,
order the set of heads decreasingly, that is,
$j_{b+1} < \cdots < j_{2} < j_{1}$.
Otherwise, if $TTHH$ type of pairs of arrows cross,
order the set of heads increasingly, that is,
$j_{1} < j_{2} < \cdots < j_{b+1}$.
Order the arrows in lexicographic order.
Associate a North step to each
instance when the tail remains unchanged from the next arrow in the
list, an East step to each instance when the head remains unchanged and
a NE step to each instance when both head and tail change. We obtain a
Delannoy path from~$(0,0)$ to~$(a,b)$ in which the number of steps is
one less than the number of arrows on the list. The correspondence is
a bijection between all saturated faces on the given set of heads
and tails and all Delannoy paths from~$(0,0)$ to~$(a,b)$. 
\end{proof}

\section{Face enumeration in the Simion class}
\label{section_Simion}

In this section we compute generating functions for the class of uniform
triangulations of the boundaries of the root polytope
in the Simion class, that is,
exactly one of the $THTH$ and $HTHT$ types of pairs of arrows nest.
Our main result on enumerating faces is the following theorem:
\begin{theorem}
Let $\mathcal{F}$ be a collection of uniform flag triangulations
belonging to the Simion class and let $\widehat{\mathcal{F}}$ be the
collection of families of saturated faces. If $THTH$ types of pairs of arrows
do not nest and $HTHT$ types of pairs of arrows nest then the following
identity holds:
\begin{align}
F(\widehat{\mathcal{F}},x,y,z)
& =
\frac{C(yz(z+1))+z}{1+z}
+
\frac{xz\cdot (1+zC(yz(z+1)))\cdot C(yz(z+1))^{2}}
{(1+z) \cdot (1-2C(yz(z+1))xz-C(yz(z+1))^{2}xz^{2})}. 
\label{equation_enumeration_in_the_simion_class}  
\end{align}
If $THTH$ types of pairs of arrows
nest and $HTHT$ types of pairs of arrows do not nest,
then the symmetric identity holds:
\begin{align}
F(\widehat{\mathcal{F}},x,y,z)
& =
\frac{C(xz(z+1))+z}{1+z}
+
\frac{yz\cdot (1+zC(xz(z+1)))\cdot C(xz(z+1))^{2}}
{(1+z) \cdot (1-2C(xz(z+1))yz-C(xz(z+1))^{2}yz^{2})}. 
\label{equation_enumeration_in_the_dual_simion_class}  
\end{align}
\label{theorem_enumeration_in_the_simion_class}  
\end{theorem}  
It suffices to prove the first half of
Theorem~\ref{theorem_enumeration_in_the_simion_class}. The second half
is a direct consequence of Lemmas~\ref{lemma_involutions} and
\ref{lemma_dual_enumeration}, using the duality $\triangle\longmapsto
\triangle^{*}$. We prove the first half for each subclass
of the Simion class separately, in
Propositions~\ref{proposition_simion_subclass_a}, 
\ref{proposition_simion_subclass_b}
and~\ref{proposition_simion_subclass_c}, respectively.

We begin by studying the type~$a$ subclass of the
Simion class;
see Figure~\ref{figure_Simion_type_a}.

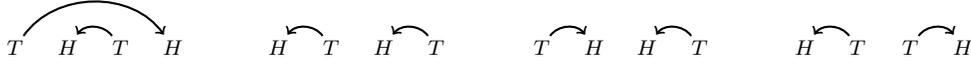
\begin{figure}
\begin{center}
\begin{tikzpicture}
\node[circle, inner sep = 0.9pt] (a1) at (0, 0) {$\scriptstyle{T}$};
\node[circle, inner sep = 0.9pt] (a2) at (0.7, 0) {$\scriptstyle{H}$};
\node[circle, inner sep = 0.9pt] (a3) at (1.4, 0) {$\scriptstyle{T}$};
\node[circle, inner sep = 0.9pt] (a4) at (2.1, 0) {$\scriptstyle{H}$};
\draw[->, thick] (a1) to[out=55, in=125] (a4);
\draw[<-, thick] (a2) to[out=55, in=125] (a3);

\node[circle, inner sep = 0.9pt] (b1) at (3.5, 0) {$\scriptstyle{H}$};
\node[circle, inner sep = 0.9pt] (b2) at (4.2, 0) {$\scriptstyle{T}$};
\node[circle, inner sep = 0.9pt] (b3) at (4.9, 0) {$\scriptstyle{H}$};
\node[circle, inner sep = 0.9pt] (b4) at (5.6, 0) {$\scriptstyle{T}$};
\draw[<-, thick] (b1) to[out=55, in=125] (b2);
\draw[<-, thick] (b3) to[out=55, in=125] (b4);

\node[circle, inner sep = 0.9pt] (c1) at (7.0, 0) {$\scriptstyle{T}$};
\node[circle, inner sep = 0.9pt] (c2) at (7.7, 0) {$\scriptstyle{H}$};
\node[circle, inner sep = 0.9pt] (c3) at (8.4, 0) {$\scriptstyle{H}$};
\node[circle, inner sep = 0.9pt] (c4) at (9.1, 0) {$\scriptstyle{T}$};
\draw[->, thick] (c1) to[out=55, in=125] (c2);
\draw[<-, thick] (c3) to[out=55, in=125] (c4);

\node[circle, inner sep = 0.9pt] (d1) at (10.5, 0) {$\scriptstyle{H}$};
\node[circle, inner sep = 0.9pt] (d2) at (11.2, 0) {$\scriptstyle{T}$};
\node[circle, inner sep = 0.9pt] (d3) at (11.9, 0) {$\scriptstyle{T}$};
\node[circle, inner sep = 0.9pt] (d4) at (12.6, 0) {$\scriptstyle{H}$};
\draw[<-, thick] (d1) to[out=55, in=125] (d2);
\draw[->, thick] (d3) to[out=55, in=125] (d4);
\end{tikzpicture}
\end{center}
\caption{The rules for pairs of arrows in the Simion type~$a$ subclass.}
\label{figure_Simion_type_a}
\end{figure}

\begin{proposition}
Let $\mathcal{F}$ be a collection of uniform flag triangulations defined by a
set of rules that contain the following rules:
\begin{enumerate}
\item $THTH$ type of pairs of arrows nest.
\item $HTHT$ type of pairs of arrows do not nest.
\item Both $THHT$ and $HTTH$ types of pairs of arrows do not cross.
 \end{enumerate}
Then the collection $\widehat{\mathcal{F}}$ of families of
saturated faces satisfies:
\begin{align*}
F(\widehat{\mathcal{F}},x,y,z)
& =
\frac{C(yz(z+1))+z}{1+z} \\
&
+ \frac{xz\cdot (1+zC(yz(z+1)))\cdot C(yz(z+1))^{2}}{1+z}
\cdot D(z \cdot C(yz(z+1)), z \cdot C(yz(z+1)), x).
\end{align*}
\label{proposition_simion_subclass_a}  
\end{proposition}  
\begin{proof}
First we show that 
\begin{align}
F(\widehat{\mathcal{F}},x,y,z)
&=
\frac{1}{1+z}\cdot\left(C(yz(z+1))+z\right)
\label{simion-0c-rec} \\
&+
\sum_{a,b\geq 0} xz^{a+b+1} D_{a,b}(x)\cdot
\frac{1+zC(yz(z+1))}{1+z}
\cdot 
C(yz(z+1))^{a+b+2}. 
\nonumber
\end{align}
By equation~\eqref{equation_follow_full} the term
$\frac{1}{1+z}\cdot\left(C(yz(z+1))+z\right)$ accounts for the
possibility of a face containing backward arrows only. In all the other
cases, the face also contains forward arrows. If these arrows are
incident to $a+1$ tails and $b+1$ heads, then by
Proposition~\ref{proposition_nest_full} these contribute
$D_{a,b}(x)\cdot xz^{a+b+1}$. The $a+1$ tails and $b+1$ heads of forward
arrows partition the number line into $a+b+3$ segments. Since $THHT$ and
$HTTH$ type of pairs of arrows do not cross, backward arrows that have at
least one end 
in one of these segments must have both ends in the same segment. Hence
the contribution of the backward arrows may be written as a product of
$a+b+3$ independent factors. Tails of backward arrows whose endpoints
are
contained in 
one of the leftmost $a+1$ segments may coincide with the tail of
a forward arrow. 
On these segments, the total weight of nonempty sets of backward arrows
must be multiplied by $(1+z)$ to account for the possibility of (not)
identifying the rightmost tail of a backward arrow with the tail of a forward
arrow. By equation~\eqref{equation_follow_full} the total weight of nonempty sets
of arrows is $\frac{1}{1+z}\cdot\left(C(yz(z+1))+z\right)-1$. Keeping in
mind also the possibility of not inserting any 
backward arrow between the tails of two forward arrows, or to the left 
of all forward arrows, we obtain that the backward arrows discussed so
far contribute a factor of 
$$
\left(1+(1+z)\cdot\left(
\frac{1}{1+z}\cdot\left(C(yz(z+1))+z\right)-1\right)\right)^{a+1}
=C(yz(z+1))^{a+1}.
$$
Similarly, backward arrows inserted between the $b+1$ heads of forward
arrows or to the right of the heads of all forward arrows contribute a
factor of $C(yz(z+1))^{b+1}$. The only remaining possibility is to
insert backward arrows between the rightmost tail of a forward arrow and
the leftmost head of a forward arrow. Heads and tails of backward arrows
inserted on this segment cannot coincide with the head or tail of a
backward arrow. They contribute a factor of
$$
1+z\cdot\left(\frac{1}{1+z}\cdot\left(C(yz(z+1))+z\right)-1\right)
=
\frac{1+zC(yz(z+1))}{1+z}.
$$
The statement is a direct consequence of~\eqref{simion-0c-rec} and the
definition of $D(u,v,x)$ given in~\eqref{equation_Delannoy}.   
\end{proof}

We now consider the uniform
triangulations which belong to 
the Simion subclass of type~$b$.
Since there is still no restriction on the rules for the $TTHH$
and $HHTT$ types of pairs, at a first glance
this subclass appears to be the largest one. 
This appearance is misleading, as it is closed under taking the
reflected dual triangulations. By Lemma~\ref{lemma_involutions} this
operation takes any uniform flag triangulation in which $THHT$ type
pairs cross and $HTTH$ type pairs do not cross into a uniform flag
triangulation in which $THHT$ type pairs do not cross and $HTTH$ type
pairs cross. As a direct consequence of
Lemma~\ref{lemma_dual_enumeration}, the generating function
$F(\widehat{\mathcal{F}},x,y,z)$ does not change if we take the reverse
of ${\mathcal{F}}$.    

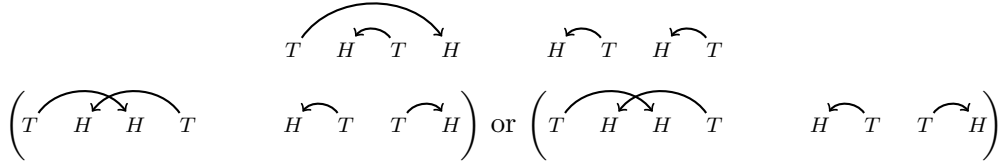
\begin{figure}
\begin{center}
\begin{tikzpicture}
\node[circle, inner sep = 0.9pt] (a1) at (3.5, 1) {$\scriptstyle{T}$};
\node[circle, inner sep = 0.9pt] (a2) at (4.2, 1) {$\scriptstyle{H}$};
\node[circle, inner sep = 0.9pt] (a3) at (4.9, 1) {$\scriptstyle{T}$};
\node[circle, inner sep = 0.9pt] (a4) at (5.6, 1) {$\scriptstyle{H}$};
\draw[->, thick] (a1) to[out=55, in=125] (a4);
\draw[<-, thick] (a2) to[out=55, in=125] (a3);

\node[circle, inner sep = 0.9pt] (b1) at (7.0, 1) {$\scriptstyle{H}$};
\node[circle, inner sep = 0.9pt] (b2) at (7.7, 1) {$\scriptstyle{T}$};
\node[circle, inner sep = 0.9pt] (b3) at (8.4, 1) {$\scriptstyle{H}$};
\node[circle, inner sep = 0.9pt] (b4) at (9.1, 1) {$\scriptstyle{T}$};
\draw[<-, thick] (b1) to[out=55, in=125] (b2);
\draw[<-, thick] (b3) to[out=55, in=125] (b4);

\node[circle, inner sep = 0.9pt] () at (-0.2, 0) {$\biggl ($};

\node[circle, inner sep = 0.9pt] (c1) at (0, 0) {$\scriptstyle{T}$};
\node[circle, inner sep = 0.9pt] (c2) at (0.7, 0) {$\scriptstyle{H}$};
\node[circle, inner sep = 0.9pt] (c3) at (1.4, 0) {$\scriptstyle{H}$};
\node[circle, inner sep = 0.9pt] (c4) at (2.1, 0) {$\scriptstyle{T}$};
\draw[->, thick] (c1) to[out=55, in=125] (c3);
\draw[<-, thick] (c2) to[out=55, in=125] (c4);

\node[circle, inner sep = 0.9pt] (d1) at (3.5, 0) {$\scriptstyle{H}$};
\node[circle, inner sep = 0.9pt] (d2) at (4.2, 0) {$\scriptstyle{T}$};
\node[circle, inner sep = 0.9pt] (d3) at (4.9, 0) {$\scriptstyle{T}$};
\node[circle, inner sep = 0.9pt] (d4) at (5.6, 0) {$\scriptstyle{H}$};
\draw[<-, thick] (d1) to[out=55, in=125] (d2);
\draw[->, thick] (d3) to[out=55, in=125] (d4);

\node[circle, inner sep = 0.9pt] () at (6.3, 0) {$\biggr )$ or $\biggl ($};

\node[circle, inner sep = 0.9pt] (e1) at (7.0, 0) {$\scriptstyle{T}$};
\node[circle, inner sep = 0.9pt] (e2) at (7.7, 0) {$\scriptstyle{H}$};
\node[circle, inner sep = 0.9pt] (e3) at (8.4, 0) {$\scriptstyle{H}$};
\node[circle, inner sep = 0.9pt] (e4) at (9.1, 0) {$\scriptstyle{T}$};
\draw[->, thick] (e1) to[out=55, in=125] (e3);
\draw[<-, thick] (e2) to[out=55, in=125] (e4);

\node[circle, inner sep = 0.9pt] (f1) at (10.5, 0) {$\scriptstyle{H}$};
\node[circle, inner sep = 0.9pt] (f2) at (11.2, 0) {$\scriptstyle{T}$};
\node[circle, inner sep = 0.9pt] (f3) at (11.9, 0) {$\scriptstyle{T}$};
\node[circle, inner sep = 0.9pt] (f4) at (12.6, 0) {$\scriptstyle{H}$};
\draw[<-, thick] (f1) to[out=55, in=125] (f2);
\draw[->, thick] (f3) to[out=55, in=125] (f4);

\node[circle, inner sep = 0.9pt] () at (12.8, 0) {$\biggr )$};
\end{tikzpicture}
\end{center}
\caption{The rules for pairs of arrows in the Simion type~$b$ subclass.}
\label{figure_largest_Simion_subclass}
\end{figure}

\begin{proposition}
Let $\mathcal{F}$ be a family of uniform flag triangulations defined by a
set of rules that contains the following requirements:
\begin{enumerate}
\item $THTH$ type of pairs of arrows nest.
\item $HTHT$ type of pairs of arrows do not nest.
\item Exactly one of the $THHT$ and $HTTH$ types of pairs of arrows cross.
\end{enumerate}
Then the collection $\widehat{\mathcal{F}}$ of families of
saturated faces satisfies
\begin{align*}
F(\widehat{\mathcal{F}},x,y,z)
&=
\frac{1}{1+z}\cdot
\left(C(yz(z+1))+z\right)\\
&+
D\left(C(yz(z+1)) z,\frac{z}{1-yz(z+1)C(yz(z+1))},x\right) 
\\
& \cdot \frac{xz \cdot C(yz(z+1)) \cdot (1-yzC(yz(z+1)))}{(1-yz(z+1) \cdot C(yz(z+1)))^{2}}.
\end{align*}
\label{proposition_simion_subclass_b}
\end{proposition}  
\begin{proof}
The proof is similar to the proof of Proposition~\ref{proposition_simion_subclass_a}.
As a consequence of Lemmas~\ref{lemma_involutions}
and~\ref{lemma_dual_enumeration}, 
without loss of generality we may assume that $THHT$ type of pairs do not 
cross and $HTTH$ type of pairs cross. We first prove that  
\begin{align}
\label{equation_simion-1c-rec}  
F(\widehat{\mathcal{F}},x,y,z)
&=
\frac{1}{1+z}\cdot
\left(C(yz(z+1))+z\right)\\
&+
\sum_{a,b\geq 0} xz^{a+b+1} D_{a,b}(x)\cdot C(yz(z+1))^{b+1}
\nonumber \\
&\cdot \left(1+\sum _{i\geq 1}
\frac{(yz(z+1)\cdot C(yz(z+1)))^{i}}{z+1}\cdot\left(\binom{i+a+1}{a+1}z 
+\binom{i+a}{a}\right)\right). 
\nonumber
\end{align}
Just as in the proof of Proposition~\ref{proposition_simion_subclass_a},  the term
$\frac{1}{1+z}\cdot\left(C(yz(z+1))+z\right)$ accounts for the 
possibility of a face containing backward arrows only. In all the other
cases, forward arrows that are incident to $a+1$ tails and $b+1$ heads
contribute a factor of $D_{a,b}(x)\cdot xz^{a+b+1}$. The $b+1$ heads of
the forward arrows partition the number line into $b+2$ segments.
Since $THHT$ type of pairs of arrows do not cross, backward arrows that
have at least one end between the heads of two forward arrows or to the
right of the largest tail of a forward arrow must have both ends in the
same position. Backward arrows contained in the right $b+1$ segments
contribute a factor of $C(yz(z+1))^{b+1}$, just as in the proof of
Proposition~\ref{proposition_simion_subclass_a}.

It remains to consider the possibility of having backward arrows that
are entirely to the left of the head of any forward arrow. Let us list
the endpoints of these backward arrows in increasing 
order. This list must begin with a positive number of heads, followed by
a tail.
Let $i$ be the number of heads of backward arrows preceding all tails.
By equation~\eqref{equation_follow_refined_full}, the total weight of
these backward arrows is $\frac{1}{1+z} \cdot
\left(y z (1+z) C(yz(z+1))\right)^{i}$.  
Since $HTTH$ type of pairs of arrows cross, the $a+1$ tails of
the forward arrows  must all appear before the first tail of a backward
arrow, only the rightmost of them may coincide with the leftmost tail of
a backward arrow. There are $\binom{i+a+1}{a+1}$ ways to insert the
tails of the forward arrows strictly in front of the leftmost tail of a
backward arrow, and there are $\binom{i+a}{a}$ ways to perform this
insertion if the rightmost tail of a forward arrow is equal to the
leftmost head of a backward arrow. The contribution of these arrows
is the sum after $1$ on the last line of~\eqref{equation_simion-1c-rec}.  

Observe that by applying the identity
$\sum_{i \geq 1} \binom{i+m}{m} \cdot t^{i} = {1}/{(1-t)^{m+1}} - 1$
twice to the factor appearing
on the last line of~\eqref{equation_simion-1c-rec},
we can rewrite this factor as
$$
1+\frac{z}{z+1}\cdot \left(\frac{1}{(1-yz(z+1)C(yz(z+1)))^{a+2}}-1\right)
+\frac{1}{z+1}\left(\frac{1}{(1-yz(z+1)C(yz(z+1))^{a+1}}-1\right).
$$
Simplifying this expression, including canceling a factor of $z+1$
in the numerator and the denominator, yields
$$
\frac{1}{(1-yz(z+1)C(yz(z+1))^{a}} \cdot
\frac{(1-yzC(yz(z+1)))}{(1-yz(z+1)C(yz(z+1)))^{2}}. 
$$
Finally, using~\eqref{equation_Delannoy}, we see that
equation~\eqref{equation_simion-1c-rec}
simplifies to the desired expression in the proposition.
\end{proof}

We now examine the Simion subclass of type~$c$.
This is the smallest subclass, as by
Proposition~\ref{proposition_nest_follow}
the $TTHH$ and $HHTT$ types of pairs of arrows must nest.  

\begin{figure}
\begin{center}
\begin{tikzpicture}
\node[circle, inner sep = 0.9pt] (a1) at (0, 0) {$\scriptstyle{T}$};
\node[circle, inner sep = 0.9pt] (a2) at (0.7, 0) {$\scriptstyle{H}$};
\node[circle, inner sep = 0.9pt] (a3) at (1.4, 0) {$\scriptstyle{T}$};
\node[circle, inner sep = 0.9pt] (a4) at (2.1, 0) {$\scriptstyle{H}$};
\draw[->, thick] (a1) to[out=55, in=125] (a4);
\draw[<-, thick] (a2) to[out=55, in=125] (a3);

\node[circle, inner sep = 0.9pt] (b1) at (3.5, 0) {$\scriptstyle{H}$};
\node[circle, inner sep = 0.9pt] (b2) at (4.2, 0) {$\scriptstyle{T}$};
\node[circle, inner sep = 0.9pt] (b3) at (4.9, 0) {$\scriptstyle{H}$};
\node[circle, inner sep = 0.9pt] (b4) at (5.6, 0) {$\scriptstyle{T}$};
\draw[<-, thick] (b1) to[out=55, in=125] (b2);
\draw[<-, thick] (b3) to[out=55, in=125] (b4);

\node[circle, inner sep = 0.9pt] (c1) at (7.0, 0) {$\scriptstyle{T}$};
\node[circle, inner sep = 0.9pt] (c2) at (7.7, 0) {$\scriptstyle{H}$};
\node[circle, inner sep = 0.9pt] (c3) at (8.4, 0) {$\scriptstyle{H}$};
\node[circle, inner sep = 0.9pt] (c4) at (9.1, 0) {$\scriptstyle{T}$};
\draw[->, thick] (c1) to[out=55, in=125] (c3);
\draw[<-, thick] (c2) to[out=55, in=125] (c4);

\node[circle, inner sep = 0.9pt] (d1) at (10.5, 0) {$\scriptstyle{H}$};
\node[circle, inner sep = 0.9pt] (d2) at (11.2, 0) {$\scriptstyle{T}$};
\node[circle, inner sep = 0.9pt] (d3) at (11.9, 0) {$\scriptstyle{T}$};
\node[circle, inner sep = 0.9pt] (d4) at (12.6, 0) {$\scriptstyle{H}$};
\draw[<-, thick] (d1) to[out=55, in=125] (d3);
\draw[->, thick] (d2) to[out=55, in=125] (d4);

\node[circle, inner sep = 0.9pt] (e1) at (3.5, -1) {$\scriptstyle{T}$};
\node[circle, inner sep = 0.9pt] (e2) at (4.2, -1) {$\scriptstyle{T}$};
\node[circle, inner sep = 0.9pt] (e3) at (4.9, -1) {$\scriptstyle{H}$};
\node[circle, inner sep = 0.9pt] (e4) at (5.6, -1) {$\scriptstyle{H}$};
\draw[->, thick] (e1) to[out=55, in=125] (e4);
\draw[->, thick] (e2) to[out=55, in=125] (e3);

\node[circle, inner sep = 0.9pt] (f1) at (7.0, -1) {$\scriptstyle{H}$};
\node[circle, inner sep = 0.9pt] (f2) at (7.7, -1) {$\scriptstyle{H}$};
\node[circle, inner sep = 0.9pt] (f3) at (8.4, -1) {$\scriptstyle{T}$};
\node[circle, inner sep = 0.9pt] (f4) at (9.1, -1) {$\scriptstyle{T}$};
\draw[<-, thick] (f1) to[out=55, in=125] (f4);
\draw[<-, thick] (f2) to[out=55, in=125] (f3);
\end{tikzpicture}
\end{center}
\caption{The rules for pairs of arrows in the Simion type~$c$ subclass.}
\label{figure_smallest_Simion_subclass}
\end{figure}
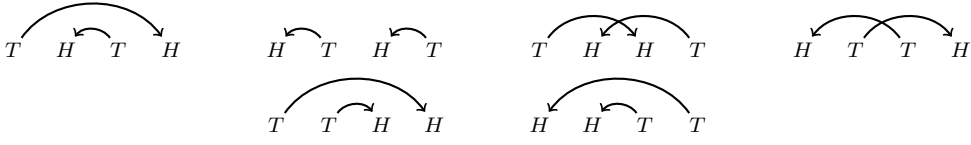

\begin{proposition}
Let $\mathcal{F}$ be a family of uniform flag triangulations defined by
the following rules:
\begin{enumerate}
\item $THTH$ type of pairs of arrows nest.
\item $HTHT$ type of pairs of arrows do not nest.
\item Both $THHT$ and $HTTH$ types of pairs of arrows cross.
\item Both $TTHH$ and $HHTT$ types of pairs of arrows nest. 
\end{enumerate}
Then the collection $\widehat{\mathcal{F}}$ of families of
saturated faces satisfies:
\begin{align*}
F(\widehat{\mathcal{F}},x,y,z)
&=
\frac{1}{1+z} \cdot \left(C(yz(z+1))+z\right) \\
&+
D\left(\frac{z}{1 - yz(z+1) \cdot C(yz(z+1))},\frac{z}{1 - yz(z+1) \cdot C(yz(z+1))},x\right)
\\  
&\cdot \frac{xz (z+1-yz(z+1) \cdot C(yz(z+1)))^{2}}{(1+z)
  (1+z \cdot C(yz(z+1))) \cdot (1-yz(z+1) \cdot C(yz(z+1)))^{4}} .
\end{align*}
\label{proposition_simion_subclass_c}
\end{proposition}
\begin{proof}
The proof is similar to the proof of
Proposition~\ref{proposition_simion_subclass_b} in many details. We will highlight the
substantial differences. First we show the following equality:
\begin{align}
\label{equation_simion-2c-rec}  
F(\widehat{\mathcal{F}},x,y,z)
&=
\frac{1}{1+z}\cdot
\left(C(yz(z+1))+z\right)\\
&+
\sum_{a,b\geq 0} xz^{a+b+1} D_{a,b}(x)\cdot \frac{1}{1+z((C(yz(z+1))+z)/(z+1)-1)}
\nonumber \\
&\cdot \left(1+\sum _{i\geq 1} 
\frac{(yz(z+1)\cdot C(yz(z+1)))^{i}}{1+z}\cdot\left(\binom{i+a+1}{a+1}z 
+\binom{i+a}{a}\right)\right)
\nonumber \\
&\cdot \left(1+\sum _{j\geq 1} 
\frac{(yz(z+1)\cdot C(yz(z+1)))^{j}}{1+z}\cdot\left(\binom{j+b+1}{b+1}z 
+\binom{j+b}{b}\right)\right). 
\nonumber
\end{align}
Just as in equation~\eqref{equation_simion-1c-rec}, the term
$\frac{1}{1+z}\cdot\left(C(yz(z+1))+z\right)$ results from the
faces containing backward arrows only. The second sum is contributed by
faces that contain forward arrows as well; these forward arrows are
incident to $a+1$ tails and $b+1$ heads. The total contribution of the
forward arrows is $xz^{a+b+1} D_{a,b}(x)$. For the precise count of the
contribution of the backward arrows, we use the fact that no
pair of backward arrows crosses. We call a saturated face $\{(i_{1},j_{1}),
(i_{2},j_{2}), \ldots, (i_{k},j_{k})\}$ of backward arrows {\em connected} if it
is empty or the arrow $(\max(i_{1},i_{2}, \ldots, i_{k}),
\min(j_{1},j_{2}, \ldots, j_{k}))$ belongs to the set. Clearly each saturated face of backward arrows is uniquely the disjoint
union of maximal connected sets. Introducing $\widehat{{\mathcal G}}$ as the
collection of families of connected saturated sets of backward arrows,
we have the equality
\begin{align*}
1+z\cdot (F(\widehat{\mathcal{F}},0,y,z)-1)
& =
\sum_{k\geq 0} F(\widehat{\mathcal{G}},0,y,z)^{k} \cdot z^{k}
=
\frac{1}{1 - z \cdot F(\widehat{\mathcal{G}},0,y,z)},
\end{align*}
where $k$ stands for the number of maximal connected sets. Indeed, for
$k=0$ the empty set is saturated and connected. For all nonempty
saturated sets, the first component contributes
an unnecessary additional factor of $z$ on the right-hand side.  
Substituting~\eqref{equation_follow_full} yields
\begin{align}
1+z\cdot\left(
\frac{1}{1+z}\cdot\left(C(yz(z+1))+z\right)-1\right)
& =
\frac{1}{1-z\cdot F(\widehat{\mathcal G},0,y,z)} .
\label{equation_connected}  
\end{align}
We partition the backward arrows of a saturated face into three
classes. The first class is formed by all backward arrows whose head is
weakly to the left of the head of some forward arrow. Since $HTTH$ type of
pairs of arrow cross, the tail of such a backward arrow is to the right
of all backward arrows. The same condition on the $HTTH$ type of pairs of
arrows also guarantees that to the left of any tail of a forward arrow
we can only have a head of a backward arrow belonging to the first
class.  Since $HTHT$ type of pairs do not nest, the tail of
a backward arrow in the first class is between the heads and tails of
all forward arrows. All such backward arrows form a connected component:
they all contain or arch over the rightmost tail of a forward arrow, and any
backward arrow that does not contain or arch over this rightmost tail
has its head to the right of all backward arrows in the first
class. The same reasoning also shows that all heads of backward arrows
in the first class are to the left of the tails of these arrows. 
Introducing $i$ as the number of heads of backward arrows in the
first class, the contribution of all backward arrows in the first class
is
\begin{align*}
\left(1+\sum _{i\geq 1} 
\frac{(yz(z+1)\cdot C(yz(z+1)))^{i}}{1+z}\cdot\left(\binom{i+a+1}{a+1}z 
+\binom{i+a}{a}\right)\right)\cdot (1-zF(\widehat{\mathcal G},0,y,z)).
\end{align*}
Just as in the proof of Proposition~\ref{proposition_simion_subclass_b}, each summand in
the first factor of the above is the total weight of all saturated faces of backward
arrows in which $i$ heads are followed by a tail in the left-to-right
order, and the factor of $1-zF(\widehat{\mathcal G},0,y,z)$ represents
dividing by $1/(1-zF(\widehat{\mathcal G},0,y,z))$, i.e., removing the
contribution of the additional connected components. Hence the above
expression represents the total weight of connected saturated faces.

The second class is formed by all backward arrows whose head is weakly
to the right of the head of some forward arrow. A completely analogous
reasoning shows that the total weight of these arrows is
\begin{align*}
\left(1+\sum _{j\geq 1} 
\frac{(yz(z+1)\cdot C(yz(z+1)))^{j}}{1+z}\cdot\left(\binom{j+b+1}{b+1}z 
+\binom{j+b}{b}\right)\right)\cdot (1-zF(\widehat{\mathcal G},0,y,z)).
\end{align*}
The remaining arrows form the third class: the heads and tails of these
arrows are to the right to the tails of the arrows in the first class
and to the left of the heads of the arrows in the second class.
They contribute a factor of
\begin{align*}
1 + z \cdot (F(\widehat{\mathcal{F}},0,y,z)-1)
& =
1 + z \cdot \left(\frac{1}{1+z}\cdot\left(C(yz(z+1))+z\right)-1\right).
\end{align*}
Equation~\eqref{equation_simion-2c-rec} is now a consequence
of~\eqref{equation_connected}.     
The algebraic manipulations used to derive the statement
from~\eqref{equation_simion-2c-rec} are very similar to the proof of
Proposition~\ref{proposition_simion_subclass_b}, and are therefore omitted.
\end{proof}

\begin{proof}[Proof of 
Theorem~\ref{theorem_enumeration_in_the_simion_class}.]
We begin with the proof
of~\eqref{equation_enumeration_in_the_simion_class}.
We have proved three
variants of this formula in 
Propositions~\ref{proposition_simion_subclass_a},
\ref{proposition_simion_subclass_b}
and~\ref{proposition_simion_subclass_c}.
It remains to show that
the three generating functions in
these propositions are equal to
the generating function in
equation~\eqref{equation_enumeration_in_the_simion_class}.
This is straightforward by 
expanding $D(u,v,x)$ using equation~\eqref{equation_Delannoy}
and using the fact that the Catalan generating function
satisfies the quadratic relation
$C(u) = 1 + u \cdot C(u)^{2}$,
especially in the form
$1/(1 - u \cdot C(u)) = C(u)$.
Equation~\eqref{equation_enumeration_in_the_dual_simion_class}
follows now from
equation~\eqref{equation_enumeration_in_the_simion_class} 
by applying the involution
$\triangle \longmapsto \triangle^{*}$
and Lemmas~\ref{lemma_involutions} and~\ref{lemma_dual_enumeration}.
\end{proof}

By combining
Theorem~\ref{theorem_enumeration_in_the_simion_class}
and Lemma~\ref{lemma_all_full} it is possible to
give the generating function of all faces.
We now explicitly count the facets using
Corollary~\ref{corollary_facets} and
equation~\eqref{equation_enumeration_in_the_simion_class}. 
Since
\begin{align*}
F\left(\widehat{\mathcal{F}},\frac{x}{w},\frac{y}{w},wz\right)
& =
\frac{C(yz(wz+1))+wz}{1+wz} 
+
\frac{xz\cdot (1+wzC(yz(wz+1)))\cdot C(yz(wz+1))^{2}}
{(1+wz)(1-2C(yz(wz+1))xz-C(yz(wz+1))^{2}xwz^{2})}, 
\end{align*}
we conclude that the facet generating function is given by
\begin{align}
\sum_{n \geq 0} \sum_{i=0}^{n} f(\triangle_{n},i,n-i) x^{i}y^{n-i}z^{n}
& =
C(yz) + \frac{xz\cdot C(yz)^{2}}{1-2xzC(yz)}
\label{equation_simion_0c_facetgen}  \\
& =
C(yz)
+
\sum_{i \geq 1}
2^{i-1} \cdot (xz)^{i} \cdot C(yz)^{i+1}
\nonumber .
\end{align}

\begin{theorem}
Let $\triangle_{n}$ be a uniform flag triangulation of
the boundary $\partial P_{n}$ of the root polytope
belonging to the Simion class satisfying the property that $THTH$ type of
pairs of arrows nest and $HTHT$ type of arrows do not nest.
Then for $i \geq 1$, the number of facets
of $\triangle_{n}$
consisting of $i$ forward arrows and
$n-i$ backward arrows is given by
\begin{align}
f(\triangle_{n}, i, n-i)
& =
2^{i-1} \cdot \frac{(i+1) \cdot (2n-i)!}{(n-i)! \cdot (n+1)!} .
\label{equation_number_of_facets_in_Rodicatope}
\end{align}
The number of facets
of the triangulation $\triangle_{n}$
consisting of no forward arrows and
$n$ backward arrows is given by
the Catalan number $C_{n}$.
\label{theorem_facet_enumeration_class_closest_to_Rodicatope}
\end{theorem}
\begin{proof}
Observe that the coefficient of $y^{n}z^{n}$
in equation~\eqref{equation_simion_0c_facetgen}
is the $n$th Catalan number $C_{n}$.
For $i \geq 1$, the coefficient of
$x^{i} y^{n-i} z^{n} = (xz)^{i} \cdot (yz)^{n-i}$
is
$2^{i-1}$ times
the coefficient of
$(yz)^{n-i}$
in
$C(yz)^{i+1}$,
which is
$2^{i-1} \cdot \frac{i+1}{2n-i+1} \cdot \binom{2n-i+1}{n+1}$
by an identity due to Catalan~\cite[Formula (19)]{Catalan}.
\end{proof}
\begin{remark}
{\em The formula (\ref{equation_number_of_facets_in_Rodicatope}) given in
  Theorem~\ref{theorem_facet_enumeration_class_closest_to_Rodicatope}
  may be restated as
\begin{align}
f(\triangle_{n}, i, n-i)
& =
2^{i-1} \cdot C(n,n-i)
\end{align}
where the numbers
$$C(n,k)=\frac{(n+k)! \cdot (n-k+1)}{k! \cdot (n+1)!}$$
are the entries in the {\em Catalan triangle}.
See OEIS sequence~A009766~\cite{OEIS}.  
}    
\end{remark}

We end with two observations.
First, it is amusing how the expression
in~\eqref{equation_number_of_facets_in_Rodicatope}
is off by a factor of $1/2$ in the case when $i=0$.
Second, when $\triangle_{n}$ is
the Simion type~$B$ associahedron triangulation
of the boundary of the root polytope,
it is possible to give a more
constructive proof of
Theorem~\ref{theorem_facet_enumeration_class_closest_to_Rodicatope}
by analyzing the tree structure of the digraphs
corresponding to facets.

\section{Face enumeration in the revlex class}
\label{section_revlex}

In this section we study the revlex class,
that is, the class containing the revlex pulling triangulation.
Our first result is similar to equations~\eqref{simion-0c-rec},
\eqref{equation_simion-1c-rec} and~\eqref{equation_simion-2c-rec}. The
revles class is perfectly suitable to compute the face numbers with a prescribed
number of forward and backward arrows. Unfortunately, it does not seem
feasible to produce a closed form formula without infinite sums, that is
similar to
Propositions~\ref{proposition_simion_subclass_a},
\ref{proposition_simion_subclass_b}
and~\ref{proposition_simion_subclass_c}.  

\begin{theorem}
Let $\mathcal{F}$ be a family of uniform flag triangulations defined by a
set of rules that contains the following rules:
\begin{enumerate}
\item Both $THTH$ and $HTHT$ types of pairs of arrows nest.
\item Both $HTTH$ and $THHT$ types of pairs of arrows cross.
\end{enumerate}
Then the collection $\widehat{\mathcal{F}}$ of families of
saturated faces satisfies
\begin{align}  
\label{equation_revlex_rec}
F(\widehat{\mathcal{F}},x,y,z)
&=
1
+
\sum_{a,b\geq 0}
(x\cdot D_{a,b}(x) + y\cdot D_{a,b}(y)) 
\cdot
z^{a+b+1}
\\
&+
xy
\cdot
\sum_{a^{\prime},b^{\prime},a^{\prime\prime},b^{\prime\prime}\geq 0} D_{a^{\prime},b^{\prime}}(x)
\cdot
D_{a^{\prime\prime},b^{\prime\prime}}(y)
\cdot
z^{a^{\prime}+b^{\prime}+a^{\prime\prime}+b^{\prime\prime}+2}
\cdot 
C(a^{\prime},b^{\prime},a^{\prime\prime},b^{\prime\prime},z)
\nonumber
\end{align}
where
\begin{align}  
C(a^{\prime},b^{\prime},a^{\prime\prime},b^{\prime\prime},z)
&=
\binom{a^{\prime}+b^{\prime\prime}+2}{a^{\prime}+1}
\cdot
\binom{a^{\prime\prime}+b^{\prime}+2}{b^{\prime}+1} \cdot z \\
&+
\binom{a^{\prime}+b^{\prime\prime}+1}{b^{\prime\prime}}
\cdot
\binom{a^{\prime\prime}+b^{\prime}+1}{b^{\prime}}
+
\binom{a^{\prime}+b^{\prime\prime}+1}{a^{\prime}}
\cdot
\binom{a^{\prime\prime}+b^{\prime}+1}{a^{\prime\prime}}.
\nonumber
\end{align}
\label{theorem_revlex}  
\end{theorem}  
\begin{proof}
By Proposition~\ref{proposition_nest_full} the first three terms on the
right-hand side of~\eqref{equation_revlex_rec} are the total weights of
  all faces that do not contain arrows in both directions. The last sum is the total
weight of all faces containing arrows in both directions: forward arrows on
$a^{\prime}+1$ tails and $b^{\prime}+1$ heads and backward arrows on $a^{\prime\prime}+1$
tails and $b^{\prime\prime}+1$ heads. Again, by
Proposition~\ref{proposition_nest_full} the subfaces of forward and
backward arrows, respectively, contribute factors of $xD_{a^{\prime},b^{\prime}}(x)
z^{a^{\prime}+b^{\prime}+1}$ and $yD_{a^{\prime\prime},b^{\prime\prime}}(y)z^{a^{\prime\prime}+b^{\prime\prime}+1}$
respectively. We may collect the contribution of all faces that
contain only forward or only backward arrows by identifying $a^{\prime}$
and $a^{\prime\prime}$ with $a$, and $b^{\prime}$ and $b^{\prime\prime}$
with $b$. For the remaining faces there is an additional factor of $z$
when the set of 
endpoints of forward arrows is disjoint from the set of endpoints of
backward arrows. By Proposition~\ref{proposition_nest} $THHT$ and $HTTH$
type of pairs of arrows cross. As a consequence, heads of backward arrows
are to the left of the heads of forward arrows, and tails of backward
arrows are to the right of the tails of the backward arrows. These
conditions also ensure that the set of endpoints of the backward arrows cannot
have two or more nodes in common with the set of endpoints of the
forward arrows. The first term factor $C(a^{\prime},b^{\prime},a^{\prime\prime},b^{\prime\prime},z)$
accounts for the number of ways we may line up $a^{\prime}+1$ tails of
forward arrows with $b^{\prime\prime}+1$ heads of backward arrows on one
side, and 
independently, $a^{\prime\prime}+1$ tails of
backward arrows with $b^{\prime}+1$ heads of forward arrows on the other
side. The remaining terms correspond to the cases when the forward
arrows and the backward arrows share one head or one tail, respectively.  
 \end{proof}  

We obtain a more compact expression using the proof of
Theorem~\ref{theorem_revlex} by introducing the following generating
function.

\begin{definition}
Let $\widehat{{\mathcal F}}=(\widehat{{\mathcal F}}_{0},
\widehat{{\mathcal F}}_{1}, \ldots)$ be a collection of families of arrows
such that for each $n$ the family $\widehat{{\mathcal F}_{n}}$ consists of
saturated subsets of $V_{n}$. We define the {\em node-enriched
  exponential generating function} of $\widehat{\mathcal{F}}$ as follows: 
\begin{enumerate}
\item
The empty set (if it belongs to $\widehat{{\mathcal F}}_{0}$)
contributes a factor of $1$.
\item
Each nonempty $\sigma\in \widehat{{\mathcal F}}_{n}$ contributes a term
$$
x^{i} y^{j} \cdot \frac{u^{a+1} \cdot v^{b+1}}{(a+1)! \cdot (b+1)!} \cdot t^{n},
$$
where $i$ is the number of forward arrows, $j$ is the number of backward
arrows, $a+1$ is the number of nodes that are left ends of arrows and
$b+1$ is the number of nodes that are right ends of arrows.   
\end{enumerate}  
\end{definition}  

It should be noted that the numbers $a+1$ and $b+1$
respectively count the left and
right ends of arrows and not their heads or tails: a left end is the
tail of a forward arrow or the head of a backward arrow. A common tail
of a forward and a backward arrow is counted twice: once as a left end
and once as a right end. It is easy to derive from the requirements on
the $THHT$ and $HTTH$ type of pairs of arrows that for the triangulations
in the revlex class, there is at most one node that is simultaneously
the left end and the right end of some arrow.

The node-enriched exponential generating function of the saturated
faces in a triangulation in the revlex class has a compact expression in
terms of the following exponential generating function of the Delannoy
polynomials:
\begin{equation}
\Ddd(u,v,x)
=
\sum_{a,b\geq 0} \frac{D_{a,b}(x) \cdot u^{a+1} \cdot v^{b+1}}{(a+1)! \cdot (b+1)!}. 
\end{equation}  

\begin{theorem}
Let $\mathcal{F}$ be a family of uniform flag triangulations defined by a
set of rules that contains the following rules:
\begin{enumerate}
\item Both $THTH$ and $HTHT$ types of pairs of arrows nest.
\item Both $HTTH$ and $THHT$ types of pairs of arrows cross.
\end{enumerate}
Then the node-enriched exponential generating function of the collection
$\widehat{\mathcal F}$ of families of saturated faces is given by
\begin{align*}
&1+\frac{1}{z} \cdot\Ddd(uz,vz,x)+\frac{1}{z}\cdot \Ddd(vz,uz,y)+\frac{1}{z}
  \cdot\Ddd(uz,vz,x)\cdot\Ddd(vz,uz,y)\\
  &+\frac{1}{z^{2}}\cdot\frac{\partial}{\partial u} \Ddd(uz,vz,x)\cdot
  \frac{\partial}{\partial v}\Ddd(vz,uz,y)
  +\frac{1}{z^{2}}\cdot\frac{\partial}{\partial v} \Ddd(uz,vz,x)\cdot
  \frac{\partial}{\partial u}\Ddd(vz,uz,y). 
\end{align*}  
\label{theorem_revlex_exp}
\end{theorem}  
\begin{proof}
The proof is essentially the same as that of
Theorem~\ref{theorem_revlex}  and is thus omitted. 
\end{proof}  

Theorem~\ref{theorem_revlex_exp} motivates computing $\Ddd(u,v,x)$
explicitly.

\begin{theorem}
The exponential generating function $\Ddd(u,v,x)$ is given by
$$
\Ddd(u,v,x)
=
\sum_{k \geq 0}
\frac{(uv \cdot (x^{2}+x))^k}{k!^{2}} \cdot \psi_{k+1}(ux) \cdot \psi_{k+1}(vx)  
$$
where $\psi_{k+1}(z)=\frac{d^{k}}{dz^k} \left(\frac{e^z-1}{z}\right)$.  
\label{theorem_psi}
\end{theorem}  
\begin{proof}
We use the identity
$$
D_{a,b}(x) = \sum_{k} \binom{a}{k} \cdot \binom{b}{k} \cdot (x^{2}+x)^{k} \cdot x^{a+b-2k}.
$$
Here $k$ counts the total number of $NE$ steps and ``northwest corners''
(i.e., $N$ steps immediately followed by $E$ steps) in a Delannoy path
from $(0,0)$ to $(a,b)$. There are $\binom{a}{k}\binom{b}{k}$ ways to
select the positions of these steps and corners in the plane. Each
such place contributes a factor of $x^{2}+x$
as the weight of a $N$ step
followed by an $E$ step is $x^{2}$, whereas the weight of a $NE$ step is $x$.
Using the above expression for $D_{a,b}(x)$, the definition of
$\Ddd(u,v,x)$ may be rearranged as follows:
$$
\Ddd(u,v,x)
=
\sum_{k\geq 0}
\frac{(uv \cdot (x^{2}+x))^k}{k!^{2}}
\cdot
\sum_{a,b\geq k}
\frac{(ux)^{a-k}}{(a-k)!} \cdot \frac{(vx)^{b-k}}{(b-k)!} \cdot
\frac{1}{a+1} \cdot \frac{1}{b+1}.
$$
The statement follows after noticing that
$$
\psi_{k+1}(z) = \sum_{n\geq 0} \frac{1}{n+k+1} \cdot \frac{z^n}{n!}
$$
which can easily be shown by induction on $k$. 
\end{proof}  

\begin{remark}
  {\em
    It is a direct consequence of Theorems~\ref{theorem_revlex_exp} and
    \ref{theorem_psi} that
    $$
    \frac{\partial}{\partial u}\frac{\partial}{\partial v} \Ddd(u,v,x)=
    \exp(x \cdot (u+v))\cdot I_{0}\left(2\sqrt{(x^{2}+x) \cdot uv}\right)
    $$
  where $I_{0}(z)$ is the modified Bessel function of the first
  kind~\cite{McLachlan}.   
   } 
\end{remark}

\begin{remark}
{\em
It can be shown by induction that
$$ \psi_{k}(z)
=
\frac{
\left(\sum_{i=0}^{k-1} (-1)^{i} \cdot \frac{(k-1)!}{(k-1-i)!} \cdot z^{k-1-i}\right)
\cdot e^{z}
+
(-1)^{k} \cdot (k-1)!}
{z^k}. $$
}
\end{remark}

We conclude this section by counting the facets using
Corollary~\ref{corollary_facets}. 

\begin{theorem}
Let $\triangle_{n}$ be a uniform flag triangulation of the boundary
$\partial P_{n}$ 
of the root polytope that belongs to the revlex class.  
For $1 \leq k \leq n-1$, the number of facets
consisting of $k$ forward arrows and
$n-k$ backward arrows, that is, $f(\triangle_{n},k,n-k)$, is given by
\begin{align*}
\sum_{i=1}^{k}
\sum_{j=1}^{n-k}
\binom{k-1}{i-1}
\cdot
\binom{n-k-1}{j-1}
\cdot
\left[
\binom{n-k+i-j}{i}
\cdot
\binom{k-i+j}{j}
+
\binom{n-k+i-j}{i-1}
\cdot
\binom{k-i+j}{j-1}
\right].
\end{align*}
The number of facets with $n$ forward arrows and no backward arrows;
and
the number of facets with no forward arrows and $n$ backward arrows
are both equal to $2^{n-1}$.
\label{theorem_reverse_lexicographic_refined_facet_enumeration}
\end{theorem}
\begin{proof}
Inspecting~\eqref{equation_revlex_rec}
we see that the total degree of $x$ and $y$
is strictly less than the degree of $z$, except for the contributions,
in which the following rules are observed:
\begin{enumerate}
\item
In the expansion of $D_{a^{\prime},b^{\prime}}(x)$ only the contribution of
those Delannoy paths are kept which contain no NE steps. Hence, to
compute the contribution of the facets only, we must replace each
appearance of $D_{a^{\prime},b^{\prime}}(x)$ in \eqref{equation_revlex_rec} with
$\binom{a^{\prime}+b^{\prime}}{a^{\prime}}\cdot x^{a^{\prime}+b^{\prime}}$. 
\item
Similarly, we must replace each
appearance of $D_{a^{\prime\prime},b^{\prime\prime}}(y)$ in~\eqref{equation_revlex_rec} 
with
$\binom{a^{\prime\prime}+b^{\prime\prime}}{a^{\prime\prime}}\cdot y^{a^{\prime\prime}+b^{\prime\prime}}$. 
\item
Only the $z$-free part of the factor
$C(a^{\prime},b^{\prime},a^{\prime\prime},b^{\prime\prime},z)$ contributes to the calculation of the
contribution of the facets.   
\end{enumerate}
Therefore we obtain
\begin{align}
\label{equation_revlex_facets}
\rest{F\left(\widehat{\mathcal{F}},\frac{x}{w},\frac{y}{w},wz\right)}{w=0}
&=
1
+
\sum_{a^{\prime},b^{\prime}\geq 0}
\binom{a^{\prime}+b^{\prime}}{a^{\prime}}  \cdot (xz)^{a^{\prime}+b^{\prime}+1}
+
\sum_{a^{\prime\prime},b^{\prime\prime}\geq 0}
\binom{a^{\prime\prime}+b^{\prime\prime}}{a^{\prime\prime}}
\cdot (yz)^{a^{\prime\prime}+b^{\prime\prime}+1} \\
&+
\sum_{\substack{a^{\prime},b^{\prime} \geq 0 \\ a^{\prime\prime},b^{\prime\prime}\geq 0}}
C_{0}(a^{\prime},b^{\prime},a^{\prime\prime},b^{\prime\prime})
\cdot
\binom{a^{\prime}+b^{\prime}}{a^{\prime}}
\cdot (xz)^{a^{\prime}+b^{\prime}+1}
\cdot
\binom{a^{\prime\prime}+b^{\prime\prime}}{a^{\prime\prime}}
\cdot (yz)^{a^{\prime\prime}+b^{\prime\prime}+1}
\nonumber
\end{align}
where
\begin{align}  
C_{0}(a^{\prime},b^{\prime},a^{\prime\prime},b^{\prime\prime})
&=
\binom{a^{\prime}+b^{\prime\prime}+1}{b^{\prime\prime}}
\cdot
\binom{a^{\prime\prime}+b^{\prime}+1}{b^{\prime}}
+
\binom{a^{\prime}+b^{\prime\prime}+1}{a^{\prime}}
\cdot
\binom{a^{\prime\prime}+b^{\prime}+1}{a^{\prime\prime}}.
\nonumber
\end{align}
The contribution of all facets consisting of forward arrows only is
$\sum_{a^{\prime},b^{\prime}\geq 0}
\binom{a^{\prime}+b^{\prime}}{a^{\prime}}
\cdot
(xz)^{a^{\prime}+b^{\prime}+1}$
on the right-hand side of~\eqref{equation_revlex_facets}. 
The part of the statement regarding these facets is a direct consequence
of the binomial theorem. Similarly, the part of the statement on facets
consisting entirely of backward arrows follows from inspecting the next sum
on the right-hand side of~\eqref{equation_revlex_facets}. The
contribution of all other facets is collected in the last sum. The
contribution of all facets consisting of $k$ forward and $n-k$ backward
arrows is the sum of all terms satisfying 
$a^{\prime}+b^{\prime}+1 = k$
and
$a^{\prime\prime}+b^{\prime\prime}+1=n-k$.
The statement now follows after substituting $i=a^{\prime}+1$ and $j=a^{\prime\prime}+1$.
\end{proof}

\section{Face enumeration in the lex class}
\label{section_enumeration_nonest}

We now turn our attention to face enumeration in the lex class
consisting of the four triangulations
studied in Subsection~\ref{subsection_nonest}.
Among them are the lexicographic pulling triangulation.
So far, the lex class and revlex class
have been similar to each other;
see
Propositions~\ref{proposition_nonest}
and~\ref{proposition_nest},
Subsections~\ref{subsection_nonest}
and~\ref{subsection_nest}.
In this section this similarity breaks down.
This section differs from the previous ones in the simplicity
and uniformity of its main result.
As the attentive reader might suspect,
there is also a purely combinatorial way to prove it.
For space considerations, we will present this combinatorial proof in an upcoming
paper~\cite{Ehrenborg_Hetyei_Readdy_Catalan}.
Here we depend on the tools developed in
Section~\ref{section_tools}.

\begin{theorem}
Let $\triangle_{n}$ be a uniform flag triangulation
of the boundary $\partial P_{n}$ of
the root polytope in the lex class,
that is, $\triangle_{n}$ satisfies the rules:
\begin{enumerate}
\item Both $THTH$ and $HTHT$ types of pairs of arrows do not nest.
\item Both $THHT$ and $HTTH$ types of pairs of arrows do not cross.
\end{enumerate}
Then the number of $(k-1)$-dimensional faces
in the triangulation~$\triangle_{n}$
consisting of $i$ forward arrows and
$k-i$ backward arrows is given by
\begin{align*}
f(\triangle_{n}, i, k-i)
=
\frac{f_{k-1}(\triangle_{n})}{k+1}
& =
\frac{1}{k+1}
\cdot
\binom{n+k}{k}
\cdot
\binom{n}{k} .
\end{align*}
Furthermore, this quantity is independent of the parameter $i$.
\label{theorem_lexicographic_refined_face_enumeration}
\end{theorem}

We begin with a lemma about Catalan numbers.
For a subset $S$ of the integers, call a {\em run} a maximal
interval in $S$.
\begin{lemma}
For $k$ a nonnegative integer,
let $i$ be an integer satisfying
$0 \leq i \leq k$.  Then the $k$th Catalan number is given by the
sum of products
\begin{align}
C_{k}
& =
\sum_{\substack{S \subseteq [k] \\ |S| = i}}
\prod_{\substack{\text{$R$ run in $S$} \\ \text{or in $[k] - S$}}}
C_{|R|} .
\label{equation_new_Catalan_identity}
\end{align}
\label{lemma_Catalan_identity}
\end{lemma}
Note that when $i=0$ or $i=k$,
the lemma does not give anything new.
When $i=1$ or $i=k-1$, the lemma yields
the classical recursion for the Catalan numbers.

\begin{proof}[Proof of Lemma~\ref{lemma_Catalan_identity}]
Recall that  $C(x) = \sum_{k \geq 0} C_{k} \cdot x^{k}$
is the generating function for the Catalan numbers,
which satisfies the quadratic relation
$C(x) = 1 + x \cdot C(x)^{2}$.
Let $f(x)$ be $C(x)$ without the constant term,
that is, $f(x) = C(x) - 1$.
Multiply the right-hand side of~\eqref{equation_new_Catalan_identity}
with $x^{i}y^{k-i}$
and notice that
$x^{i}y^{k-i}
=
\prod_{\text{$R$ run in $S$}} x^{|R|}
\cdot
\prod_{\text{$R$ run in $[k] - S$}} y^{|R|}$.
Sum over all $i$ and $k$ such that $0 \leq i \leq k$. The resulting
generating function is given by
\begin{align}
\sum_{\substack{k \geq 0 \\ 0 \leq i \leq k}}
\sum_{\substack{S \subseteq [k]\\ |S| = i}}
\prod_{\substack{\text{$R$ run in $S$} \\ \text{or in $[k] - S$}}}
    C_{|R|} \cdot x^{i} y^{k-i}
&= 
1
+ \frac{f(x)}{1 - f(y)f(x)}
+ \frac{f(y)}{1 - f(x)f(y)}
+ 2 \cdot \frac{f(x)f(y)}{1 - f(x)f(y)}
\label{equation_Catalan_4_cases}
\\
& = 
\frac{(1+f(x)) \cdot (1 + f(y))}{1 - f(x)f(y)} .
\label{equation_Catalan_identity_1}
\end{align}
Note that the constant term $1$ on the right hand side of~\eqref{equation_Catalan_4_cases}
corresponds to $k=0$,
the second term to subsets $S \subseteq [k]$ with $1,k \in S$,
the third to subsets $S \subseteq [k]$ with $1,k \not\in S$,
and finally, the fourth term to subsets $S$ such that $|S \cap \{1,k\}| = 1$.
Next observe that
\begin{align}
(1 - f(x)f(y))
\cdot
(x \cdot C(x) - y \cdot C(y))
& =
(x-y)
\cdot
C(x) \cdot C(y) ,
\label{equation_Catalan_identity_2}
\end{align}
by expanding the product on the
left-hand side 
of~\eqref{equation_Catalan_identity_2}
and simplifying using
the quadratic relation 
$C(x) = 1 + x \cdot C(x)^{2}$
four times.
Using~\eqref{equation_Catalan_identity_2}
the generating function in~\eqref{equation_Catalan_identity_1}
simplifies to
\begin{align*}
\frac{x \cdot C(x) - y \cdot C(y)}{x - y}
& =
\sum_{k \geq 0} \sum_{0 \leq i \leq k} C_{k} \cdot x^{i} y^{k-i} ,
\end{align*}
which is the generating function for the left-hand side
of~\eqref{equation_new_Catalan_identity}.
\end{proof}

\begin{remark}
{\em
Lemma~\ref{lemma_Catalan_identity} is equivalent to
the following statement about lattice paths
from the origin to $(2n,0)$ taking up steps $(1,1)$
and down steps $(1,-1)$. For such a lattice path $p$,
considered as a piecewise linear function,
let $b(p)$ be $1/2$ times the sum of the lengths of the
intervals where the lattice path is below the $x$-axis.
Then for an integer $0 \leq i \leq n$, the number of
lattice paths $p$ such that $b(p) = i$ is given by the
Catalan number $C_{n}$.
}
\end{remark}

Observe that Corollary~\ref{corollary_Catalan_times_1+z}
applies to the case when we only have backward arrows
or only forward arrows.
Now we turn our attention to the case when we have
both forward and backward arrows.
\begin{proposition}
Consider digraphs such that
both $THTH$ and $HTHT$ types of pairs of arrows do not nest
and both $THHT$ and $HTTH$ types of pairs of arrows do not cross.
The sum over all forests~$F$ consisting of $i$
forward arrows, $k-i$ backward arrows and no isolated nodes,
where $k \geq 1$, is
\begin{align*}
\sum_{F} z^{\text{\#nodes of $F$}}
& = 
C_{k} \cdot z^{k+1} \cdot (z+1)^{k-1} .
\end{align*}
\label{proposition_lex_class_forward_and_backwards}
\end{proposition}
\begin{proof}
Given that we have $k$ undirected arrows,
pick a subset $S$ of them.
Let $S$ be the set of the forward
arrows, and let the complement be the backward arrows.
Hence the generating function can be expressed as
\begin{align*}
\sum_{F} z^{\text{\#nodes of $F$}}
& = 
\sum_{\substack{S \subseteq [k] \\ |S| = i}}
(1+z^{-1})^{r(S) - 1}
\cdot
\prod_{\substack{\text{$R$ run in $S$} \\ \text{or in $[k] - S$}}}
G_{|R|}(z) ,
\end{align*}
where $r(S)$ is the sum of the number of runs in the subset $S$ 
and the number of runs in the complement subset $[k] - S$
and $G_{|R|}(z)$ is the polynomial
appearing in
equation~\eqref{equation_G_1}.
The factor $1+z^{-1}$ appears since when we switch
the direction of the arrows either the node set is disjoint, yielding the factor $1$,
or they share an a vertex, yielding the factor $z^{-1}$;
see~Figure~\ref{figure_forest_with_both_forward_and_backward_arrows}.
Expanding $G_{|R|}(z)$
using Corollary~\ref{corollary_Catalan_times_1+z}
we have
\begin{align*}
\sum_{F} z^{\text{\#nodes of $F$}}
& = 
\sum_{\substack{S \subseteq [k] \\ |S| = i}}
z^{-r(S) + 1}
\cdot
(z+1)^{r(S) - 1}
\cdot
\prod_{\substack{\text{$R$ run in $S$} \\ \text{or in $[k] - S$}}}
C_{|R|} \cdot z^{|R|+1} \cdot (z+1)^{|R|-1} \\
& =
z^{k+1}
\cdot
(z+1)^{k-1}
\cdot
\sum_{\substack{S \subseteq [k] \\ |S| = i}}
\prod_{\substack{\text{$R$ run in $S$} \\ \text{or in $[k] - S$}}}
C_{|R|} , 
\end{align*}
where we used $\sum_{R} |R| = k$.
Now by 
Lemma~\ref{lemma_Catalan_identity}
the result follows.
\end{proof}

\begin{figure}
\begin{center}
\begin{tikzpicture}
\foreach \u in {1,2, ..., 12} {\fill[black] (\u,0) circle (0.05);};
\draw[->, thick,=>spaced stealth] (1,0) to[out=45, in=135] (3.95,0.05);
\draw[->, thick,=>spaced stealth] (1,0) to[out=45, in=135] (2.95,0.05);
\draw[->, thick,=>spaced stealth] (2,0) to[out=45, in=135] (3.95,0.05);
\draw[<-, thick,=>spaced stealth] (4.05,0.05) to[out=45, in=135] (6,0);
\draw[<-, thick,=>spaced stealth] (5.05,0.05) to[out=45, in=135] (6,0);
\draw[->, thick,=>spaced stealth] (7,0) to[out=45, in=135] (7.95,0.05);
\draw[->, thick,=>spaced stealth] (9,0) to[out=45, in=135] (10.95,0.05);
\draw[->, thick,=>spaced stealth] (10,0) to[out=45, in=135] (10.95,0.05);
\draw[<-, thick,=>spaced stealth] (11.05,0.05) to[out=45, in=135] (12,0);
\end{tikzpicture}
\end{center}
\caption{A forest with $6$ forward arrows and $3$ backward arrows
on $12$ nodes, counted by the term
$G_{3}(z) \cdot z^{-1} \cdot G_{2}(z) \cdot 1
\cdot G_{4}(z) \cdot z^{-1} \cdot G_{1}(z)$
in the proof of 
Proposition~\ref{proposition_lex_class_forward_and_backwards}.}
\label{figure_forest_with_both_forward_and_backward_arrows}
\end{figure}
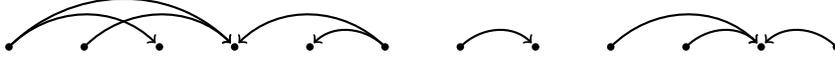

\begin{proof}[Proof of Theorem~\ref{theorem_lexicographic_refined_face_enumeration}]
Observe that the enumeration 
in
Proposition~\ref{proposition_lex_class_forward_and_backwards}
is independent of $i$,
that is, the distribution is uniform.
Inserting isolated vertices will not change this fact.
Since the total number of $k$-dimensional faces is given
by $f_{k-1}(\triangle_{n}) = \binom{n+k}{k} \cdot \binom{n}{k}$, the number
of faces with exactly $i$ forward arrows is $f_{k-1}(\triangle_{n})/(k+1)$.
\end{proof}

\section{Concluding Remarks}

The result of Oh and Yoo~\cite[Theorem~5.4]{Oh_Yoo_2} characterizing
triangulations of products of simplices is deep but not very direct. It
is a straightforward exercise, left to the reader, to show that most of the
fifteen uniform flag triangulations of the boundary $\partial P_{n}$ of
the root polytope are pulling triangulations. Since all pulling
triangulations of $\partial P_{n}$ are flag, it suffices to come up with
a pulling order that satisfies the given flag conditions. Three of the
fifteen uniform flag triangulations do not seem to arise in such an easy
manner: the triangulation in the lex class, where both $TTHH$ and
$HHTT$ types of pairs of arrows nest, the triangulation in the revlex
class where both $TTHH$ and $HHTT$ types of pairs of arrows cross, and
the triangulation (up to taking the dual) in the type~$c$ subclass
of the Simion class. The geometry of these three triangulations is
worth a closer look.

All triangulations of the boundary of the root polytope
discussed in this paper have the
same face numbers. Setting the variables $x$ and $y$ equal
in our results yields many
equalities linking the Legendre polynomials, the Catalan numbers, the
Delannoy numbers and their weighted generalizations. Exploring
these identities, relating them to known results, and proving them
combinatorially are all subjects of future investigation.

Stanley gives a condition for a lattice polytope so
that the $f$-vector of a triangulation of the polytope
would be independent of the triangulation;
see~\cite[Example~2.4 and Corollary~2.7]{Stanley_Dec}
or~\cite[Theorem~2.6]{Ehrenborg_Hetyei_Readdy}.
Is there an extension of this condition to explain
the invariance of the refined face count, as displayed
in Theorems~\ref{theorem_enumeration_in_the_simion_class},
\ref{theorem_revlex}
and~\ref{theorem_lexicographic_refined_face_enumeration}?
For instance, the condition ``if an arrow is forward or backward''
can be replaced with the more geometric condition
``if the inner product between the
vector $(1,2, \ldots, n+1)$ and the lattice point $v$ is
positive or negative''.

\section*{Acknowledgments}

The first and third authors thank the Institute for Advanced Study in
Princeton, New Jersey for supporting a research visit in Summer 2018.
This work was partially supported by grants from the
Simons Foundation
(\#429370 to Richard~Ehrenborg,
\#245153 and \#514648 to G\'abor~Hetyei,
\#422467 to Margaret~Readdy).

\newcommand{\journal}[6]{{\sc #1,} #2, {\it #3} {\bf #4} (#5), #6.}
\newcommand{\book}[4]{{\sc #1,} ``#2,'' #3, #4.}
\newcommand{\bookf}[5]{{\sc #1,} ``#2,'' #3, #4, #5.}
\newcommand{\arxiv}[3]{{\sc #1,} #2, {\tt #3}.}
\newcommand{\preprint}[3]{{\sc #1,} #2, preprint {(#3)}.}
\newcommand{\oeis}[3]{{\sc #1,} #2, {#3}.}
\newcommand{\preparation}[2]{{\sc #1,} #2, in preparation.}

\end{document}